\documentclass{article}[11pt]
\usepackage{titling}
\usepackage{amssymb}
\usepackage{amsfonts}
\usepackage[dvipsnames]{xcolor}
\usepackage{amsmath}
\usepackage{bbm}
\usepackage{setspace}
\usepackage[normalem]{ulem}
\usepackage{tikz-cd}

\usepackage{float}

\usepackage[utf8]{inputenc}

\def\showauthornotes{0}

\def\showkeys{0}
\def\showdraftbox{0}
\def\showcolorlinks{1}
\def\usemicrotype{1}
\def\showfixme{0}

\usepackage{csquotes}

\usepackage[
backend=biber,
maxcitenames=50,
maxbibnames=99,
style=alphabetic,
]{biblatex}
\usepackage{titlesec}
\title{Swap Cosystolic Expansion}
\author{Yotam Dikstein\thanks{Institute for Advanced Study, USA. email: yotam.dikstein@gmail.com.} \;and Irit Dinur\thanks{Weizmann Institute of Science, ISRAEL. email: irit.dinur@weizmann.ac.il. Both authors are supported by Irit Dinur's ERC grant 772839, and ISF grant 2073/21.}}
\date{\today}

\providecommand{\NN}{\mathbb{N}}

\usepackage[l2tabu, orthodox]{nag}

\usepackage{xspace,enumerate}

\usepackage[dvipsnames]{xcolor}

\usepackage[T1]{fontenc}
\usepackage[full]{textcomp}

\usepackage[american]{babel}

\usepackage{mathtools}

\usepackage{amsthm}
\usepackage{amsmath}

\newtheorem{theorem}{Theorem}[section]
\newtheorem*{theorem*}{Theorem}

\newtheorem{proposition}{Proposition}[subsection] 
\newtheorem*{proposition*}{Proposition}
\newtheorem{lemma}[theorem]{Lemma}
\newtheorem*{lemma*}{Lemma}
\newtheorem{corollary}[theorem]{Corollary}
\newtheorem*{corollary*}{Corollary}
\newtheorem*{conjecture*}{Conjecture}

\newtheorem*{fact*}{Fact}

\newtheorem*{hypothesis*}{Hypothesis}

\theoremstyle{definition}
\newtheorem{definition}[theorem]{Definition}
\newtheorem*{definition*}{Definition}

\theoremstyle{remark}
\newtheorem{claim}{Claim}[subsection] 
\newtheorem*{claim*}{Claim}
\newtheorem{remark}[theorem]{Remark}
\newtheorem*{remark*}{Remark}
\newtheorem{observation}[theorem]{Observation}
\newtheorem*{observation*}{Observation}

 \usepackage[letterpaper,
 top=1in,
 bottom=1in,
 left=1in,
 right=1in]{geometry}

\usepackage[varg]{pxfonts} 

\usepackage[sc,osf]{mathpazo}
\usepackage{lmodern}

\ifnum\showkeys=1
\usepackage[color]{showkeys}
\fi

\ifnum\showcolorlinks=1
\usepackage[
colorlinks=true,
urlcolor=blue,
linkcolor=blue,
citecolor=OliveGreen,
]{hyperref}
\fi

\ifnum\showcolorlinks=0
\usepackage[
colorlinks=false,
pdfborder={0 0 0}
]{hyperref}
\fi

\usepackage{prettyref}

\newcommand{\savehyperref}[2]{\texorpdfstring{\hyperref[#1]{#2}}{#2}}

\newrefformat{eq}{\savehyperref{#1}{\textup{(\ref*{#1})}}}
\newrefformat{lem}{\savehyperref{#1}{Lemma~\ref*{#1}}}
\newrefformat{def}{\savehyperref{#1}{Definition~\ref*{#1}}}
\newrefformat{thm}{\savehyperref{#1}{Theorem~\ref*{#1}}}
\newrefformat{cor}{\savehyperref{#1}{Corollary~\ref*{#1}}}
\newrefformat{cha}{\savehyperref{#1}{Chapter~\ref*{#1}}}
\newrefformat{sec}{\savehyperref{#1}{Section~\ref*{#1}}}
\newrefformat{app}{\savehyperref{#1}{Appendix~\ref*{#1}}}
\newrefformat{tab}{\savehyperref{#1}{Table~\ref*{#1}}}
\newrefformat{fig}{\savehyperref{#1}{Figure~\ref*{#1}}}
\newrefformat{hyp}{\savehyperref{#1}{Hypothesis~\ref*{#1}}}
\newrefformat{alg}{\savehyperref{#1}{Algorithm~\ref*{#1}}}
\newrefformat{rem}{\savehyperref{#1}{Remark~\ref*{#1}}}
\newrefformat{item}{\savehyperref{#1}{Item~\ref*{#1}}}
\newrefformat{step}{\savehyperref{#1}{step~\ref*{#1}}}
\newrefformat{conj}{\savehyperref{#1}{Conjecture~\ref*{#1}}}
\newrefformat{fact}{\savehyperref{#1}{Fact~\ref*{#1}}}
\newrefformat{prop}{\savehyperref{#1}{Proposition~\ref*{#1}}}
\newrefformat{prob}{\savehyperref{#1}{Problem~\ref*{#1}}}
\newrefformat{claim}{\savehyperref{#1}{Claim~\ref*{#1}}}
\newrefformat{relax}{\savehyperref{#1}{Relaxation~\ref*{#1}}}
\newrefformat{red}{\savehyperref{#1}{Reduction~\ref*{#1}}}
\newrefformat{part}{\savehyperref{#1}{Part~\ref*{#1}}}
\newrefformat{ex}{\savehyperref{#1}{Example~\ref*{#1}}}
\newrefformat{para}{\savehyperref{#1}{Paragraph~\ref*{#1}}}
\newrefformat{ass}{\savehyperref{#1}{Assumption~\ref*{#1}}}
\newrefformat{question}{\savehyperref{#1}{Question~\ref*{#1}}}
\newrefformat{obs}{\savehyperref{#1}{Observation~\ref*{#1}}}

\newcommand{\Sref}[1]{\hyperref[#1]{\S\ref*{#1}}}

\usepackage{nicefrac}

\ifnum\usemicrotype=1
\usepackage{microtype}
\fi

\ifnum\showauthornotes=1
\newcommand{\Authornote}[2]{{\sffamily\small\color{red}{[#1: #2]}}}
\newcommand{\Authornotecolored}[3]{{\sffamily\small\color{#1}{[#2: #3]}}}
\newcommand{\Authorcomment}[2]{{\sffamily\small\color{gray}{[#1: #2]}}}
\newcommand{\Authorstartcomment}[1]{\sffamily\small\color{gray}[#1: }

\newcommand{\Authorfnote}[2]{\footnote{\color{red}{#1: #2}}}
\newcommand{\Authorfixme}[1]{\Authornote{#1}{\textbf{??}}}
\newcommand{\Authormarginmark}[1]{\marginpar{\textcolor{red}{\fbox{\Large #1:!}}}}
\else
\newcommand{\Authornote}[2]{}
\newcommand{\Authornotecolored}[3]{}
\newcommand{\Authorcomment}[2]{}
\newcommand{\Authorstartcomment}[1]{}

\newcommand{\Authorfnote}[2]{}
\newcommand{\Authorfixme}[1]{}
\newcommand{\Authormarginmark}[1]{}
\fi

\ifnum\showfixme=0

\fi

\usepackage{boxedminipage}

\newcommand{\brac}[1]{[#1]}
\newcommand{\Brac}[1]{\left[#1\right]}

\newcommand{\abs}[1]{\lvert#1\rvert}
\newcommand{\Abs}[1]{\left\lvert#1\right\rvert}

\newcommand{\card}[1]{\lvert#1\rvert}

\newcommand\sett[2]{\left\{ #1 \left| \; \vphantom{#1 #2} \right. #2  \right\}}
\newcommand{\set}[1]{\{#1\}}

\newcommand{\iprod}[1]{\langle#1\rangle}

\def\dim{\mathrm{ dim}}

\newcommand{\Esymb}{\mathbb{E}}
\newcommand{\Psymb}{\mathbb{P}}

\DeclareMathOperator*{\E}{\Esymb}

\DeclareMathOperator*{\ProbOp}{\Psymb}

\renewcommand{\Pr}{\ProbOp}
\def\one{{\mathbf{1}}}
\def\wt{{\mathrm{wt}}}

\newcommand{\prob}[1]{\Pr \left[ {#1} \right] }
\newcommand{\Prob}[2][]{\Pr_{{#1}}\left[#2\right]} 
\newcommand{\cProb}[3]{\Pr_{{#1}}\left[ #2 \left| \; \vphantom{#2 #3} \right. #3  \right]} 

\newcommand{\ex}[1]{\E\brac{#1}}
\newcommand{\Ex}[2][]{\E_{{#1}}\Brac{#2}}

\newcommand{\ve}{\;\hbox{and}\;}

 \usepackage{dsfont}
\usepackage{mathrsfs}

\newcommand{\textparen}[1]{\text{(#1)}}

\ifx\because\undefined
\newcommand{\because}[1]{\textparen{because #1}}
\else
\renewcommand{\because}[1]{\textparen{because #1}}
\fi

\newcommand\bdot\bullet

\DeclareMathOperator{\poly}{poly}

\DeclareMathOperator{\dist}{dist}

\DeclareMathOperator{\sign}{sign}

\DeclareMathOperator*{\maj}{maj}

\newcommand{\Z}{\mathbb Z}

\newcommand{\C}{\mathcal{C}}
\newcommand{\A}{\mathcal{A}}

\renewcommand{\leq}{\leqslant}

\renewcommand{\geq}{\geqslant}

\ifnum\showdraftbox=1

\else

\fi

\let\epsilon=\varepsilon

\numberwithin{equation}{section}

\newcommand{\MYstore}[2]{%
  \global\expandafter \def \csname MYMEMORY #1 \endcsname{#2}%
}

\newcommand{\MYload}[1]{%
  \csname MYMEMORY #1 \endcsname%
}

\newcommand{\MYnewlabel}[1]{%
  \newcommand\MYcurrentlabel{#1}%
  \MYoldlabel{#1}%
}

\newcommand{\MYdummylabel}[1]{}

\newcommand{\torestate}[1]{%
  \let\MYoldlabel\label%
  \let\label\MYnewlabel%
  #1%
  \MYstore{\MYcurrentlabel}{#1}%
  \let\label\MYoldlabel%
}

\newcommand{\restatetheorem}[1]{%
  \let\MYoldlabel\label
  \let\label\MYdummylabel
  \begin{theorem*}[Restatement of \prettyref{#1}]
    \MYload{#1}
  \end{theorem*}
  \let\label\MYoldlabel
}

\newcommand{\restatelemma}[1]{%
  \let\MYoldlabel\label
  \let\label\MYdummylabel
  \begin{lemma*}[Restatement of \prettyref{#1}]
    \MYload{#1}
  \end{lemma*}
  \let\label\MYoldlabel
}

\newcommand{\restateprop}[1]{%
  \let\MYoldlabel\label
  \let\label\MYdummylabel
  \begin{proposition*}[Restatement of \prettyref{#1}]
    \MYload{#1}
  \end{proposition*}
  \let\label\MYoldlabel
}

\newcommand{\restateclaim}[1]{%
  \let\MYoldlabel\label
  \let\label\MYdummylabel
  \begin{claim*}[Restatement of \prettyref{#1}]
    \MYload{#1}
  \end{claim*}
  \let\label\MYoldlabel
}

\newcommand{\restatecorollary}[1]{%
  \let\MYoldlabel\label
  \let\label\MYdummylabel
  \begin{corollary*}[Restatement of \prettyref{#1}]
    \MYload{#1}
  \end{corollary*}
  \let\label\MYoldlabel
}

\newcommand{\restatefact}[1]{%
  \let\MYoldlabel\label
  \let\label\MYdummylabel
  \begin{fact*}[Restatement of \prettyref{#1}]
    \MYload{#1}
  \end{fact*}
  \let\label\MYoldlabel
}

\newcommand{\restatedefinition}[1]{
\let\MYoldlabel\label 
\let\label\MYdummylabel 
\begin{definition*}[Restatement of \prettyref{#1}] 
    \MYload{#1} 
\end{definition*} 
\let\label\MYoldlabel 
} 

\newcommand{\restate}[1]{%
  \let\MYoldlabel\label
  \let\label\MYdummylabel
  \MYload{#1}
  \let\label\MYoldlabel
}

\newcommand{\eps}{\epsilon}

\let\origparagraph\paragraph
\renewcommand{\paragraph}[1]{\origparagraph{#1.}}

\allowdisplaybreaks

\newcommand{\dunion}{\mathbin{\mathaccent\cdot\cup}}

\let\pref=\prettyref

\newcommand{\dir}[1]{\overset{\to}{#1}}
\newcommand{\bproof}[1]{\begin{proof}[Proof of \pref{#1}]}
\newcommand{\eproof}{\end{proof}}
\newcommand{\coboundary}{{\delta}}

\newcommand{\Img}{{\textrm {Im}}}
\newcommand{\agr}{{\textrm {Agree}}}

\def\S{\mathcal{S}}
\def\X{\mathcal{Z}}
\def\J{\mathcal{J}}
\newcommand{\FX}[1][r]{\mathrm{F}^{#1}\!X}

\newcommand{\FS}[1][r]{\mathrm{F}^{#1}\!\S}
\newcommand{\FD}[1][r]{\mathrm{F}^{#1}\!\Delta}
\def\d_1{d}

\newcommand{\remove}[1]{}

\begin{document}\clearpage\thispagestyle{empty}

\maketitle

\begin{abstract}
We introduce and study \emph{swap cosystolic expansion}, a new expansion property of simplicial complexes. We prove lower bounds for swap coboundary expansion of spherical buildings and use them to lower bound swap cosystolic expansion of the LSV Ramanujan complexes. Our motivation is the recent work (in a companion paper) showing that swap cosystolic expansion implies agreement theorems. Together the two works show that these complexes support agreement tests in the low acceptance regime.  

Swap cosystolic expansion is defined by considering, for a given complex $X$, its faces complex $\FX$, whose vertices are $r$-faces of $X$ and where two vertices are connected if their disjoint union is also a face in $X$. The faces complex $\FX$ is a derandomizetion of the product of \(X\) with itself \(r\) times.
The graph underlying $\FX$ is the swap walk of $X$, known to have excellent {\em spectral} expansion. The swap cosystolic expansion of $X$ is defined to be the {\em cosystolic} expansion of $\FX$.

Our main result is a $\exp(-O(\sqrt r))$ lower bound on the swap coboundary expansion of the spherical building and the swap cosystolic expansion of the LSV complexes. For more general coboundary expanders we show a weaker lower bound of $\exp(-O(r))$.
\end{abstract}

\newpage
\tableofcontents
    
\newpage\pagestyle{plain}
\setcounter{page}{1}

\section{Introduction}

Expansion of simplicial complexes, known as high dimensional expansion, has been gaining attention \cite{lubotzky2018high, gotlib2023no}. Two main notions are spectral expansion, and coboundary and cosystolic expansion. The first is related to higher order random walks \cite{KM2017color,DinurK2017,KaufmanO-RW20,DiksteinDFH2018} (from $r$-face to $r$-face) and the second is related to property testing of cohomological and topological notions \cite{KaufmanL2014,DinurM2019}. In this work we study a new notion of high dimensional expansion, which we call {\em swap coboundary (cosystolic) expansion}. This notion has to do with the so-called {\em faces} complex of a given complex, which we describe next.

Given a $d$-dimensional simplicial complex $X$, and a parameter $r<d$, its {\em Faces Complex}, denoted by $\FX$ is the following complex. The vertices of $\FX$ are the $r$-faces of $X$, and two $r$-faces $s,s'$ are connected by an edge if they are disjoint {\em and} $s\dunion s'\in X$. 
More generally, \[\set{s_0,\ldots,s_k}\in \FX(k)\quad \hbox{iff}\quad s_0\dunion\cdots\dunion s_k\in X.\]
Thus $\FX[0]=X$, but $\FX$ for $r\geq 1$ is a new complex. One may view the faces complex as a generalization of the {\em Kneser graph} \cite{kneser1955aufgabe}, which is nothing but the faces complex of the complete complex. 

The spectral expansion of the $1$-skeleton of $\FX$ has been studied previously \cite{DiksteinD2019,AlevFT2019} under the name `complement walk' or `swap walk'. A priori, it seems more natural to consider a walk from $s$ to $s'$ such that $s,s'$ intersect, but it turns out that the swap walk has much stronger spectral mixing. This already  turned out useful in applications for constraint satisfaction problems (CSPs) and for agreement tests, as we discuss below. 
The transformation from $X$ to $\FX$ is analogous to a {\em derandomized graph product}. In a graph product we move from a graph $G$ to $G^{\otimes r}$, a new graph whose vertices are $r$-tuples or $r$-sets of the old graph. Unlike the graph product case, in $\FX$ the choice of which sets of vertices to consider is not an arbitrary ``take all possible $r$-sets'', but rather specified by the complex $X$ itself. The number of $r$-sets is often much smaller, and this potentially implies greater efficiency, justifying the term `derandomized'. 

There is a significant body of work on graph products and their applications in theoretical computer science,  specifically as a method for hardness amplification. For example, parallel repetition \cite{Raz-parrep} can be viewed as a graph product that amplifies the hardness of label cover. Hardness amplification in general is an important direction in computational complexity, where one generates very hard instances from mildy hard instances, usually in a black box manner by having the new instance encode several copies of the initial instance, see \cite{IJKW08} for example.  

Derandomization, in this context, has to do with more efficient amplification, obtained by choosing a smaller collection of $r$-sets. For comparison, in non-derandomized parallel repetition, an instance of size $n$ is mapped to a new instance whose size is $n^r$. This means that in order to keep the instance polynomial size, one is restricted to $r=O(1)$. 
Derandomizing parallel repetition has been studied for three decades with only limited success. There are some general known impossibility results \cite{feige1995impossibility, moshkovitz2016no}. 
Impagliazzo et al. \cite{ImpagliazzoKW2012} have shown a successful derandomization of direct product tests, which are combinatorial analogs of parallel repetition, and this was later pushed to a derandomized parallel-repetition-like PCP theorem in \cite{DinurM11}.
The powering step in the gap amplification proof of the PCP theorem \cite{Dinur07} is a form of derandomized parallel repetition in which, as in $\FX$, the $r$-sets are chosen based on the topology of the graph itself. A major shortcoming of the powering step is its failure for values below $1/2$ \cite{Bogdanov-comment05}.  
Very recent work on derandomized direct product tests, also known as agreement tests, has highlighted the importance of the faces complex $\FX$, \cite{DiksteinD2023agr,BafnaM2023}. 
Quite mysteriously, not only the spectral expansion of this complex makes an appearance, but also coboundary and cosystolic expansion of $\FX$ turn out to be crucial. This motivates the following definition
\begin{definition}
    A simplicial complex $X$ is said to have $(\beta,r)$-{\em swap coboundary (cosystolic) expansion} if $\FX$ is a $\beta$ coboundary (cosystolic) expander for $1$-cochains.
\end{definition}
Unpacking this definition involves two aspects. First, we explain the definition of coboundaries and cocycles and relate them to unique games instances on $X$. Next, we discuss the relevant notion of expansion and its context. Only then will we be able to describe our main results and how they relate to other work.

\subsection{Coboundaries and Unique Games Instances}
Let \(X\) be a simplicial complex, let \(\dir{X}(1)\) be the set of oriented edges and fix \(Sym(\ell)\) to be the group of symmetries of \(\ell\) elements. A \(1\)-cochain puts a permutation on each edge. Namely, it is a function \(f:\dir{X}(1) \to Sym(\ell)\) such that \(f(uv)=f(vu)^{-1}\). The set of \(1\)-cochains\footnote{More generally, $i$-cochains are functions from $X(i)$ to a group of coefficients, but in this paper we focus only on $1$-cochains.} is denoted \(C^1=C^1(X,Sym(\ell))\).

Every $1$-cochain can be viewed as an instance of unique games, which is a type of constraint satisfaction problem. An instance is given by a graph $(X(0),X(1))$ such that the variables are \(X(0)\), and each edge is associated with a constraint \(\pi_{uv}\in Sym(\ell)\). We wish to find an assignment \(h:X(0) \to [\ell]\) such that \(\pi_{uv}(h(u))=h(v)\) for as many edges as possible. The maximal possible fraction is called the {\em value} of the instance. We say that an instance is \emph{satisfiable} if its value is $1$, namely if there is an assignment that satisfies all the edges.
Every $1$-cochain \(f \in C^1\) corresponds to a unique games instance on the underlying graph of \(X\), by letting \(\pi_{uv}=f(uv)\) for every edge. 
Khot famously conjectured that it is NP-hard to approximately solve unique games \cite{Khot02}, and we expand on this further in \pref{sec:UG}. 
\remove{More accurately, the so-called unique games conjecture  says that for all $\eps,\zeta>0$, given an instance of unique games whose value is at least $1-\zeta$, it is NP-hard to find a solution whose value is at least $\eps$. This is a central open problem, which, if true, would have a huge number of consequences in the area of hardness of approximation. Our work however, focuses on complexes on which unique games can be solved easily. The ability to do so is central to applications such as \cite{DinurHKLT2018,GotlibK2022,DiksteinD2023agr,BafnaM2023}.}

Just like $1$-cochains are assignments to edges, $0$-cochains are assignments to vertices, namely, functions $g:X(0)\to Sym(\ell)$. 
Every $0$-cochain $g:X(0)\to Sym(\ell)$ gives rise to a \emph{coboundary} $\coboundary_0 g:X(1) \to Sym(\ell)$ which is a $1$-cochain defined by \[ \forall uv\in \dir X (1),\qquad \coboundary_0 g(uv)=g(v)g(u)^{-1}.\] 
The set of $1$-coboundaries is defined to be \[B^1 = \sett{\coboundary_0 g }{g\in C^0} \subset C^1.\]
Per their definition, coboundaries are cochains that correspond to \emph{strongly satisfiable} unique games instances. A strongly satisfiable instance is an instance \(f:X(1) \to Sym(\ell)\) such that there exists \(\ell\) satisfying assignments \(h_1,h_2,\dots,h_\ell:X(0) \to [\ell]\) such that for any \(v \in X(0)\), \(\set{h_i(v)}_{i=1}^\ell = [\ell]\). In words, this means that given \(v \in X(0)\), one can freely assign \(v\) any \(j \in [\ell]\), and propagate this assignment to a satisfying solution to the entire instance. We prove the following lemma in \pref{sec:UG},
\begin{lemma}[See \pref{lem:cob-equiv-to-strong-sat} for a slightly stronger statement]
    Let \(X\) be a graph and let \(f\in C^1\). Then \(f \in B^1\) if and only if the unique games instance defined by \(f\) is strongly satisfiable.
\end{lemma}

As far as unique games instances go, not every satisfiable instance is also strongly satisfiable. However, some classes of unique games instances, including \emph{affine linear unique games}, are satisfiable if and only if they are strongly satisfiable. This is a rather popular class of unique games studied for example in \cite{KhotKMO2007, bafna2021playing, BafnaM2023}. Moreover, in terms of computational hardness, \cite{KhotKMO2007} showed that the hardness of approximating general unique games reduces to that of approximating affine linear unique games. 

Even if \(f\) is not strongly satisfiable (so, not a coboundary), it could be close to a coboundary. This would imply an assignment satisfying \emph{most} of the edges. Whereas the unique games conjecture asserts that finding or even approximating such an assignment is hard (even for affine linear instances), it becomes tractable when the underlying graph is a so-called coboundary expander (see \pref{claim:UGeasy}), which we define next. 

\subsection{Coboundary and cosystolic expansion}

Suppose the graph \((X(0),X(1))\) also comes with a set of triangles \(X(2)\). One can check that if $g\in C^0$ and $f=\coboundary_0 g$ then for every triangle \(uvw \in \dir{X}(2)\), the following ``triangle equation'' holds:
\begin{equation} \label{eq:cob-test-intro}
    f(vw)f(uv) =  f(uw).
\end{equation}
The reason is cancellations: clearly $g(w)g(v)^{-1}g(v)g(u)^{-1}= g(w)g(u)^{-1}$.

A coboundary satisfies all triangle equations. A coboundary expander is a complex where a robust inverse statement also holds: any \(f \in C^1\) that satisfies \eqref{eq:cob-test-intro} on most triangles is \emph{close} in Hamming distance to some strongly satisfiable \(\tilde{f} \in B^1\). That is, \(X\) is a \(\beta\)-coboundary expander if for any \(f \in C^1\) there exists \(g \in C^0\) such that 
\begin{equation}
    \beta\cdot \dist(\coboundary_0 g, f) \leq  \Prob[uvw \in X(2)]{f(vw)f(uv) =  f(uw)}.
\end{equation}
This closeness implies that there exist assignments \(\set{h_1,h_2,\dots,h_\ell}\), that satisfy almost all edges in the unique games instance corresponding to \(f\).

We note that in some complexes a cochain $f$ might satisfy all triangle equations without being a coboundary. Such cochains are called $1$-cocycles, and denoted by $Z^1$:
\[
Z^1 = \sett{ f:\dir X(1)\to Sym(\ell)}{\forall uvw\in X(2), \;f(vw)f(uv) =  f(uw)}.
\]
By the above, $B^1\subseteq Z^1 \subseteq C^1$. The set $Z^1$ of cocycles can be thought of as the set of unique games instances without local contradictions. If \(X\) is a coboundary expander, this in particular implies that for \(X\), \(B^1 = Z^1\).\\

Coboundary expansion is a generalization of graph edge expansion to higher dimensions using cohomological terms, see \cite{KaufmanL2014}. A similar but equally important notion is \emph{cosystolic} expansion. \(X\) is a \(\beta\)-cosystolic expander if for any \(f\in C^1\) that such that on most triangles \eqref{eq:cob-test-intro} holds, $f$ is close to some \(\tilde{f} \in Z^1\). If \(Z^1=B^1\) this notion becomes identical to coboundary expansion, but it has other uses even when \(Z^1 \ne B^1\).

Cosystolic expansion turns out to be important in the proof of Gromov's topological overlapping property \cite{Gromov2010,KaufmanKL2014,DotterrerKW2018}, in constructions of locally testable codes and quantum LDPC codes \cite{EvraKZ20,DinurELLM2022,PanteleevK22}, and even in proof complexity lower bounds \cite{DinurFHT2020, HopkinsL2022}. Furthermore, the main application that motivates this work is agreement testing that, while
formulated in purely combinatorial terms, ends up being inherently related to cosystolic expansion. The equivalence between coboundary (or cosystolic) expansion and local testability of the set $B^1$ of coboundaries (or $Z^1$), with respect to the triangle test has been discovered by \cite{KaufmanL2014}, who studied the case of \(\mathbb{F}_2=Sym(2)\). The connection was extended to all \(Sym(\ell)\) (and in fact, for all groups) in \cite{DinurM2019}, who showed that this related to the testability of near-covers.

\medskip
One motivation for the study of coboundary expansion is that local testability can be used as an instrument for showing existence of a perfect solution to unique games instances, assuming that we are given an instance that satisfies most of the triangle tests. This idea is important in recent works on agreement testing \cite{DiksteinD2023agr, BafnaM2023}, where a global function is constructed by moving, at a certain stage, from solution to a unique games instance which is rough but satisfies most of the triangle tests, to a perfect solution. 

\subsection{Our contribution}
The focus of this work is proving lower bounds for swap coboundary expansion. Our main result is a sub-exponential lower bound for the swap coboundary expansion of the spherical building (see \pref{thm:coboundary-expansion-intro} and \pref{thm:cosys-expansion-intro} below). First, we show that for a generic local spectral expander $X$, the coboundary expansion of $\FX$ is at least exponential in that of $X$. 
\begin{theorem} \torestate{\label{thm:generic-swap-expansion}
    Let \(X\) be a \(d\)-dimensional simplicial complex. Let \(r\) be such that \(7r+7 \leq d\). Assume that for every \(-1 \leq m \leq r\) and \(s \in X(m)\), \(h^1(X_s) \geq \beta\) and that \(X\) is a \(\lambda\)-two sided local spectral expander for \(\lambda < \frac{1}{2r^2}\), then \(X\) is a \((\beta^{O(r)},r)\)-swap coboundary expander.} 
\end{theorem}
We do not know whether this theorem is tight in general. We show in \pref{sec:generic} that the faces complex of the complete complex over \(\geq 6r\) vertices has constant coboundary expansion. We will soon move to discuss the spherical building, for which we show better bounds. It is interesting to understand the relation between $h^1(X)$ and $h^1(\FX)$ in greater detail.\\

Our main result is a lower bound for a specific family of complexes, called the spherical buildings, for which we show a sub-exponential lower bound. The \(SL_{n+1}(\mathbb{F}_q)\)-spherical building is a simplicial complex \(\S\) whose vertices are all non-trivial subspaces of \(\mathbb{F}_q^{n+1}\) and whose higher dimensional faces are all \(\set{w_0,w_1,\dots,w_m}\) that form a flag, that is, that there is an ordering so that \(w_0 \subset w_1 \subset \dots \subset w_m\). We prove, 
\begin{theorem} \label{thm:coboundary-expansion-intro}
   Let \(d,n\) be integers such that \(n>d^5\). There is some \(q_0=q_0(n)\) such that the following holds. Let \(q > q_0\) be any prime power.  Let \(\S\) be the \(SL_{n+1}(\mathbb{F}_q)\)-spherical building. Then $h^1(\FS)\geq \exp(-O(\sqrt d))$, namely \(\S\) has \((\exp(-O(\sqrt{d})),d)\)-swap coboundary expansion.
\end{theorem}
To clarify, the expression \(\exp(-O(\sqrt d))\) means that there is a constant independent \(c > 0\) of \(d\) and \(n\) such that $h^1(\FS)\geq \exp(-c\sqrt d)$.
As a direct consequence of the above, via a local-to-global theorem by \cite{EvraK2016, DiksteinD2023cbdry}, we derive a similar theorem for any complex $X$ whose vertex-links are isomorphic to spherical buildings. This includes the famous Ramanujan complexes of \cite{LubotzkySV2005a, LubotzkySV2005b} which are bounded-degree families of high dimensional expanders.
\begin{theorem} \label{thm:cosys-expansion-intro}
   Let \(d,n\) be integers such that \(n>d^5\). There is some \(q_0=q_0(n)\) such that the following holds. Let \(q > q_0\) be any prime power. Let \(X\) be a complex whose vertex links are isomorphic to the \(SL_{n}(\mathbb{F}_q)\)-spherical building. Then \(X\) has \((\exp(-O(\sqrt{d})),d)\)-swap cosystolic expansion.
\end{theorem}
There are two parts to this paper. The first part contains general tools for lower bounding coboundary expansion. These are developed for our main theorem, and may be of independent interest. The second part applies these tools towards proving \pref{thm:coboundary-expansion-intro}, and then deriving
\pref{thm:cosys-expansion-intro}.\\

The proof of \pref{thm:coboundary-expansion-intro} relies on two reductions. First, we use a color-restriction technique showing that we can decompose the faces complex of the spherical building into sub-complexes. We show that if most of these sub-complexes are coboundary expanders then so is the complex itself. This extends a similar idea from  \cite{DiksteinD2023cbdry}. These complexes are ``partite'', and only consider flags with spaces in certain, specified, dimensions. This allows us to move into an analysis of a simpler complex instead of considering the whole complex at once.

After reducing to colors, we have a second reduction. Instead of analyzing the coboundary expansion of the sub-complex directly, we lower bound the coboundary expansion in its links. Then we rely on a local-to-global theorem in \cite{DiksteinD2023cbdry} to infer coboundary expansion of the sub-complex itself. This local-to-global argument follows the argument that was first discovered in \cite{KaufmanKL2014} and \cite{EvraK2016}.

Finally, after these reductions it remains to lower bound the links of the sub-complexes. The reason we reduced to links in the first place is because the links have a similar structure to a faces complex of much lower dimension (i.e of flags of length \(\sqrt{d}\) instead of \(d\)). This allows us to use an inductive approach. This induction gives us a lower bound of \(\exp(-O(\sqrt{d}))\) (i.e.\ a lower bound of \(\exp(-c\sqrt{d})\) for some constant \(c > 0\)). This improve upon the simpler lower bound of \(\exp(-{d})\) for any large enough \(d\); the bound which we would get without considering links. This may seem mild, but this saving is what allows the use of the swap-coboundary expansion in the applications described below. However, even to prove a lower bound of \(\exp(-O(\sqrt{d}))\) on the links turns out to be technically challenging. To do so we need to apply a variety of tools.

%
\bigskip
Let us describe the tools developed in the first part of the paper. First, we generalize cones to the non-abelian setting. A main tool for lower bounding coboundary expansion is the cones technique discovered by \cite{Gromov2010}, and further developed by \cite{LubotzkyMM2016,KaufmanM2018,KozlovM2019,KaufmanO2021}. We generalize cones on \(1\)-cochains to non-abelian group coefficients. Previously, when cones were used to bound coboundary expansion, an ad-hoc proof was needed to deal with the non-abelian case. 

Let \(X\) be a simplicial complex and for now, let \(\Gamma = \mathbb{F}_2\). A cone consists of three parts: a base vertex \(v_0 \in X(0)\), a set of paths \(P_u\) from \(v_0\) to \(u\) for every \(u \in X(0)\), and a set of ``fillings'' \(T_{uw}\) for the cycle \(P_u \circ (u,w) \circ P_w^{-1}\), where \(P_w^{-1}\) is the reverse path from \(w\) to \(v_0\). Here a filling is just a chain (i.e. set of triangles) so that its \(\mathbb{F}_2\)-boundary is the cycle \(P_u \circ (u,w) \circ P_w^{-1}\). See \pref{fig:filling} for an illustration. It was observed in \cite{Gromov2010}, that if \(X\) has a transitive symmetry group and there exists a cone such that all fillings contain few triangles, then \(X\) is a coboundary expander. In fact, Gromov generalized this idea to higher-dimensional coboundary expansion (see \cite{LubotzkyMM2016} for a more formal proof, and \cite{KozlovM2019} and \cite{KaufmanO2021} for a more general setup).

\begin{figure}
    \centering
    \includegraphics[scale=0.45]{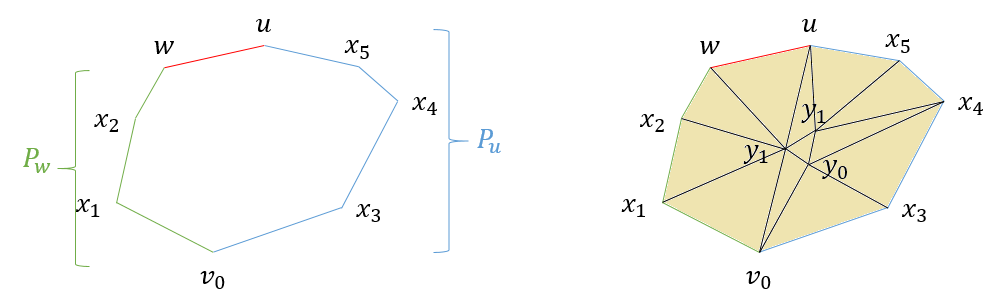}
    \caption{The cycle \(P_u \circ (u,w) \circ P_w^{-1}\) and a filling of triangles.} 
    \label{fig:filling}
\end{figure}

This technique was generalized from \(\mathbb{F}_2\) to other abelian groups in a straightforward manner (see \cite{KaufmanM2018} for the spherical building and \cite{DiksteinD2023cbdry} for the general case) but in non-abelian groups these definitions stop making sense. In this work we generalize the notion of a ``filling'' so that it will fit the non-abelian setting as well. Instead of a cochain of triangles, we require \(T_{uw}\) to be \emph{a contraction} of \(P_u \circ (u,w) \circ P_w^{-1}\) to the trivial loop around \(v_0\). This contraction is a sequence of loops \(P_1,P_2,P_3,\dots,P_m\) such that \(P_m\) is the trivial loop, and every loop differs from the previous one by replacing an edge \((u,w)\) in \(P_i\) with two edges \((u,v),(v,w)\) so that \(uvw \in X(2)\), or vice versa, i.e. replacing \((u,v,w)\) with \((u,w)\). The diameter of the cone is the maximal size of such a sequence. See \pref{fig:contraction} for an example of such a contraction and \pref{sec:cones} for the formal definitions. We prove the following lemma.
\begin{lemma} \torestate{\label{lem:group-and-cones}
    Let \(X\) be a simplicial complex such that \(Aut(X)\) is transitive on \(k\)-faces. Suppose that there exists a cone \(C\) with diameter \(R\). Then \(X\) is a \(\frac{1}{\binom{k+1}{3}\cdot R}\)-coboundary expander.}
\end{lemma}
This turns out to be the correct definition for generalizing the abelian cones theorem to the non-abelian setting. Cones constructed in previous works for concrete complexes such as in \cite{LubotzkyMM2016, DiksteinD2023cbdry} translate to the non-abelian setting in a straightforward manner.

Another tool is based on an elegant decomposition idea by Gotlib and Kaufman \cite{GotlibK2022}, which we call the GK-decomposition. In their work, they analyze a local test on some simplicial complex (which they call the representation complex).  They do so by decomposing it to small pieces,  applying a local correction argument on every piece separately, and finally ``patching up'' the corrections using some additional properties of the decomposition. We observe that their proof is actually a proof of coboundary expansion. We abstract and generalize it to a theorem applicable to other complexes as well (see \pref{thm:decomposition-to-coboundary-expanders}).

Here is the main idea. We consider a simplicial complex \(X\) that is a union of many smaller sub-complexes \(\set{Y_i}_{i \in I}\), each sub-complex is itself a coboundary expander. Given a cochain \(f \in C^1(X)\) such that \(wt(\coboundary f ) \approx 0\) our goal is to find some $g\in C^0(X)$ such that $f\approx \coboundary g$. We first find local corrections \(\set{g_i:Y_i(0) \to \Gamma}_{i \in I}\) such that \(f|_{Y_i} \approx \coboundary g_i\). A priori, we would like to construct a single  \(g:X(0) \to \Gamma\) by answering \(g(v) = g_i(v)\) for \(v \in Y_i\). The problem is that if \(v \in Y_i(0) \cap Y_j(0)\), and \(g_i(v) \ne g_j(v)\) it is not clear how to set \(g(v)\).

To model this problem, we consider a graph of intersections. The vertices of this graph are \(\set{Y_i}_{i\in I}\). The (multi-)edges of this graph are \(\set{Y_i,Y_j}_v\) such that \(v \in Y_i \cap Y_j\). It turns out that if the sub-complexes \(\set{Y_i}_{i \in I}\) are such that this graph is itself a skeleton of a coboundary expander, then we can lower bound the coboundary expansion of \(X\).
 We give a more formal overview and the precise theorem statement in \pref{sec:gk-decomposition}.

\subsection{Motivation: Agreement testing in the low acceptance regime}
In this subsection we give a brief exposition to agreement testing, and how it relates to  swap cosystolic expansion, as proven in \cite{DiksteinD2023agr}. We follow the introduction in \cite{DiksteinD2023agr}, and point the reader there for a more complete picture and further motivation.

A function $G:[n]\to\Sigma$ can be specified by a truth table, or, alternatively, by providing its restrictions $G|_{s_1}, G|_{s_2},\ldots$ for a pre-determined family of subsets ${s_1,s_2,\ldots \subset [n]}$. Such a representation has built-in redundancy which  potentially can be used for amplifying distances while providing local testability. Let $X$ be a family of $k$-element subsets of $[n]$ and let $\sett{f_s:s\to\Sigma}{s\in X}$ be an ensemble of local functions, each defined over a subset $s\subset [n]$. Is there a global function $G:[n]\to\Sigma$ such that $f_s = G|_s$ for all $s\in X$ ? An agreement test is a randomized property tester for this question. 

One such test is the V-test, that chooses a random pair of sets $s_1,s_2\in X$ with prescribed intersection size and accepts if $f_{s_1},f_{s_2}$ agree on the elements in $s_1\cap s_2$. 
We denote the success probability of the test by $\agr(\set{f_s})$.

There are two regimes of interest, depending on whether we assume that $\agr(\set{f_s}) = 1-\epsilon$ (``high acceptance'') or $\agr(\set{f_s}) = \epsilon$ (``low acceptance''). The former is known to hold for all $X$ which are spectral high dimensional expanders, see \cite{DinurK2017,DiksteinD2019}. The later is well studied in the non-derandomized setting, where $X=\binom{[n]}k$, and known as the direct product testing question.

A low acceptance agreement test theorem is a statement as follows:
   \begin{equation*} \tag{$*$}
    \agr (\set{f_s}) > \eps  \quad \Longrightarrow \quad \exists G:[n]\to\Sigma,\quad \Pr_s[f_s\overset{0.99}{\approx} G|_s]\geq\poly(\eps).
\end{equation*}
Such statements are motivated by PCP questions. A major goal is to find a family $X$ that is as sparse as possible for which $(*)$ holds. This is known as {\em derandomized} direct product testing. Unlike the high acceptance regime, the question of whether high dimensional expanders satisfy a statement such as $(*)$ has remained open, despite being more interesting for PCP applications. 

Two recent works \cite{BafnaM2023,DiksteinD2023agr} have studied low acceptance agreement tests on high dimensional expanders, and have given sufficient conditions (in \cite{BafnaM2023} the condition is also necessary) for an agreement theorem to hold. In this work we prove that the condition from \cite{DiksteinD2023agr} is satisfied for spherical buildings and for LSV complexes, thus deriving {\em unconditional agreement theorems} for these complexes in the low acceptance regime.  We describe this now in slightly more detail. The main result in \cite[Theorem 1.3]{DiksteinD2023agr} shows that whenever  $X$ is a swap-cosystolic expander, a meaningful agreement theorem follows.

First, if we assume that $X$ is a swap coboundary expander, then we can deduce $(*)$. 
By combining \cite[Theorem 1.3]{DiksteinD2023agr} with our \pref{thm:coboundary-expansion-intro} we derive the following corollary:
\begin{corollary}[Agreement for spherical buildings]\label{cor:Sph}
    There is an absolute constant \(c > 0\) such that the following holds. Let $k\in \mathbb{N}$, and let $\eps >\Omega( 1/(\log k)^c)$. Let $d>k$ be sufficiently large and let $X$ be a $d$-dimensional spherical building that is an $\exp(-d)$ high dimensional expander. For any ensemble $\set{f_s}_{s\in X(k)}$ that satisfies 
   $\agr(\set{f_s})> \eps$, there must exist a global function $G:X(0)\to\Sigma$, such that 
   \[\Pr_{s}[f_s\stackrel{0.99}\approx G|_s]\geq\poly(\eps).\]
\end{corollary}
Here we clarify that \(\poly(\frac{1}{\varepsilon})\) means \(c_1 \varepsilon^{-c_2}\) for some large enough universal constants \(c_1,c_2 > 0\), that do not depend on the other parameters of the corollary (similarly for \(\poly(\varepsilon)\).

Next, when $X$ is a swap cosystolic expander (which is weaker than swap coboundary expander), \cite[Theorem 1.4]{DiksteinD2023agr} shows that it satisfies an agreement theorem that is (necessarily) weaker than $(*)$ but still meaningful. By combining this with our \pref{thm:cosys-expansion-intro} we get:
\begin{corollary}[Agreement for LSV complexes]\label{cor:LSV}
     There is an absolute constant \(c > 0\) such that the following holds. Let $k\in \mathbb{N}$, and let $\eps >\Omega( 1/(\log k)^c)$. Let $d>k$ be sufficienty large and let $X$ be a $d$-dimensional $\lambda=\exp(-d)$ high dimensional expander, whose vertex links are spherical buildings.

For any ensemble $\set{f_s}_{s\in X(k)}$ that satisfies 
   $\agr(\set{f_s})> \eps$, there must exist a $\poly(1/\eps)$-cover $\rho:Y\twoheadrightarrow X$, and a global function $G:Y(0)\to\Sigma$, such that 
   \[\Pr_{s}[f_s\hbox{ is explained by }G]\geq\poly(\eps).\]
\end{corollary}
Here the phrase ``$f_s$ is explained by $G$'' informally means that $G|_{s'}\approx f_s$ for some $s'\in Y$ that covers $s$. More formally, the covering map $\rho$ gives a bijection from $s'\in Y(k)$ to $s\in X(k)$, and by $G|_{s'}\approx f_s$ we mean that for (almost) every $v\in s'$, $G(v) = f_s(\rho(v))$. 

We do not go into the details of covers of complexes, as these are described in depth in \cite{DiksteinD2023agr}.
We only point out that moving to a cover $Y$ of $X$ is unavoidable, as is shown in \cite{DiksteinD2023agr}.

\subsection{Related work}
Coboundary and Cosystolic expansion was defined by Linial, Meshulam and Wallach \cite{LinialM2006}, \cite{MeshulamW09}, and indpendently by Gromov \cite{Gromov2010}.

Kaufman, Kazhdan and Lubotzky \cite{KaufmanKL2014} introduced an elegant local to global argument for proving cosystolic expansion of $1$-chains in the \emph{bounded-degree} Ramanujan complexes of \cite{LubotzkySV2005a,LubotzkySV2005b}. This was significantly extended by Evra and Kaufman \cite{EvraK2016} to cosystolic expansion in all dimensions, thereby resolving Gromov's conjecture about existence of bounded degree simplicial complexes with the topological overlapping property in all dimensions. Kaufman and Mass \cite{KaufmanM2018, KaufmanM2021} generalized the work of Evra and Kaufman from \(\mathbb{F}_2\) to all other groups as well, and used this to construct lattices with good distance.

Following ideas that appeared implicitly in Gromov's work, Lubotzky Mozes and Meshulam analyzed the expansion of many ``building like'' complexes \cite{LubotzkyMM2016}. Kozlov and Meshulam \cite{KozlovM2019} abstracted the main lower bound in \cite{LubotzkyMM2016} to the definition of cones (which they call chain homotopies), in order to analyze the coboundary expansion of geometric lattices and other complexes. Their work also connects coboundary expansion to other homological notions, and gives an upper bound to the coboundary expansion of bounded degree simplicial complexes.
In \cite{KaufmanO2021}, Kaufman and Oppenheim defined the notion of cones in order to analyze the cosystolic expansion of their high dimensional expanders (see \cite{KaufmanO181}). Techniques for lower bounding coboundary expansion were further developed in \cite{DiksteinD2023cbdry}, removing dependencies on degree and dimension that appeared in some previous works.

Dinur and Meshulam observed the connection between cosystolic expansion and cover-stability. Later on, this connection was used by \cite{GotlibK2022} to analyze the problem of list-agreement on coboundary expanders. A companion paper \cite{DiksteinD2023agr} along with independent work by \cite{BafnaM2023}, analyzes agreement tests on sparse high dimensional expanders as discussed above.
\subsection{Open questions}
We give a collection of tools for lower bounding coboundary expansion. Our main application is a lower bound of \(\exp(-O(\sqrt{d}))\) on the coboundary expansion of the faces complex $\FS[d]$ of the spherical building. Our proof is involved, yet it yields a modest bound only.  Therefore, this gives rise to three questions. What is the tightest bound for the coboundary expansion of $\FS[d]$? 

We also give a general bound for the coboundary expansion of the faces complex $\FX[d]$ for a given complex $X$ (provided that it and ints links are coboundary expanders)This bound decays exponentially with \(d\). What is the correct bound in this case?

As an intermediate step towards the general case, we suggest analyzing the faces complex of KO complexes (high dimensional expanders constructed by \cite{KaufmanO181}). Kaufman and Oppenheim showed these complexes are cosystolic expanders and that their links are coboundary expanders \cite{KaufmanO2021}. Can one give a similar analysis to their faces complex? A bound that is better than inverse-exponential could lead to new sparse agreement expanders via the theorem in \cite{DiksteinD2023agr}.

In this work we also analyze partite tensor products of simplicial complexes as defined in \cite{FriedgutI2020}. We show that a partite tensor product of a \(k\)-partite \(\beta\)-coboundary expander with a complete \(k\)-partite complex is also a \(\beta \cdot \exp(-O(k))\)-coboundary expander. It is interesting to improve and generalize this result. Is it true that if two \(k\)-partite complexes \(X\) and \(Y\) are \(\beta_X\)- and \(\beta_Y\)-coboundary expanders (respectively) then their partite tensor product is a \(\beta_X \beta_Y\) coboundary expander?

\subsubsection*{Generalized Kneser graphs} We had mentioned above that the faces complex  may be viewed as a generalization of the {\em Kneser graph}. Indeed the Kneser graph $KG_{n,k}$ on ground set $n$ is precisely the faces complex of the complete complex on $n$ vertices, which we later denote by $\FD[k-1]_n$. This is a very well-studied object in combinatorics, and it would be interesing to see which of its properties continue to hold when the complete complex $\Delta_n$ is replaced by a high dimensional expander $X$. 
\section{Preliminaries} \label{sec:preliminaries}
\subsection{Probability distributions}
The following definition quantifies absolute continuity of probability measures.
\begin{definition} \label{def:smooth-pair-of-distributions}
    Let \((P,Q)\) be an (ordered) pair of probability distributions supported on a set \(\Omega\). We say that \((P,Q)\) are \((A,\alpha)\)-smooth if for every \(v \in A\) it holds that \(\alpha \Prob[P]{v} \leq \Prob[Q]{v}\). We say that \((P,Q)\) are \(\alpha\)-smooth if they are \((\Omega,\alpha)\)-smooth.    

\end{definition}

The following property is easy to verify from the definitions. We omit its proof.
\begin{claim}\label{claim:property-of-smooth-dist}
    Let \((P,Q)\) be \((A,\alpha)\) smooth. Then for every \(B \subseteq A\) it holds that \(\alpha \Prob[P]{B} \leq \Prob[Q]{B}\). \(\qed\)
\end{claim}

\subsection{Expander graphs}
Let \(G=(V,E)\) be a graph and let \(\mu:E \to (0,1]\) be a probability distribution. The distribution on the edges extends naturally to a distribution on the vertices where probability of a vertex is \(\mu(v) = \frac{1}{2} \sum_{e \ni v} \mu(e)\) (we abuse notation and denote this distribution \(\mu\) as well). Let \(A\) be the normalized adjacency operator. This operator takes as input \(f:V \to \mathbb{R}\) and outputs \(Af:V \to \mathbb{R}, \; Af(v) = \frac{1}{2\mu(v)} \sum_{u \sim v} \mu(uv) f(u)\). This operator is self adjoint with respect to the inner product on \(\ell_2(V) = \set{f:V\to \mathbb{R}}\) given by 
\[\iprod{f,g} = \sum_{v \in V}\mu(v) f(v)g(v).\]

We denote by \(\lambda(A)\) to be the (normalized) second largest eigenvalue of the adjacency operator of the graph \(X_s^{\leq 1}\). We denote by \(\abs{\lambda}(X_s)\) to be the (normalized) second largest eigenvalue of the adjacency operator of the graph \(X_s^{\leq 1}\) \emph{in absolute norm}. We say that \(G\) is a \(\lambda\)-one sided spectral expander if for every \(\lambda(A) \leq \lambda\) and say that \(G\) is a \(\lambda\)-two sided spectral expander if \(\abs{\lambda}(A) \leq \lambda\).

We say that \(G\) is an \(\eta\)-edge expander if for every subset \(S \subseteq V\), \(S \ne \emptyset, V\) it holds that
\[\Prob[uv \in E]{u \in S,v \in V \setminus S} \geq \eta \prob{S} \prob{V \setminus S}.\]

The following claim is well known so we omit its proof.
\begin{claim} \label{claim:spectral-implies-edge-expander}
    Let \(X\) be a \(\lambda\)-one sided spectral expander, then \(X\) is a \(1-\lambda\)-edge expander. \(\qed\)
\end{claim}

Let \(G=(V,E,\mu)\) be as above. Let \(H=(V,E')\) be a subgraph of \(G\). The distribution associated with \(H\) is \(\mu_H: E' \to (0,1]\), \(\mu_H(e):= \frac{\mu(e)}{\sum_{e' \in E'} \mu(e')}\). We say that \(H\) is has the same stationary distribution for \(G\) if for every vertex \(v \in V\), \(\mu_H(v)=\mu(v)\).

\begin{claim} \label{claim:convex-combination-expanders}
    Let \(G = (V,E)\) with distribution \(\mu\) and let \(H = (V,E')\) be a subgraph of \(G\) that has the same stationary distribution. Assume that \(\prob{H}:=\prob{e \in E'} \geq \alpha\) and that \(H\) is a \(\lambda\)-one sided spectral expander. Then \(G\) is a \(1-\alpha(1-\lambda)\)-spectral expander.
\end{claim}
We prove this claim in \pref{app:outstanding-coboundary-expansion-proofs}.

\subsubsection{Graph homomorphisms and expansion}
Let \(G_1=(V_1,E_1,\mu_1)\) and \(G_2 = (V_2,E_2,\mu_2)\) be two weighted graphs. A homomorphism is a function \(\rho:V_1 \to V_2\) such that for every \(e_2 \in E_2\),
\[\mu(e_2) = \sum_{e_1 \in E_1: \rho(e_1)=e_2}\mu(e_1).\]
It can be easily verified that this implies that \(\rho(e_1) \in E_2\) for any \(e_1 \in E_1\). For every edge \(e_2=\set{v,u} \in E_2\) one can define the bipartite  \(G_1^{e_2}\) whose vertices are \(L=\rho^{-1}(v), R=\rho^{-1}(u)\) and edges are \(\rho^{-1}(e_2)\). The distribution over \(G_1^{e_2}\) is \[\mu^{e_2}(e) = \frac{\mu(e)}{\sum_{e \in \rho^{-1}(e_2)}\mu(e)}.\]

The following claim is well known. See e.g. \cite{Dikstein2022} for a proof.
\begin{claim} \label{claim:expansion-from-subexpanders}
    Let \(\lambda \in [0,1)\). Let \(G_1=(V_1,E_1,\mu_1)\) and \(G_2 = (V_2,E_2,\mu_2)\) be two weighted graphs. Let \(\rho:V_1 \to V_2\) be a homomorphism. Assume that \(\abs{\lambda}(G_2) \leq \lambda\) and that for every \(e_2 \in E_2\) \(\lambda(G_1^{e_2}) \leq \lambda\). Then \(\lambda(G_1) \leq \lambda\).
\end{claim}

\subsection{Majority and expansion}
It is well known that in expander graphs, local agreement implies agreement with a majority function. Let \(G=(V,E)\) be a graph and let \(g:V \to \set{1,2,\dots,n}\) be some function. Denote by \(S_i = \sett{v \in V}{g(v)=i}\). The majority assignment \(maj(g) \in \set{1,2,\dots,n}\) is the \(i\) such that \(\prob{S_i}\) is largest (ties broken arbitrarily).

Observe that if \(\Prob[v]{g(v)=maj(g)} \approx 1\) then for most edges \(uv \in E\), it holds that \(g(v)=g(u)\) (since with high probability they are both equal to \(maj(g)\). In expander graphs a converse to this statement also holds. That is, if for most edges \(g(v) = g(u)\) then \(\Prob[v]{g(v)=maj(g)} \approx 1\). 

\begin{claim} \label{claim:expander-and-majority}
    Let \(G\) be an \(\eta\)-edge expander. Let \(S_1,S_2,\dots,S_m\) be a partition of the vertices of \(V\) as above. Assume that \(\Prob[uv \in E]{\exists i \; u \in S_i \ve v \notin S_i} \leq \varepsilon\). Then there exists \(i\) such that \(\prob{S_i} \geq 1- \frac{\varepsilon}{\eta}\).

    Stated differently, for every \(g:V \to \set{1,2,\dots,n}\), \(\Prob[v]{g(v) \ne maj(g)} \leq \frac{\Prob[uv \in E]{g(u) \ne g(v)}}{\eta}\).
\end{claim}
We prove this claim in \pref{app:outstanding-coboundary-expansion-proofs}.

\subsection{Local spectral expanders}
Most of the definitions in this subsection are standard.

A pure \(d\)-dimensional simplicial complex \(X\) is a hypergraph that consists of an arbitrary collection of sets of size \((d+1)\) together with all their subsets. The sets of size \(i+1\) in \(X\) are denoted by \(X(i)\). The vertices of \(X\) are denoted by \(X(0)\) (we identify between a vertex \(v\) and its singleton \(\set{v}\)). We will sometimes omit set brackets and write for example \(uvw\in X(2)\) instead of \(\set{u,v,w}\in X(2)\). As a convention \(X(-1) = \set{\emptyset}\).
Let \(X\) be a \(d\)-dimensional simplicial complex. Let \(k \leq d\). We denote the set of oriented \(k\)-faces in \(X\) by \(\dir{X}(k) = \sett{(v_0,v_1,...,v_k)}{\set{v_0,v_1,...,v_k} \in X(k)}\).

For \(k \leq d\) we denote by \(X^{\leq k} = \bigcup_{j=-1}^k X(j)\) the \(k\)-skeleton of \(X\). When \(k=1\) we call this complex \emph{the underlying graph of \(X\)}, since it consists of the vertices and edges in \(X\) (as well as the empty face).

A \emph{clique complex} is a simplicial complex where every clique in the underlying graph of \(X\) is also a face in \(X\).

For a simplicial complex \(X\) we denote by \(diam(X)\) the diameter of the underlying graph.

\subsubsection*{Partite Complexes} 
A \((d+1)\)-partite \(d\)-dimensional simplicial complex is a generalization of a bipartite graph. It is a complex \(X\) such that one can decompose \(X(0) = A_0 \dunion A_1 \dunion \dots \dunion A_d\) such that for every \(s \in X(d)\) and \(i \in \set{0,1,\dots,d}\) it holds that \(\abs{s \cap A_i} = 1\). The \emph{color} of a vertex \(col(v) = i\) such that \(v \in A_i\). More generally, the color of a face \(s\) is \(c = col(s) = \sett{col(v)}{v \in s}\). We denote by \(X[c]\) the set of faces of color \(c\) in \(X\), and for a singleton \(\set{i}\) we sometimes write \(X[i]\) instead of \(X[\set{i}]\).

We also denote by $X^c$, for $c\subset [d+1]$, the complex induced on vertices whose colors are in $c$.

\subsubsection*{Probability over simplicial complexes}
Let \(X\) be a simplicial complex and let \(\Pr_d:X(d)\to (0,1]\) be a density function on \(X(d)\) (that is, \(\sum_{s \in X(d)}\Pr_d(s)=1\)). This density function induces densities on lower level faces \(\Pr_k:X(k)\to (0,1]\) by \(\Pr_k(t) = \frac{1}{\binom{d+1}{k+1}}\sum_{s \in X(d),s \supset t} \Pr_d(s)\). We can also define a probability over directed faces, where we choose an ordering uniformly at random. Namely, for \(s\in \dir{X}(k)\), \(\Pr_k(s) = \frac{1}{(k+1)!}\Pr_k(set(s))\) (where \(set(s)\) is the set of vertices participating in \(s\)). When clear from the context, we omit the level of the faces, and just write \(\Pr[T]\) or \(\Prob[t \in X(k)]{T}\) for a set \(T \subseteq X(k)\).

\subsubsection*{Links and local spectral expansion} Let \(X\) be a \(d\)-dimensional simplicial complex and let \(s \in X\) be a face. The link of \(s\) is the \(d'=d-|s|\)-dimensional complex
\[X_s = \sett{t \setminus s}{t \in X, t \supseteq s}.\]
For a simplicial complex \(X\) with a measure \(\Pr_d:X(d) \to (0,1]\), the induced measure on \(\Pr_{d',X_s}:X_s(d-|s|)\to (0,1]\) is
\[\Pr_{d',X_s}(t \setminus s) = \frac{\Pr_d(t)}{\sum_{t' \supseteq s} \Pr_d(t')}.\] 

We denote by \(\lambda(X_s)\) to be the (normalized) second largest eigenvalue of the adjacency operator of the graph \(X_s^{\leq 1}\). We denote by \(\abs{\lambda}(X_s)\) to be the (normalized) second largest eigenvalue of the adjacency operator of the graph \(X_s^{\leq 1}\) \emph{in absolute norm}.

\begin{definition}[local spectral expander]
    Let \(X\) be a \(d\)-dimensional simplicial complex and let \(\lambda \in (0,1)\). We say that \(X\) is a \emph{\(\lambda\)-one sided local spectral expander} if for every \(s \in X^{\leq d-2}\) it holds that \(\lambda(X_s) \leq \lambda\). We say that \(X\) is a \emph{\(\lambda\)-two sided local spectral expander} if for every \(s \in X^{\leq d-2}\) it holds that \(\abs{\lambda}(X_s) \leq \lambda\).
\end{definition}
We stress that this definition includes \(s= \emptyset\), which also implies that the graph \(X^{\leq 1}\) should have a small second largest eigenvalue.

\subsubsection*{Walks on local spectral expanders}
Let \(X\) be a \(d\)-dimensional simplicial complex. Let \(\ell \leq k \leq d\). The \((k,\ell)\)-containment graph \(G_{k,\ell}=G_{k,\ell}(X)\) is the bipartite graph whose vertices are \(L=X(k), R=X(\ell)\) and whose edges are all \((t,s)\) such that \(t \supseteq s\). The distribution we associate to the edges is the natural distribution induced by the complex \(X\), that is, sampling \(t \sim X(k)\) and then uniformly sampling \(s \subseteq t\) of size \(|s|=\ell+1\).

\begin{theorem}[\cite{KaufmanO-RW20}] \label{thm:eignevalues-of-walk}
    Let \(X\) be a \(d\)-dimensional \(\lambda\)-one sided local spectral expander. Let \(\ell \leq k \leq d\). Then the second largest eigenvalue of \(G_{k,\ell}(X)\) is upper bounded by \(\lambda(G_{k,\ell}(X)) \leq \frac{\ell+1}{k+1} + O(k\lambda)\). Here \(\lambda(G_{k,\ell}(X))\) means the second largest eigenvalue of the normalized adjacency operator of \(G_{k,\ell}(X)\).
\end{theorem}

A related walk is the \emph{swap walk}. Let \(k,\ell,d\) be integers such that \(\ell+k\leq d-1\). The \(k,\ell\)-swap walk \(S_{k,\ell}=S_{k,\ell}(X)\) is the bipartite graph whose vertices are \(L=X(k), R=X(\ell)\) and whose edges are all \((t,s)\) such that \(t \dunion s \in X\)(here \(\dunion\) means \emph{disjoint} union). The probability of choosing such an edge is the probability of choosing \(u \in X(k+\ell+1)\) and then uniformly at random partitioning it to \(u=t\dunion s\). This walk has been defined and studied independently by \cite{DiksteinD2019} and by \cite{AlevFT2019}, who bounded its spectral expansion.

\begin{theorem}[\cite{DiksteinD2019,AlevFT2019}] \label{thm:swap-walk-exp}
    Let \(X\) be a \(\lambda\)-two sided local spectral expander. Then the second largest eigenvalue of \(S_{k,\ell}(X)\) is upper bounded by \(\lambda(S_{k,\ell}(X)) \leq (k+1)(\ell+1)\lambda\).
\end{theorem}

For a \(d\)-partite complex and two disjoint set of colors \(J_1,J_2 \subseteq [d]\) one can also define the \emph{colored swap walk} \(S_{J_1,J_2}\) as the bipartite graph whose vertices are \(LX[J_1],R=X[J_2]\). and whose edges are all \((s,t)\) such that \(t \dunion s \in X[J_1 \dunion J_2]\). The probability of choosing this edge is \(\Prob[{X[J_1 \dunion J_2]}]{t \dunion s}\).

\begin{theorem}[\cite{DiksteinD2019}]\label{thm:color-swap-walk-exp}
    Let \(X\) be a \(d\)-partite \(\lambda\)-one sided local spectral expander. Then the second largest eigenvalue of \(S_{J_1,J_2}(X)\) is upper bounded by \(\lambda(S_{J_1,J_2}(X)) \leq |J_1|\cdot |J_2|\cdot \lambda\).
\end{theorem}
We note that this theorem also make sense even when \(J_1 = \set{i}, J_2 = \set{i'}\), and the walk is between \(X[i]\) and \(X[i']\) that are subsets of the vertices.

We will also need the following claim on the (uncolored) swap walk on partite simplicial complexes. We prove this claim in \pref{app:outstanding-coboundary-expansion-proofs}.
\begin{claim} \label{claim:partite-walk-is-a-const-spectral-expander}
    Let \(X\) be a \(d\)-partite complex such that the colored swap walk is a \(\lambda\)-one sided spectral expander. Let \(G\) be the graph whose vertices are \(X(0)\) and whose edges are obtained by taking two steps in the swap walk \(S_{0,j}\) for \(j \leq d-2\). Then \(\lambda(G) \leq \frac{1+\max (\lambda,\frac{1}{d-1})}{2}\).
\end{claim}

\subsubsection*{Partitification}
Let \(X\) be an \(n\)-dimensional simplicial complex and let \(\ell \leq n\). The \(\ell\)-partitification of \(X\) is the following \(\ell\)-partite complex.
\[X^{\dagger_\ell}(0) = X(0) \times [\ell].\]
\[X^{\dagger_\ell}(\ell-1) = \sett{\set{(v_1,\pi(1)),(v_2,\pi(2)),\dots,(v_\ell,\pi(\ell))}}{\set{v_1,v_2\dots,v_\ell} \in X(\ell-1), \pi \in Sym(\ell)}.\]
We choose a top-level face by choosing \(s \in X(\ell-1)\) and a uniform at random permutation (and independent) \(\pi \in Sym(\ell)\).
As one observes, this is an \(\ell\)-partite complex where \(X[i] = X(0) \times \set{i}\). For a set \(s = \set{(v_0,i_0),(v_1,i_1),\dots,(v_j,i_j)} \in X^{\dagger_\ell}(j)\) we denote by \(p_1(s) = \set{v_0,v_1,\dots,v_j} \in X(j)\) and \(p_2(s)=\set{i_0,i_1,\dots,i_j}\).

The following claim is easy to verify.
\begin{claim} \label{claim:link-of-partitification}
    Let \(X\) be an \(n\)-dimensional simplicial complex and let \(\ell \leq n\). Let \(s \in X(j)\) for \(j \leq \ell-3\). Then \(X^{\dagger_\ell}_s \cong X_{p_1(s)} \times K_{\ell-j-1}\) as graphs where \(K_{\ell-j-1}\) is the complete graph over \(\ell-j-1\) elements.
\end{claim}
\begin{proof}[Sketch]
    Observe that the choice of edges is by choosing an edge \(\set{u,v} \in X_s(1)\) and \(i_0 \ne i_1\) in \([\ell] \setminus p_2(s)\).
\end{proof}

The following is easily derived from the eigenvalues of products with the complete graph.
\begin{corollary} \label{cor:partitification-expansion}
    If \(X\) is a \(\lambda\)-two sided local spectral expander then \(X^{\dagger_\ell}\) is a \(\lambda\)-one sided local spectral expander. 
\end{corollary}

\begin{remark}
    One can define an unordered tensor product of \(\ell\)-dimensional complexes via \(Z = X \tilde{\otimes} Y\) such that \(Z(j) = \sett{\set{(v_0,u_0),(v_1,u_1),\dots,(v_j,u_j)}}{\set{v_0,v_1,\dots,v_j} \in X(j), \set{u_0,u_1,\dots,u_j} \in Y(j)}\). With this definition the partitification is in fact the unordered tensor product of \(X\) and the complete complex over \(\ell\)-vertices. A similar claim to \pref{claim:link-of-partitification} can be proven in the general case. We will not explore this construction, and we will not refer to it later in the paper so that we won't be confused with the \emph{ordered} tensor defined in the section. The coboundary expansion and other properties of this construction are interesting but left to future work. 
\end{remark}

\subsubsection*{The spherical building}
Let \(d \in \NN\) and \(q\) be a prime power. 
\begin{definition}\label{def:sph-bldg}
    The spherical building (sometimes called the \(SL_d(\mathbb{F}_q)\)-spherical building), is the complex \(X\) whose vertices are all non-trivial linear subspaces of \(\mathbb{F}_q^d\). It's higher dimensional faces are all flags \(\sett{W_0 \subseteq W_1 \subseteq \dots \subseteq W_m}{W_0,W_1,\dots,W_m \subseteq \mathbb{F}_q^d}\). 
\end{definition}
This complex is \((d-2)\)-dimensional.

\begin{claim}[\cite{EvraK2016}, \cite{DiksteinD2019} for the color restriction] \label{claim:spherical-building-hdxness}
    Let \(X\) be a \(SL_d(\mathbb{F}_q)\)-spherical building. Then \(X\) is a \(O(\frac{1}{\sqrt{q}})\)-one sided local spectral expander. Moreover, \(X^{\leq k}\) is a \(\max \set{O(\frac{1}{\sqrt{q}}), \frac{1}{d-k}}\)-two sided local spectral expander. The same holds for \(X^J\) for all subsets \(J \subseteq [d]\).
\end{claim}

\subsection{Coboundary and Cosystolic Expansion}
In this paper we focus on coboundary and cosystolic expansion on \(1\)-cochains, with respect to non-abelian coefficients. For a more thorough introduction, we refer the reader to \cite{DiksteinD2023cbdry}.

Let \(X\) be a \(d\)-dimensional simplicial complex for \(d \geq 2\) and let \(\Gamma\) be any group. For \(i=-1,0\) let 
\(C^i(X,\Gamma) = \set{f:X(i) \to \Gamma}\). We sometimes identify \(C^{-1}(X,\Gamma) \cong \Gamma\). For \(i=1,2\) let
\[C^1(X,\Gamma) = \sett{f:\dir{X}(1) \to \Gamma}{f(u,v)=f(v,u)^{-1}}\]
and
\[C^2(X,\Gamma) = \sett{f:\dir{X}(i) \to \Gamma}{\forall \pi \in Sym(3), (v_0,v_1,v_2) \in \dir{X}(2) \; f(v_{\pi(0)},v_{\pi(1)},v_{\pi(2)}) = f(v_0,v_1,v_2)^{\sign(\pi)}}.\]
be the spaces of so-called {\em anti-symmetric} functions on edges and triangles. For \(i=-1,0,1\) we define functions \(\coboundary_i : C^i(X,\Gamma) \to C^{i+1}(X,\Gamma)\) by
\begin{enumerate}
    \item \(\coboundary_{-1}:C^{-1}(X,\Gamma)\to C^{0}(X,\Gamma)\) is \(\coboundary_{-1} h (v) = h(\emptyset)\).
    \item \(\coboundary_{0}:C^{0}(X,\Gamma)\to C^{1}(X,\Gamma)\) is \(\coboundary_{0} h (v,u) = h(v)h(u)^{-1}\).
    \item \(\coboundary_{1}:C^{1}(X,\Gamma)\to C^{2}(X,\Gamma)\) is \(\coboundary_{1} h (v,u,w) = h(v,u)h(u,w)h(w,v)\).
\end{enumerate}
Let \(Id = Id_i \in C^i(X,\Gamma)\) be the function that always outputs the identity element. It is easy to check that \(\coboundary_{i+1} \circ \coboundary_i h \equiv Id_{i+2}\) for all \(i=-1,0\) and \(h \in C^{i}(X,\Gamma)\). Thus we denote by
\[Z^i(X,\Gamma) = \ker \coboundary_{i} \subseteq C^i(X,\Gamma),\]
\[B^i(X,\Gamma) = \Img \coboundary_{i-1} \subseteq C^i(X,\Gamma),\]
and have that \(B^i(X,\Gamma) \subseteq Z^i(X,\Gamma)\). 

Henceforth, when the dimension $i$ of the cochain $f$ is clear from the context we denote $\coboundary_i f$ by $\coboundary f$.

Coboundary and cosystolic expansion is a property testing notion so for this we need a notion of distance. Let \(f,g \in C^i(X,\Gamma)\). Then
\begin{equation} \label{eq:def-of-dist}
    \dist(f,g) = \Prob[s \in \dir{X}(i)]{f(s) \ne g(s)}.
\end{equation}
We also denote the weight of the function \(\wt(f) = \dist(f,Id)\).

We are ready to define coboundary and cosystolic expansion.
\begin{definition}[Cosystolic expansion] \label{def:def-of-cosyst-exp}
    Let \(X\) be a \(d\)-dimensional simplicial complex for \(d \geq 2\). Let \(\beta >0\). We say that \(X\) is a \(\beta\)-cosystolic expander if for every group \(\Gamma\), and every \(f \in C^1(X,\Gamma)\) there exists some \(g \in Z^1(X,\Gamma)\) such that
    \begin{equation} \label{eq:def-of-cosyst-exp}
        \beta \dist(f,g) \leq \wt(\coboundary f).
    \end{equation}
    In this case we denote \(h^1(X) \geq \beta\).
\end{definition}

\begin{definition}[Coboundary expansion] \label{def:def-of-cob-exp}
    Let \(X\) be a \(d\)-dimensional simplicial complex for \(d \geq 2\). Let \(\beta >0\). We say that \(X\) is a \(\beta\)-coboundary expander if it is a \(\beta\)-cosystolic expander and in addition \(Z^1(X,\Gamma) = B^1(X,\Gamma)\) for every group \(\Gamma\).
\end{definition}

Another way of phrasing coboundary expansion is the following. If \(X\) is a \(\beta\)-coboundary expander, then it holds that for every \(f \in C^1(X,\Gamma)\) there exists a function \(h \in C^0(X,\Gamma)\) such that \begin{equation*}
        \beta \dist(f,\coboundary h) \leq \wt(\coboundary f).
\end{equation*}

Although this definition of cosystolic and coboundary expansion related to such expansion over \emph{every} group \(\Gamma\), one can also consider cosystolic expansion with respect to a specific group \(\Gamma\). All the results in this paper apply to all groups simultaneously, so we do not make this distinction.

We remark that in other works, many other coefficient groups were used instead of \(Sym(\ell)\). For our result this definition is sufficient. Note that some prior works such as \cite{EvraK2016} also require that \(\psi \in Z^1 \setminus B^1\) have large support. This separate requirement is not necessary in our work.

Dinur and Meshulam already observed that cosystolic expansion (and coboundary expansion) equivalent testability of covers, which they call \emph{cover stability} \cite{DinurM2019}.

\subsection{The faces complex}
\begin{definition} \label{def:face-complex}
    Let \(X\) be a \(d\)-dimensional simplicial complex. Let \(r \leq d\). We denote by \(\FX\) the simplicial complex whose vertices are \(\FX(0)=X(r)\) and whose faces are all \(\sett{\set{s_0,s_1,...,s_j}}{s_0\dunion s_1 \dunion \dots \dunion s_j \in X((j+1)(r+1)-1)}\). 
\end{definition}
It is easy to verify that this complex is \(\left ( \lfloor \frac{d+1}{r+1} \rfloor - 1 \right )\)-dimensional and that if \(X\) is a clique complex then so is \(\FX\).

Let \(X\) be a \(d\)-dimensional simplicial complex, and let $r<d$. The distribution on the top-level faces of \(\FX\) is given by the following. Let \(m = \left ( \lfloor \frac{d+1}{r+1} \rfloor - 1 \right )\)
\begin{enumerate}
    \item Sample a \(d\)-face \(t=\set{v_0,v_1,\dots,v_d} \in X(d)\).
    \item Sample \(s_0,s_1,\dots,s_m \subseteq t\) such that $|s_i|=r+1$, \(s_i \cap s_j = \emptyset\) and output \(\set{s_0,s_1,\dots,s_m}\).
\end{enumerate}

It is convenient to view the faces complex as a subcomplex of the following complex.
\begin{definition}[Generalized faces complex]
    Let $X$ be a simplicial complex. The generalized faces complex, denoted $FX$, has a vertex for every $w\in X$, and a face $s=\set{w_0,\ldots,w_i}\in FX$ iff $\dunion s := w_0\dunion w_1\dunion\cdots\dunion w_i\in X$.
\end{definition}
This complex is not pure so we do not define a measure over it.
One can readily verify that links of the faces complex correspond to faces complexes of links in the original complex. That is,
\begin{claim} \label{claim:link-of-a-faces-complex}
    Let $s\in FX$. Then $FX_s = F(X_{\cup s})$ where \(\cup s = \bigcup_{t \in s}t\). The same holds for \(\FX[r]_s=F^r(X_{\cup s})\). \(\qed\)
\end{claim}
We are therefore justified to look at generalized links of the form $FX_{\cup s}$,
\begin{definition}[Generalized Links]
    Let $w\in X$. We denote by $FX_w  = F(X_{ w})$. We also denote by $\FX_w = \FX \cap FX_w$. Note that this is not necessarily a proper link of $\FX$.
\end{definition}
\subsubsection{Colors of a faces complex}    

\begin{definition}[Simplicial homomorphism]
Let $X,Y$ be two simplicial complexes. 
A map $\varphi:X\to Y$ is called a simplicial homomorphism if $\varphi:X(0)\to Y(0)$ is onto and for every $s=\set{v_0,\ldots,v_i}\in X(i)$, $\varphi(s) = \set{\varphi(v_0),\ldots,\varphi(v_i)}\in Y(i)$. 
\end{definition}

\begin{claim}
Let $\varphi:X\to Y$ be a simplicial homomorphism. Then there is a natural homomorphism $\varphi:FX\to FY$ given by 
$\varphi(\set{s_0,\ldots,s_i})=\set{\varphi(s_0),\ldots,\varphi(s_i)}$.  
\end{claim}
\begin{proof}
    Suppose $s=\set{s_0,\ldots,s_i}\in FX(i)$. By definition this means that $\dunion s \in X$ so $\varphi(\dunion s) \in Y$. But $\varphi(\dunion s) = \varphi(s_0\dunion \cdots\dunion s_i)=\varphi(s_0)\dunion \cdots\dunion \varphi(s_i)$ (because for a simplicial homomorphism $\varphi:X\to Y$ whenever $a\dunion b\in X$, $\varphi(a\dunion b) = \varphi(a)
    \dunion\varphi(b)\in Y$). Thus $\set{\varphi(s_0),\ldots,\varphi(s_i)}\in Y$.
\end{proof}

Let $Y = \Delta_n$ be the complete complex on $n$ vertices. Recall the definition of a partite complex and observe that $X$ is $n$-partite if and only if there is a homomorphism $col:X\to \Delta_n$. 

We say that a complex is $n$ colorable if its underlying graph is $n$ colorable, namely one can partition the vertices into $n$ color sets such that every edge crosses between colors.
\begin{claim}
    Let $X$ be an $n$-colorable complex. Then $\FX$ is $\binom{n}{r+1}$-colorable.
\end{claim}
We denote the set of colors of $\FX$ by $\C= \FD(0)$ (supressing $n$ from the notation). This is the set of all subsets of $[n]$ of size $r+1$.

Fix a set $J\subset \Delta_n$, namely $J=\set{c_1,\ldots,c_m}$ and $c_j\subset [n]$ are pairwise disjoint. Let $\FX[J] = \sett{s\in FX}{col(s)\subseteq J}$ be the sub-complex of $FX$ whose vertex colors are in $J$, so $\FX[J](0) = \bigcup_{j=1}^m X[c_j]$. We will be particularly interested in the case where $J\in \FD$, namely, $J$ consists of pairwise disjoint subsets. In this case $\FX[J]$ is $|J|$-partite and $|J|-1$ dimensional. We abuse notation in this section allowing multiple \(c_j\)'s to be empty sets. In this case \(X[c_j]\) are copies of \(\set{\emptyset}\), and every empty set set is in all top level faces of \(\FX[J]\).

The measure induced on the top level faces of \(\FX[J]\) is the one obtained by sampling \(t \in X[\cup J]\) and partitioning it to \(t= s_1 \dunion s_2 \dunion \dots \dunion s_m\) such that \(s_i \in X[c_i]\).

Finally, throughout the paper we use the following notation. Let \(J',J \subseteq \FD[]\) We write $J'\leq J$, if \(J = \set{c_1,c_2,\dots,c_m}\) and $J' = \set{c'_1,\ldots,c'_m}$ where $c'_j\subseteq c_j$. 
\subsection[Tools from previous works]{Tools from \cite{DiksteinD2023cbdry}}
\begin{theorem}[{\cite[Theorem 1.2]{DiksteinD2023cbdry}}] \label{thm:cosystolic-expansion-from-link-coboundary-expansion}
        Let \(\beta, \lambda > 0\). Let \(X\) be a \(d\)-dimensional simplicial complex for \(d \geq 3\) and assume that \(X\) is a \(\lambda\)-one-sided local spectral expander. Let \(\Gamma\) be any group. Assume that for every vertex \(v \in X(0)\), \(X_v\) is a coboundary expander and that \(h^{1}(X_v) \geq \beta\).
    Then \[h^1(X) \geq \frac{(1-\lambda) \beta}{24} - e\lambda.\]
Here \(e \approx 2.71\) is Euler's number.
\end{theorem}

The following lemma is derived by an inductive version of the above theorem. Recall that we say a complex is `simply connected` if \(Z^1(X,\Gamma)=B^1(X,\Gamma)\).
\begin{lemma}\label{lem:trick-gen}
    Let \(X\) be a simplicial complex. Let \(\min_{s\in X(i)} h^1(X_s) \geq \beta>0\).
    Assume that \(X\) is a \(\beta \cdot\exp(-O(i+1))\)-one sided local spectral expander and that every non-empty link in \(X\) is simply connected. Then \(h^1(X) \geq \beta exp(-O(i))\).
\end{lemma}
The proof can be found in \pref{app:outstanding-coboundary-expansion-proofs}.

We also use the following reduction from coboundary expansion of the complex to coboundary expansion of many of its color restrictions.
\begin{theorem}[{\cite[Theorem 1.3]{DiksteinD2023cbdry}}] \label{thm:coboundary-expansion-from-colors}
    Let \(\ell, d\) be integers so that \(3\leq \ell \leq d\) and let \(\beta,p, \lambda \in (0,1]\). Let \(\Gamma\) be some group. Let \(X\) be a \(d\)-partite simplicial complex so that
\[\Prob[F \in \binom{[d]}{\ell}]{X^F \text{ is a \(\beta\)-coboundary expander} \ve \forall s \in X(0) \; X^F_s \text{ is a \(\lambda\)-spectral expander}} \geq p.\]
Then \(X\) is a coboundary expander with \(h^{1}(X) \geq \frac{p (1-\lambda) \beta}{6e}\). Here \(e \approx 2.71\) is Euler's number.
\end{theorem}
\begin{remark}~
\begin{enumerate}
    \item Both theorems are adapted from the general case to the special case of \(1\)-cochains. In addition, both statements use the fact that the notion of ``coboundary expansion on \(0\)-cochains`` is equivalent to spectral expansion. See \cite{DiksteinD2023cbdry} for more details.
    \item \pref{thm:coboundary-expansion-from-colors} is proven in \cite{DiksteinD2023cbdry} assuming that the spectral expansion of the graph is \(1-\beta\). This assumption is not needed in the proof; following the same steps with a separate parameter \(\lambda\) gives us a bound of \(h^{1}(X) \geq \frac{p (1-\lambda) \beta}{6e}\).
\end{enumerate}
\end{remark}

\subsection{Some simple coboundary expanders}
Finally, we will need to make use of the coboundary expansion of some simple simplicial complexes. The first type of complex is what we call \emph{a cone of a complex} (not to be confused with the non-abelian cones in \pref{sec:cones}).
\begin{definition}[Cone of a complex] \label{def:cone-of-a-complex}
    Let \(X\) be a \(k\)-partite simplicial complex. We denote by \(X^*\) to be the \((k+1)\)-partite complex where \(X^*(0) = X(0) \dunion \set{v_*}\) and \(X^*(i) = X(i) \cup \sett{ s \dunion \set{v_*}}{s\in X(i-1)}\). We identify \(X[j] = X^* [j]\) for every \(j=0,1,...,k-1\) and \(X^* [k] = \set{v_*}\).
\end{definition}

\begin{claim} \label{claim:cone-is-coboundary-expander}
    Let \(X\) be a \(k\)-partite simplicial complex for \(k\geq 2\). Then \(h^1(X^*) \geq \frac{k+1}{3(k-1)} \geq \frac{1}{3}\).
\end{claim}
This claim is proven in \pref{app:outstanding-coboundary-expansion-proofs}.

\begin{definition}[Complete partite complex]
    Let \(n_1,n_2,...,n_k > 0\) be integers. The \emph{\((n_1,n_2,...,n_k)\)-complete partite complex} \(K_{n_1,n_2,\dots,n_k}\) is a \(k\)-partite complex whose vertices in each part are \(K_{n_1,n_2,\dots,n_k}[j] = [n_j]\). The top level faces are all possible \(s = \set{v_1,v_2,\dots,v_k}\) such that  \(v_i \in K_{n_1,n_2,\dots,n_k}[i]\), for each $i\in [k]$
\end{definition}

\begin{definition}[Partite tensor]\label{def:partite-tensor}
    Let \(X,Y\) be two \(k\)-partite simplicial complexes. Their (partite) tensor product \(X \otimes Y\) is the simplicial complex whose vertices are \((X \otimes Y)(0) = \bigcup_{i=0}^{k-1} X[i] \times Y[i]\). The top level faces are all \(\set{(u_0,v_0),...,(u_{k-1},v_{k-1})}\) such that \(\set{u_0,...,u_{k-1}} \in X(k), \set{v_0,...,v_{k-1}} \in Y(k)\). The distribution over top level faces is by independently choosing \(\set{u_0,...,u_{k-1}} \in X(k), \set{v_0,...,v_{k-1}} \in Y(k)\), and then pairing them by color.
\end{definition} 
This operation on simplicial complexes was defined by \cite{FriedgutI2020} which also observed that a link of a \((k-2)\)-dimensional face \(\set{(u_0,v_0),...,(u_{k-3},v_{k-3})}\) is the bipartite tensor product \(X_{\set{u_0,u_1,...,u_{k-3}}} \otimes Y_{\set{v_0,v_1,...,v_{k-3}}}\). In particular, when \(X,Y\) are \(\lambda\)-one sided local spectral expanders, then so is \(X \otimes Y\).

\begin{claim} \label{claim:triangle-complex}
    Let \(k \geq 5\). Let \(X\) be a \(k\)-partite simplicial complex, such that \(h^1(X) \geq \beta\). Assume that the colored swap walks between vertices to triangles is an \(\eta\)-spectral expander. Then \(Y = X \otimes K_{n_1,n_2,...,n_k}\) is a coboundary expander and \(h^1(Y) \geq (1-O(\eta)) \exp(-O(\ell)) \beta\) where \(\ell = \Abs{\sett{i \in [k]}{n_i > 1}}\).
\end{claim}
This claim is proven in \pref{app:outstanding-coboundary-expansion-proofs}.

\begin{corollary} \label{cor:complex-with-one-free-side}
    Let \(X\) be a \(k\)-partite simplicial complex, for \(k \geq 5\). Assume that the colored swap walks between vertices to triangles in \(X^{[k-1]}\) an \(\eta\)-spectral expander. Assume further that for every \(s \in X[ \set{0,1,\dots,k-2} ]\) and every \(v \in X[k-1]\), \(s \dunion \set{v} \in X(k-1)\). Then \(h^1(X) =\Omega(1)\).
\end{corollary}

\begin{proof}
    Let \(\abs{X[k-1]} = n_{k-1}\). By definition we can write \(X \cong (X^{[k-1]})^* \otimes K_{1,1,...,1,n_{k-1}}\) and use \pref{claim:cone-is-coboundary-expander} and \pref{claim:triangle-complex} to obtain the corollary.
\end{proof}

We will also need the following two claims that show that (in the cases we care about) the coboundary expansion of the complex and its partitification are the same up to constant factors. These are also proven in \pref{app:outstanding-coboundary-expansion-proofs}.

\begin{claim} \label{claim:coboundary-expansion-of-complex-as-good-as-its-partitification}
    Let \(X\) be a simplicial complex Then \(h^1(X) = \Omega(h^1(X^{\dagger_\ell}))\).    
\end{claim}

\begin{claim} \label{claim:coboundary-expansion-of-partitification-as-good-as-original}
    Let \(X\) be a \(\lambda\)-two sided spectral expander of dimension at least \(5\) and let \(\ell \geq 7\). Then \(h^1(X^{\dagger_\ell})=\Omega(h^1(X))\).    
\end{claim}
\section{Unique games and coboundary expanders}\label{sec:UG}
In this section we draw out the connection between cochains and coboundaries to unique games instances and satisfiable instances. For simplicity we assume that all groups in this section are finite.

Let \(X\) be a simplicial complex and let $\Gamma$ be a group. The set  $C^1(X,\Gamma)$ is the set of cochains $f:\dir X(1)\to \Gamma$ so that $f(uv) = f(vu)^{-1}$. 

Suppose that \(\Sigma\) is a set such that \(\Gamma\) acts on $\Sigma$ (i.e., $\Gamma$ is isomorphic to a subgroup of $Sym(\Sigma)$). 
One can define a unique games instance \(U\) on \(X\) whose alphabet is \(\Sigma\). The constraints on the edges are \(\pi_{uv} = f(uv)\), namely, \(\pi_{uv}(\sigma) = f(uv) . \sigma\) via the action of \(\Gamma\) on \(\Sigma\). We recall that by Cayley's theorem, every group \(\Gamma\) acts on itself by left multiplication, so without loss of generality there is always such a set $\Sigma$. In the other direction, one can also verify that every unique games instance with alphabet \([n]\) also induces a cochain whose group coefficients are \(\Gamma = Sym(n)\).

Fix a unique game instance \(U\). An assignment is a function \(h:V \to \Sigma\). Its value with respect to \(U\) is
\[Val(U,h) = \Prob[uv \in E]{\pi_{vu}(h(u))=h(v)}.\]
The value of the instance \(U\) is
\[Val(U) = \max_{h:V \to \Sigma} Val(U,h).\]
If \(Val(U,h)=1\) we say that \(h\) satisfies \(U\). If \(U\) has a satisfying assignment we say that \(U\) is satisfiable.

It turns out that coboundaries \(f \in B^1(X,\Gamma)\) correspond to unique games instances that are satisfiable in a strong sense, which we now define.
\begin{definition}[Strongly satisfiable]\label{def:ssat}
A unique games instance \(U\) over an alphabet \(\Sigma\) is \emph{strongly satisfiable} if there exist satisfying assignments \(\overline{H} = \sett{h_\sigma}{\sigma \in \Sigma}\) so that for every vertex \(v \in X(0)\) and \(\sigma \in \Sigma\) it holds that 
\begin{equation} \label{eq:strongly-satisfiable}
    \set{h_\sigma(v)}_{\sigma \in \Sigma} = \Sigma.
\end{equation}
\end{definition}
Note that for a fixed \(v \in X(0)\), \eqref{eq:strongly-satisfiable} holds if and only if the mapping \(\sigma \mapsto h_\sigma(v)\) is a permutation.

As mentioned in the introduction, not all satisfiable instances of unique games are strongly satisfiable. However, the two are equivalent for example for the well-studied class of {\em affine linear} unique games, first studied in \cite{KhotKMO2007}.

\subsection{Coboundaries are strongly satisfiable instances}
In this subsection we show that coboundaries are equivalent to unique games that are strongly satisfiable.
Recall that an action of a group $\Gamma$ on a set $\Sigma$ is called \textit{faithful} if every pair of distinct elements $g,g'\in \Gamma$ give rise to distinct permutations on $\Sigma$. For a homomorphism \(\phi: \Gamma \to Sym(\Sigma)\) and \(f \in C^i(X,\Gamma)\) the function \(\phi(f) \in C^i(X,\phi(\Gamma))\) is the function \(\phi(f)(uv)=\phi(f(uv))\).
\begin{lemma} \label{lem:cob-equiv-to-strong-sat}
    Let \(X\) be a connected simplicial complex and let \(\Gamma\) be any finite group with an action on a set \(\Sigma\), \(\phi: \Gamma \to Sym(\Sigma)\). Let $f\in C^1(X,\Gamma)$ and let \(U\) be the unique games instance that \(f\) induces over \(\Sigma\). Then
    \begin{enumerate}
        \item If \(f \in B^1(X,\Gamma)\) then \(U\) is strongly satisfiable.
        \item  If \(U\) is strongly satisfiable then \(\phi(f) \in B^1(X,\phi(\Gamma))\), namely, $f=\coboundary g$ for some $g\in C^0(X,\phi(\Gamma))$. Moreover, if \(\phi\) is faithful then \(f \in B^1(X,\Gamma)\).
    \end{enumerate}
\end{lemma}

\begin{proof}[Proof of \pref{lem:cob-equiv-to-strong-sat}]
    Let us begin with the first item. Let \(f \in B^1(X,\Gamma)\), so there is some $g:X(0)\to \Gamma$ so that $f(uv) = g(u)g(v)^{-1}$ for all $uv\in X(1)$. Let \(v \in X(0)\) be some arbitrary vertex. %
    We define \(\overline{H}=\sett{h_\sigma}{\sigma \in \Sigma}\) to be the permutations \(h_\sigma(u)=g(u).\sigma\). First, we note that indeed for every \(u \in X(0)\) it holds that the mapping \(\sigma \mapsto h_\sigma(u)=g(u) . \sigma\) is a permutation by definition of an action on \(\Sigma\). Thus it is enough to show that for every \(\sigma\) it holds that \(h_\sigma\) is a satisfying assignment. Indeed, this is equivalent to \(f(uv). h_\sigma(v)= h_\sigma(u)\), i.e. \(f(uv) (g(v) . \sigma) = g(u) . \sigma\). By assumption, \(f(uv)=g(u)g(v)^{-1}\) so indeed
    \[f(uv) (g(v) . \sigma) = g(u)g(v)^{-1}g(v) . \sigma = g(u).\sigma.\]

    For the second item, let us first assume that \(\Gamma \leq Sym(\Sigma)\) and that $\phi$ is the identity map. 
    Fix some arbitrary \(v \in X(0)\). Let \(\overline{H} = \sett{h_\sigma}{\sigma \in \Sigma}\) be the set of satisfying assignments as promised in \pref{def:ssat}. For every \(u \in X(0)\) we define \(g(u) \in Sym(\Sigma)\) to be the permutation \(g(u) . \sigma = h_\sigma(u)\) (recall that being strongly satisfiable means that for every \(u\) the mapping \(\sigma \mapsto h_\sigma(u)\) is a permutation). Note that it may hold that \(g(u) \notin \Gamma\) but we will fix this later; for now let us just show that \(f= \coboundary g\). Indeed, for every \(\sigma \in \Sigma\), we use the fact that \(h_\sigma\) is a satisfying assignment to get that
    \(f(uv) h_\sigma(v) = h_\sigma(u)\) which by definition implies that
    \(f(uv) g(v).\sigma = g(u).\sigma\). Thus, as permutations it holds that \(f(uv) g(v)=g(u)\) or \(f(uv)=g(u)g(v)^{-1}=\coboundary g(uv)\).

Finally, to show that we can also find some \(g:X(0) \to \Gamma\) (i.e. that \(g(u) \in \Gamma\) for every \(u \in X(0)\)), we need the following claim, that allows us to shift permutations, and which we prove after this lemma.
    \begin{claim} \label{claim:coboundary-that-is-equal-id}
        Let \(X\) be a simplicial complex and let \(\Gamma\) be a group. Let \(v \in X\) and let \(f \in B^1(X,\Gamma)\). Then for every \(\gamma \in \Gamma\) there exists \(g\) so that \(f = \coboundary g\) and so that \(g(v)=\gamma\). 
    \end{claim}
    By \pref{claim:coboundary-that-is-equal-id} we can take some arbitrary \(v \in X(0)\), and assume that \(f=\coboundary g'\) for some \(g'\) such that \(g'(v)=Id\). We prove that \(g'(u) \in \Gamma\) for every \(u \in X(0)\). We do so by induction on \(d=\dist(v,u)\), the path distance between \(u\) and \(v\). The base case where \(u=v\) is clear since \(g'(v)=Id \in \Gamma\).

    Assume this is true for all vertices of distance \(d\) and let \(w\) be a vertex of distance \(d+1\). Let \(u'\) be a neighbor of \(w\) of distance \(d\) from \(v\). Then \(g'(u) \in \Gamma\) from the induction hypothesis. In addition \(f(wu) \in \Gamma\). As \(f(wu)g'(u)=g'(w)\) we conclude that \(g'(w) \in \Gamma\).

    As for the ``moreover'' statement, it follows directly from the following, slightly more general claim, which we state here and prove below.
    \begin{claim} \label{claim:coboundaries-closed-to-homomorphisms}
        Let \(f \in C^1(X,\Gamma)\). Let \(\phi: \Gamma \to \Gamma'\) be a group homomorphism. Then 
        \begin{enumerate}
            \item If \(f \in B^1(X,\Gamma)\) then \(\phi(f) \in B^1(X,\Gamma')\).
            \item If \(\phi\) is injective, and \(\phi(f) \in B^1(X,\Gamma')\) then \(f \in B^1(X,\Gamma)\).
        \end{enumerate}
    \end{claim}
\end{proof}

\begin{proof}[Proof of \pref{claim:coboundary-that-is-equal-id}]
    Let \(f=\coboundary g\) be a coboundary. For every \(\eta \in \Gamma\) denote by \(g_{\eta}:X(0) \to \Gamma\) to be \(g_{\eta}(v)=g(v)\eta\). Then it is easy to verify that \(\coboundary g=\coboundary g_{\eta}\) because $\eta$ gets canceled. By setting \(\eta = g(v)^{-1}\gamma\) we get that \(f=\coboundary g_\eta\) and that \(g_\eta(v)=\gamma\).
\end{proof}

\begin{proof}[Proof of \pref{claim:coboundaries-closed-to-homomorphisms}]
    Let \(f = \coboundary g\) be a coboundary. Then 
    \[\phi(f(uv)) = \phi(\coboundary g(u v)) = \phi(g(u)g(v)^{-1}) = \phi(g(u)) \phi(g(v))^{-1}.\]
    Thus \(\phi(g) = \coboundary \phi(g)\).

    For the second item we note that if \(\phi(f)=\coboundary g\) we can choose \(g'\) so that \(\phi(f)=\coboundary g'\) and so that for every vertex \(v \in X(0)\), \(g'(v) \in Im(\phi)\). If we do so then we have that \(\phi(f) = \coboundary g'\) and by the first item that was already proven we have that \(f=\phi^{-1}(\phi(f))=\coboundary \phi^{-1}(g')\).
    
    Indeed, assume or simplicity that \(X\) is connected (otherwise we treat every connected component separately). We take an arbitrary \(v \in X(0)\). By \pref{claim:coboundary-that-is-equal-id} there exists some \(g':X(0) \to \Gamma\) so that \(\phi(f)=\coboundary g'\) and so that \(g'(v)=Id\).
    Thus we just need to prove that \(g'(u) \in Im(\phi)\) for every \(u \in X(0)\). We do so by induction on \(d=\dist(v,u)\), the path distance between \(u\) and \(v\). The base case where \(u=v\) is clear since \(g'(v)=g(v)\gamma=Id \in Im(\phi)\) since the image is a subgroup.

    Assume this is true for all vertices of distance \(d\) and let \(u\) be a vertex of distance \(d+1\). Let \(u'\) be a neighbor of \(w\) of distance \(d\) from \(v\). Then \(g'(w) \in Im(\phi)\). In addition \(\phi(f)(uw) \in Im(\phi)\). As \(\phi(f)(uw)g'(u)=g'(w)\) we conclude that \(g'(w) \in Im(\phi)\).
\end{proof}

\subsection{Discussion}\label{sec:UG-discuss}
We include here a short discussion of the potential hardness of unique games on instances whose underlying graph is a cosystolic or coboundary expander.

First, we observe that unique games on coboundary expanders are easy, when the constraints are affine linear. The reason is that there is a simple way to check if the value of the instance is close to $1$. Simply compute its self-consistency on triangles. More formally, given an instance $f\in C^1(X)$, assuming $X$ is a coboundary expander, one can compute $\eps= wt(\coboundary_1 f)$. By assumption, there is some $g\in C^0(X)$ such that $\dist(\coboundary_0 g,f) \leq \eps/h^1(X) = O(\eps)$. This $g$ gives us an assignment which satisfies all but $\eps/h^1(X)$ of the constraints.  in the following sense.  to approximate when the constraints are affine linear. To summarize, we have shown the following easy claim,
\begin{claim}\label{claim:UGeasy}
    Let $X$ be a $2$-dimensional coboundary expander. Let $f\in C^1(X,\Gamma)$ be a unique games instance with affine linear constraints, given as a $1$-cochain. Then 
    $Val(f) = 1- O(wt(\coboundary_1 f))$. \qed
\end{claim}
One might also wonder about finding the assignment (beyond the value). This can be done by a greedy local correction algorithm, although needs to also assume that $X$ is a local spectral expander.

This claim above serves as one more example of a restricted family of unique games that is tractable.  There are numerous papers that investigate algorithms for unique games on restricted families of instances, such as spectral expanders \cite{AKKSTV08, MM10}, perturbed random graphs \cite{KMM11}, graphs with small ``threshold rank" \cite{Kol10, ABS15, BRS11, GS11}, and certified small set expanders \cite{bafna2021playing}. 
Also, on certified local spectral expanders \cite{BafnaHKL22} and certified hyperconractive graphs \cite{bafna2023solving}.

One would not necessarily expect a hard unique games instance to have even cosystolic expansion, but pieces of it that correspond to gadgets might, and in fact both the long code graph as well as the Grassman graph indeed seem so. 
\section{Non-abelian cones} \label{sec:cones}
In this section we wish to prove the following lemma.
\restatelemma{lem:group-and-cones}
Before commencing with the proof, we must define non-abelian cones. For a path \(P_0 = (v_0,v_1,\dots,v_m)\) and \(P_1 = (v_m,v_{m+1},\dots,v_n)\) we define the composition of paths to by \(P_0 \circ P_1 = v_0,v_1,\dots,v_m,v_{m+1},\dots,v_n\). This composition is only defined when the end point of the first path is equal to the starting point of the second.
Fix \(X\), a simplicial complex and some \(v_0 \in X(0)\). We define two symmetric relations on loops around \(v_0\):
\begin{enumerate}
    \item[(BT)]~ We say that \(P_0 \overset{(BT)}{\sim} P_1\) if \(P_i = Q_0 \circ (u,v,u) \circ Q_1\) and \(P_{1-i} = Q_0 \circ (u) \circ Q_1\) for \(i=0,1\) (i.e. going from \(u\) to \(v\) and then backtracking is trivial).
    \item[(TR)]~ We say that \(P_0 \overset{(TR)}{\sim} P_1\) if \(P_{i} = Q_0 \circ (u,v) \circ Q_1\) and \(P_{1-i} = Q_0 \circ (u,w,v) \circ Q_1\) for some triangle \(uvw \in X(2)\) and \(i=0,1\).
\end{enumerate}

Let \(\sim\) be the smallest equivalence relation that contains the above relations (i.e. the transitive closure of two relations)\footnote{The quotient space of the space of loops with this relation is in fact \(\pi_1(X,v_0)\), the fundamental group of the simplicial complex, when equipping \(\pi_1(X,v_0)\) with the concatenation operation (c.f. \cite{Surowski1984}). However, we will not need any additional knowledge about the fundamental group to state our theorem.}.

We denote by \(P \sim_1 P'\) if there is a sequence of loops \((P_0=P,P_1,...,P_m=P')\) and \(j \in [m-1]\) such that:
\begin{enumerate}
    \item \(P_j \overset{(TR)}{\sim} P_{j+1}\) and
    \item For every \(j' \ne j\), \(P_{j'} \overset{(BT)}{\sim} P_{j'+1}\).
\end{enumerate}
I.e. we can get from \(P\) to \(P'\) by a sequence of equivalences, where exactly one equivalence is by \((TR)\).

For every pair \(\set{P,P'}\) such that \(P \sim_1 P'\), we arbitrarily fix some sequence \(P_0=P,P_1,...,P_m=P'\) as above. After fixing the sequence, we denote by \(j_{P,P'}\) the index in the sequence such that \(P_j \overset{(TR)}{\sim} P_{j+1}\). We also denote by \(t_{P,P'}\) the triangle that gives \(P_j \overset{(TR)}{\sim} P_{j+1}\) and by \(Q_{P,P'}\) the shared prefix of both \(P_j\) and \(P_{j+1}\). That is, \(P_j = Q_{P,P'} \circ e_j \circ Q'\) and \(P_{j+1} = Q_{P,P'} \circ e'_j \circ Q'\), where \(e_j,e'_j\) are \((u,w),(u,v,w)\) for some \(uvw \in \dir{X}(2)\). \(t_{P,P'}\) is this triangle \(uvw\).

Finally, we also use the following notation. Let \(P = (u_0,u_1,...,u_m)\) be a walk in \(X\). We denote by $P^{-1}$ the walk $(u_m,\ldots,u_1,u_0)$. Let \(f \in C^1(X,\Gamma)\). We denote by \[f(P) := f(u_{m-1} u_m) \dots f(u_1 u_2) f(u_0 u_1) \in \Gamma.\] 

\subsection{Coboundary Expansion via Cones}
\begin{definition}[Cone]
    A cone is a triple \(C=(v_0,\set{P_u}_{u \in X(0)}, \set{T_{uw}}_{uw \in X(1)})\) such that
\begin{enumerate}
    \item \(v_0 \in X(0)\).
    \item For every \(v_0 \ne u \in X(0)\) \(P_{u}\) is a walk from \(v_0\) to \(u\). For \(u = v_0\), we take \(P_{v_0}\) to be the loop with no edges from \(v_0\).
    \item For every \(uw \in X(1)\), \(T_{uw}\) is a sequence of loops \((P_0,P_1,...,P_m)\) such that:
    \begin{enumerate}
        \item \(P_0 = P_u \circ (u,w) \circ P_w^{-1}\), 
        \item For every \(i=0,1,...,m-1\), \(P_i \sim_1 P_{i+1}\) and
        \item \(P_m\) is equivalent to the trivial loop by a sequence of \((BT)\) relations.
    \end{enumerate}
    We call \(T_{uw}\) a \emph{contraction}, and we denote $|T_{uw}| = m$.
\end{enumerate}
\end{definition}
See \pref{fig:contraction} for an illustration of a contraction. Note that the definition of \(T_{uw}\) depends on the direction of the edge \(uw\). We take as a convention that \(T_{wu}\) has the sequence of loops $(P_0^{-1},P_1^{-1},\ldots,P_m^{-1})$, and notice that $P_0^{-1} = (P_u \circ (u,w) \circ P_w^{-1})^{-1} = P_w \circ (w,u) \circ P_u^{-1}$. Thus for each edge it is enough to define one of \(T_{uw},T_{wu}\).

\begin{figure}
    \centering
    \includegraphics[scale=0.3]{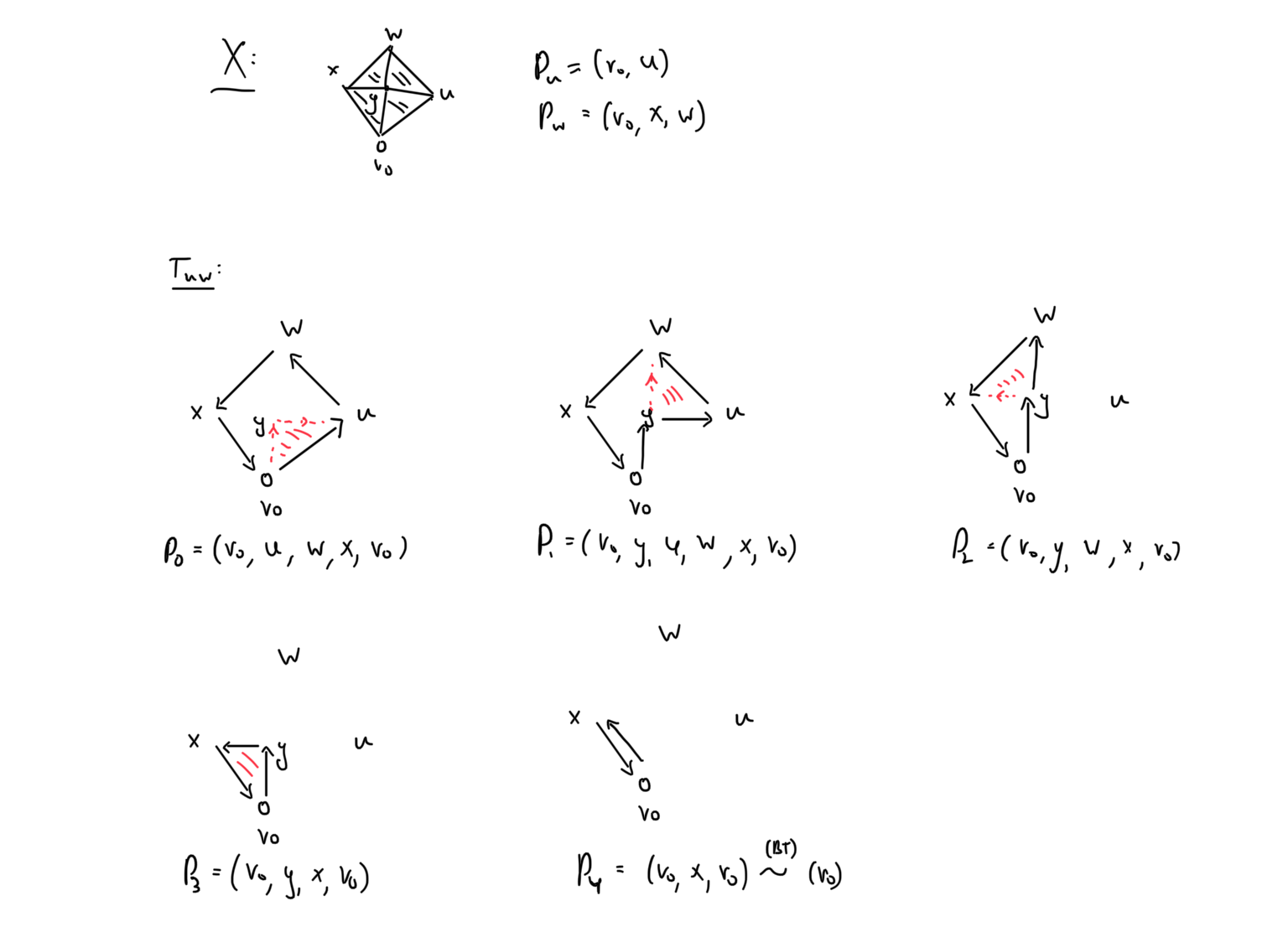}
    \caption{A contraction for an edge \(uw\)}
    \label{fig:contraction}
\end{figure}

Let \(C\) be a cone and \(f \in C^1\) we define the decoding of \(f\) by the cone \(C\) to be the function \(g_C^f:X(0)\to \Gamma\) defined by $g_C^f(v_0)=Id$ and
\[\forall u\neq v_0, \quad g_C^f(u) = f(P_u).\]
We will omit the superscript \(f\) from the notation and just write \(g_C\).

The following claim gives a sufficient condition for \(f(uw)=\coboundary g_C(uw)\) for some fixed edge \(uw \in \dir{X}(1)\). 

\begin{claim} \label{claim:cones}
Let \(C\) be a cone, let \(f \in C^1\) and let \(i \in [m]\). Let \(uw \in \dir{X}(1)\) and let \(T_{uw}=(P_0,...,P_m)\) be the contraction of \(uw\). 
Assume that for every \(i=0,1,\dots,m-1\) it holds that \(\coboundary f(t_{P_i,P_{i+1}}) = Id\). Then \(f(uw) = \coboundary g_C(uw)\).
\end{claim}

\begin{proof}[Proof of \pref{claim:cones}]
The proof of this claim follows directly from an iterated use of the following observation.
\begin{observation} \label{obs:cones-one}
Let \(P_j \sim_1 P_{j+1}\). Then if \(\coboundary f(t_{P_j,P_{j+1}}) = Id\), then \(f(P_j) = f(P_{j+1})\).
\end{observation}
If \(P_j \sim_1 P_{j+1}\) then we can find a sequence of loops \((Q_0=P_j,Q_1,\dots,Q_m=P_{j+1})\) where there is one index \(i\) such that \(Q_i \overset{(TR)} Q_{i+1}\) on \(t_{P_j,P_{j+1}}\) and the rest of the loops have \(Q_{i'} \overset{(BT)} Q_{i'+1}\). For the backtracking relations \(Q_{i'} \overset{(BT)}{\sim} Q_{i'+1}\), by anti-symmetry of \(f\) it follows that \(f(Q_{i'}) = f(Q_{i'+1})\). For \(Q_i \overset{(TR)} Q_{i+1}\), suppose that the two paths differ by \((u,v)\) and \((u,w,v)\) where \(\set{u,v,w} = t_{P_{j},P_{j+1}}\). If \(\coboundary f(t_{P_j,P_{j+1}}) = Id\) then \(f(uv)=f(wv)f(uw)\), thus concluding that \(f(Q_i) = f(Q_{i+1})\). We omit the formal proof of this observation.

The proof of this claim follows from this observation, used inductively to show that \(f(P_0)=f(P_1)=f(P_2)=\dots=f(P_m)\). As \(f\) satisfies $f(uv)=f(vu)^{-1}$ and \(P_m\) is equivalent to the trivial loop via backtracking relations, it holds that \(f(P_m) = Id\). We get that \(f(P_0)=Id\), that is, \(f(P_0)=f(P_w)^{-1} \circ f(u,w) \circ f(P_u) = Id\). Moving this around we get that
\[f(u,w) = g_C(w)g_C(u)^{-1} = \coboundary g_C(u,w).\]
\end{proof}

In light of \pref{claim:cones}, and looking ahead to a case where most triangles $uvw$ satisfy $\coboundary f(uvw)=Id$ but not all, it seems as though the fewer triangles we have in the contraction, the better. In light of this we define 
\[diam(C) = \max_{uw \in \dir{X}(1)} \abs{T_{uw}}.\]

\subsection{Coboundary Expansion from Cones}

We are now ready to prove \pref{lem:group-and-cones}. In fact, we prove this theorem in more generality. Let \(\mathcal{C} = \set{C_i}_{i\in I}\) be a family of cones. The cone distribution \(D_{\mathcal{C}}\) is the following distribution over \(\set{u,v,w} \in X(2)\):
\begin{enumerate}
    \item Sample an edge \(uw \in \dir{X}(1)\).
    \item Sample a cone \(C_i\) uniformly at random.
    \item Sample some \(P_j \in T_{uw}\).
    \item Output the triangle \(t_{P_j,P_{j+1}}\).
\end{enumerate}
We denote by \(D_{\mathcal{C}}(uvw)\) the probability of sampling \(uvw\) according to this distribution.

\begin{lemma} \label{lem:coboundary-expansion-of-complex-with-cones}
Let \(p \in (0,1)\) and let \(R \in \NN\). Let \(X\) be a simplicial complex, and denote by $\mu_X$ the distribution over triangles of $X$. Suppose $X$ has a family of cones \(\mathcal{C} = \set{C_i}_{i\in I}\) such that:
\begin{enumerate}
    \item \((D_{\mathcal{C}}, \mu_X)\) are \(p\)-smooth.
    \item \(\max_{i \in I} diam(C_i) \leq R\).
\end{enumerate}  
Then \(h^1_r(X) \geq \frac{p}{R}\).
\end{lemma}

For example, suppose \(C\) is a cone, and \(Aut(X)\) acts transitively on \(k\)-faces of $X$. This means that for any fixed triangle $T_0\in X(2)$ and $k$-face containing it $F_0\supset T_0$, when we choose a uniformly random element $\phi\in Aut(X)$ then $\phi(F_0)$ is a uniformly random $k$-face. $\phi(T_0)$ is not necessarily distributed uniformly in $X(2)$, but it captures at least $1/\binom{k+1}{3}$ fraction of the probability space of triangles. 
Thus if \(\mathcal{C}=\set{\sigma . C}_{\sigma \in Aut(X)}\), then \(D_\mathcal{C}\) is \(p=\frac{1}{\binom{k+1}{3}}\)-smooth, and so \pref{lem:coboundary-expansion-of-complex-with-cones} immediately implies \pref{lem:group-and-cones}.

\begin{proof}[Proof of \pref{lem:coboundary-expansion-of-complex-with-cones}]
Fix \(f \in C^1\). We need to find some \(g \in C^0\) such that \(\dist (f, \coboundary g) \leq Rp^{-1} wt(\coboundary f)\). 

It is enough to show that \(\Ex[i \in I]{\dist(f,\coboundary g_{C_i})} \leq Rp^{-1} wt(f)\), since this in particular proves that there is some \(g_{C_i}\) that acheives the expectation.
Indeed
\[\Ex[i \in I]{\dist(f,\coboundary g_{C_i})} = \Ex[i \in I]{\Prob[uv \in \dir{X}(1)]{f(uv) \ne \coboundary g_{C_i}(uv) }} \]
and by \pref{claim:cones} this is upper bounded by
\[\Ex[i \in I]{\Prob[uv \in \dir{X}(1)]{\exists P_\ell \in T_{uw} \; \coboundary f(t_{P_{\ell},P_{\ell+1}}) \ne Id}}.\]
There are at most \(R\) paths in every cone, hence if there exists such a pair, the probability that we uniformly sample such a pair is at least \(\frac{1}{R}\). Hence this expression is upper bounded by
\begin{equation}   \label{eq:distr-d}
R\cdot\Ex[i \in I]{\Prob[uv \in \dir{X}(1), P_\ell \in T_{uw}]{\coboundary f(t_{P_{\ell},P_{\ell+1}}) \ne Id}} =
R\cdot \Prob[t \sim D_{\mathcal{C}}]{\coboundary f(t) \ne Id}.
\end{equation}
The probability of every event according to the distribution \(D_{\mathcal{C}}\) is at most \(p^{-1}\) its probability according to the distribution of the triangles of \(X\). Hence, \eqref{eq:distr-d} is at most
\[Rp^{-1} \cdot\Prob[uvw \in \dir{X}(2)]{\coboundary f(uvw) \ne Id} = Rp^{-1} wt(\coboundary f).\]
The lemma follows.
\end{proof}
\section{The GK decomposition} \label{sec:gk-decomposition}
We describe here a very interesting technique due to recent work of Gotlib and Kaufman \cite[Appendix A]{GotlibK2022}. This technique did not appear as a theorem, rather it was used to show cover-testability of a certain complex (referred there as the representation complex). We observe that their argument can be viewed equivalently as coboundary expansion of another complex, related to the one they wish to analyze. Their technique can be generalized to show lower bounds of many other situations.

Let us describe the essence of this this technique informally. Let \(X\) be a two dimensional simplicial complex that we wish to show is a coboundary expander. Suppose that we can decompose \(X\) into many sub-complexes that are coboundary expanders. That is, let \(Y_1,Y_2,\dots, Y_m \subseteq X\) be sub complexes such that \(X = Y_1 \cup Y_2 \cup \dots \cup Y_m\), and assume that every \(Y_i\) is a coboundary expander. What can we ask from this decomposition, so that it will imply that \(X\) itself will be an coboundary expander?

For concreteness let us fix our group of coefficients to be \(\mathbb{F}_2\). Let \(f:X(1)\to \mathbb{F}_2\) be such that \(\wt(\coboundary f) = \varepsilon\). As a first attempt, we can look at the restrictions \(\set{f|_{Y_i}}_{i=1}^m\) and using the coboundary expansion of the \(Y_i\)'s separately we find \(g_1,g_2,\dots,g_m\), where \(g_i:Y_i(0)\to \set{0,1}\) is such that \[\beta \dist(f|_{Y_i}, \coboundary g_i) \leq \wt(\coboundary f|_{Y_i}).\] This is not enough though, since our goal is to find a single function \(g:X(0)\to \set{0,1}\) such that \[ \beta' \dist(f,g) \leq \wt(\coboundary f),\]
for some $\beta'$ that depends on $\beta$. The problem is that possibly some of the \(Y_i\)'s intersect, and it could be the case that \(v \in Y_i \cap Y_j\) is such that \(g_i(v) \ne g_j(v)\). To model this problem we consider the \emph{agreement graph}.

At this point let us assume for simplicity that the intersection between every two \(Y_i\)'s has at most one element. The agreement graph \(\A\) is the graph whose vertices are the \(Y_i\)'s and \(Y_i \sim Y_j\) if and only if \(Y_i \cap Y_j \ne \emptyset\). 
We also include two dimensional faces in $Y$, by putting $\set{u,v,w}\in Y(2)$ whenever the three edges $uv,vw,uw$ are in the graph. 

To continue modeling the problem, let us define the \emph{agreement function} \(h:\A(1) \to \set{0,1}\), such that \(h(Y_i,Y_j)=g_i(v)-g_j(v)\) where \(v \in Y_i \cap Y_j\). 
Note that $h$ depends on $v$, and will later really be a function of the labeled edge $\set{Y_i,Y_j}_v$.

\[\begin{tikzcd}
	& v &&& i && j \\
	\\
	u && w &&& k
	\arrow[""{name=0, anchor=center, inner sep=0}, "j"', no head, from=3-3, to=1-2]
	\arrow[""{name=1, anchor=center, inner sep=0}, "i"', no head, from=1-2, to=3-1]
	\arrow[""{name=2, anchor=center, inner sep=0}, "k"', no head, from=3-1, to=3-3]
	\arrow[""{name=3, anchor=center, inner sep=0}, "v"{description}, dotted, no head, from=1-5, to=1-7]
	\arrow[""{name=4, anchor=center, inner sep=0}, "w"{description}, dotted, no head, from=1-7, to=3-6]
	\arrow[""{name=5, anchor=center, inner sep=0}, "u"{description}, dotted, no head, from=3-6, to=1-5]
	\arrow[dotted, no head, from=0, to=2]
	\arrow[dotted, no head, from=0, to=1]
	\arrow[dotted, no head, from=2, to=1]
	\arrow[shorten <=5pt, shorten >=5pt, no head, from=3, to=5]
	\arrow[shorten <=5pt, shorten >=5pt, no head, from=4, to=3]
	\arrow[shorten <=6pt, shorten >=6pt, no head, from=4, to=5]
\end{tikzcd}\]
\begin{figure}[ht]
\centering
    \includegraphics[scale=0.6]{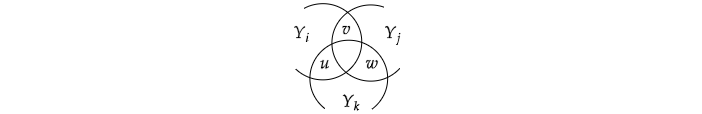}
    \caption{A triangle in $X$, in $\A$, and in the GK decomposition}
    \label{fig:triangle}
\end{figure}

A priori, we would like \(h = 0\) on every edge of $\A$, or at least \(h \approx 0\). However, Gotlib and Kaufman noticed that it is enough to require that \(h\) is close to a coboundary, i.e. \(h \approx \coboundary \ell\) for some \(\ell:\A(0)\to \set{0,1}\). The reason is that if we define \(\tilde{g}_i:Y_i(0)\to \set{0,1}\) by \(\tilde{g}_i(v)=g_i(v)+\ell(Y_i)\) then since \(\ell(Y_i)\) is constant (the input is the vertex \(v\), not \(Y_i\)), it holds that \(\coboundary \tilde{g}_i = \coboundary g_i\). Moreover, it holds that \(\tilde{g}_i(v) = \tilde{g}_j(v)\) if and only if \(g_i(v)+g_j(v)=\ell(i)+\ell(j)\), if and only if  \(h(Y_i,Y_j) = \coboundary \ell(Y_i,Y_j)\). In other words,
\begin{equation}\label{eq:agree-with-coboundary-in-agreement-graph}
    \tilde{g}_i(v) \ne \tilde{g}_j(v) \; \Leftrightarrow \; h(Y_i,Y_j) \ne \coboundary \ell(Y_i, Y_j).
\end{equation}

Hence let \(\ell:\A(0) \to \set{0,1}\) be such that \(\dist(h,\ell) = \varepsilon'\) is as small as possible. We can use these \(\set{\tilde{g}_i}\) to define a single function via majority \(g(v) = \maj_{v \in Y_i}\set{\tilde{g}_i(v)}\).

Let us see how we can bound  \(\dist(f,\coboundary g)\) in terms of \(\wt(f)\) and \(\varepsilon'\).
\begin{align} \label{eq:dist-f-maj-cob}
    \dist(f,\coboundary g) &= \Ex[i]{\dist(f|_{Y_i},\coboundary \tilde{g}_i) + \dist(\coboundary \tilde{g}_i,\coboundary \tilde g)}\\
    & \leq \frac{1}{\beta} \wt(\coboundary f) + \Ex[i]{\dist(\coboundary \tilde{g}_i,\coboundary g)},
\end{align}
where the inequality is due to \(\beta\)-coboundary expansion of every \(Y_i\) (and taking expectation). On the other hand, by expansion arguments we can argue that
\[\Ex[i]{\dist(\coboundary \tilde{g}_i,\coboundary g)} \leq 2\Ex[i]{\dist(\tilde{g}_i,g)} = 2\Ex[v]{\Prob[i]{g_i(v) \ne g(v)}}.\]
where the inequality is just by a union bound (since if \(\tilde{g}_i,g\) agree on both end points of an edge they also agree on the edge). The probability of \(\tilde{g}_i\) agreeing with \(g\) is bounded by \(\Prob[i,j]{\tilde{g}_i(v) \ne \tilde{g}_j(v)} = \Prob[i,j]{h(Y_i,Y_j) \ne \ell(Y_i,Y_j)}\) by \eqref{eq:agree-with-coboundary-in-agreement-graph}. Hence
\begin{equation} \label{eq:dist-f-maj-final}
    \eqref{eq:dist-f-maj-cob} \leq \frac{1}{\beta} \wt(\coboundary f) +2 \dist(h,\coboundary \ell) = \frac{1}{\beta} \wt(\coboundary f) + 2\varepsilon'.
\end{equation}

It looks like we've made some progress, but we are still missing a crucial component. How do we bound \(\varepsilon'\)? So far we haven't required anything yet from the structure of the agreement graph. To continue bounding \(\varepsilon'\) we require that:
\begin{enumerate}
    \item This agreement graph is (a skeleton of) a \(2\)-dimensional agreement complex \(\A\).
    \item This complex is a \(\beta'\)-coboundary expander.
    \item Moreover, we can sample a triangle \(\set{Y_i,Y_j,Y_k} \in \A(2)\) by sampling \(uvw \in X(2)\) and then sampling a triangle \(\set{Y_i,Y_j,Y_k}\) such that \(u \in Y_i \cap Y_j\), \(v \in Y_j \cap Y_k\) and \(w \in Y_k \cap Y_i\).
\end{enumerate}

If this is the case then it holds that
\begin{equation} \label{eq:bound-on-dist-h-cob-agreement-complex}
    \varepsilon' = \dist(h,\ell) \leq \frac{3+\beta}{\beta  \beta'} \wt(f).
\end{equation}

The reason is as follows.  For a triangle \(\set{Y_i,Y_j,Y_k}\) we denote by \(\set{u_{ij},u_{jk},u_{ki}} \in X(2)\) the triangle such that \(u_{ij} \in Y_i \cap Y_j\), \(u_{jk} \in Y_j \cap Y_k\) and \(u_{ki} \in Y_k \cap Y_i\). By coboundary expansion of \(\A\), it holds that \(\dist(h,\ell) \leq \frac{1}{\beta'}\wt( \coboundary h)\), so we will actually try to bound \(\wt(\coboundary h)\).

For a triangle \(\set{Y_i,Y_j,Y_k}\) \(\coboundary h = 0\) if \(\coboundary f(u_{ij},u_{jk},u_{ki}) = 0\) and for all three edges \(e\) of the triangle \(\set{u_{ij},u_{jk},u_{ki}}\) it holds that \(f(e)=\coboundary g_m(e)\) (for \(m=i,j,k\) is the \(Y_m\) that contains the edge). The reason is that
\begin{align*}
    \coboundary h(Y_i,Y_j,Y_K) &= h(Y_i,Y_j) + h(Y_j,Y_k) + h(Y_k,Y_i) \\
    & = (g_i(u_{ij}) + g_j(u_{ij})) + (g_j(u_{jk}) + g_k(u_{jk})) + (g_k(u_{ki}) + g_i(u_{ki})) \\
    &= (g_i(u_{ij})+g_i(u_{ik})) + (g_j(u_{ij})+g_j(u_{jk})) + (g_k(u_{jk})+g_k(u_{ik})) \\
    &= \coboundary g_i(u_{ij},u_{ik}) + \coboundary g_j(u_{ij},u_{jk}) + \coboundary g_k(u_{ik},u_{jk}) \\
    &\overset{f=\coboundary g_m}{=} f(u_{ij},u_{ik}) + f(u_{ij},u_{jk}) +  f(u_{ik},u_{jk}) \\
    &= \coboundary f(u_{ij},u_{jk},u_{ki}) \\
    &\overset{\coboundary f=0}{=} 0.
\end{align*}

Thus \(\wt(h) \leq 3\Ex[i]{\dist(f|_{Y_i},\coboundary{g}_i)} + \wt(f) \leq \left (\frac{3}{\beta} + 1 \right )\wt(f)\) and \eqref{eq:bound-on-dist-h-cob-agreement-complex} follows. Plugging this back in \eqref{eq:dist-f-maj-final} we get that
\[\dist(f,\coboundary g) \leq \Omega(1/\beta \beta') \wt(f).\]

\subsection{Technical aspects of the theorem}
Let us go into some more technical weeds of the theorem.

The most significant deviation from the overview (and from Gotlib and Kaufman's scope) is that we no longer require \(\Abs{Y_i \cap Y_j} \leq 1\), but instead study agreement graphs and complexes with multi edges. To our knowledge coboundary and cosystolic expansion of complexes with multi edges were not studied, but the definitions extend naturally to this case as well. In \pref{sec:blow-up} we show a useful reduction that lower bounds coboundary expansion of multi edged complexes, using the coboundary expansion of their single edged counterpart, provided the multi edges have some expanding structure.

Another significant difference is that in the technical overview we assumed that the triangle distribution \(\mu_2\) of the simplicial complex \(X\) in defined by choosing \(Y_i\) and then choosing a triangle. Moreover, it was assumed that this same of triangles in \(X\) is used to sample a triangle in the agreement complex \(\A\). I.e. that sampling a triangle \(\set{Y_i,Y_j,Y_k} \in \A(2)\) by sampling \(uvw \in X(2)\) and then sampling a triangle \(\set{Y_i,Y_j,Y_k}\) such that \(u \in Y_i \cap Y_j\), \(v \in Y_j \cap Y_k\) and \(w \in Y_k \cap Y_i\). 

It turns out that this is this assumption in too rigid to be of use in practice. For example, in \cite{GotlibK2022}, the triangles were partitioned to two sets \(X(2) = T_1 \dunion T_2\) of relative size \(\frac{1}{2}\). The triangles in one set \(T_1\) were used for the first step, where we fixed \(f\) locally to \(g_i\) for every \(Y_i\). The other set \(T_2\), was used to test agreement in the agreement complex. This set of triangles was referred to as ``empty triangles'' in \cite{GotlibK2022}.

Hence we will compare three distributions: the actual distribution \(\mu_2\) on the triangles of \(X\), the distribution \(\nu\) used to locally correct \(f\) to \(g_i\) in every \(Y_i\), and the distribution \(\pi\) used to sample triangles in the agreement complex. Our actual requirement will be that marginals of these distributions will be smooth with respect to one another as in \pref{def:smooth-pair-of-distributions}.

\subsection{Definitions for this theorem}

Recall that for a simplicial complex with multi-edges \(\set{u,v}_i\), we denote a triangle \(\set{u,v,w}_{i,j,k}\) to be the triangle whose vertices are \(u,v,w\) and whose edges are \(\set{u,v}_i, \set{v,w}_j, \set{w,u}_k\).

\begin{definition}[agreement complex] \label{def:agreement-complex}
Let \(X\) be a \(2\)-dimensional simplicial complex and let \(\mathcal{Y} = \set{Y_i}_{i \in I}\) be a family of sub-complexes of \(X\). A complex \(X\) is a agreement complex with respect to \(X,\mathcal{Y}\) if,
\begin{enumerate} 
    \item \(\A(0) = \mathcal{Y}\).
    \item \(\A(1) \subseteq \sett{\set{Y_i,Y_j}_{v}}{v \in Y_i(0) \cap Y_j(0)}\).
    \item \(\A(2) \subseteq \sett{\set{Y_i,Y_j,Y_k}_{v_1,v_2,v_3}}{\set{v_1,v_2,v_3} \in X(2)}\).
\end{enumerate}
\end{definition}

\begin{definition}[GK-decomposition]
    A GK-decomposition of \(X\) is a tuple \((\mathcal{Y},\A,\nu,\pi)\) such that 
\begin{enumerate}
    \item \(\mathcal{Y} = \set{Y_i}_{i \in I}\), are such that \(Y_i \subseteq X\) are pure sub complexes.
    \item \(\A\) is an agreement complex with respect to \(\mathcal{Y}\) whose distribution over triangles is \(\pi\).
    \item \(\nu\) is a distribution over tuples \((\set{u,v,w},Y_i)\) such that \(\set{u,v,w} \in Y_i\).
\end{enumerate}
\end{definition}

For a distribution \(\nu\) and \(i=0,1,2\) we define \(\nu_i\) to be the marginal that just outputs the \(i\)-face chosen by \(\nu\). E.g. for \(i=1\) we just take a random edge inside the triangle in the tuple chosen. We also denote by \(\nu_{i,y}\) the marginal that outputs the \(i\)-face chosen by \(\nu\) along with the sub complex \(Y_i\). Furthermore, we denote by \(\nu_y\) the marginal distribution of \(\nu\) over \(\mathcal{Y}\). Finally, we denote by \(\nu|_{Y_i}\) to be \(\nu\) conditioned on \(Y_i\) being the chosen subcomplex.

For a distribution \(\pi\) supported over \(\A(2)\) we also need similar notation. Let \(i=0,1,2\). The distribution \(\pi_i\) is the marginal distribution that outputs a \(i\)-face that is a sub face of the labels of the chosen triangle. E.g. for \(i=1\), if \(\set{Y_i,Y_j,Y_k}_{u,v,w} \sim \pi\) then \(\pi_1\) outputs one of three edges \(uv,vw,uw \in X(1)\). For \(i=1\) we denote by \(\pi_{1,y}\) the distribution that outputs \((uv,Y_j)\). That is, if \(\set{Y_i,Y_j,Y_k}_{u,v,w} \sim \pi\), then \(\pi_{1,y}\) outputs a random edge in \(\set{u,v,w}\), along with the \(Y_j\) such that \(\set{Y_i,Y_j}_u, \set{Y_j,Y_k}_v\) participate in the triangle. Finally, \(\pi_{0,y}\) is the marginal of \(\pi_{1,y}\) that samples one vertex of the sampled edge uniformly at random. I.e. if \((uv,Y_j) \sim \pi_{1,y}\) then \(\pi_{0,y}\) samples either \((u,Y_j)\) or \((v,Y_j)\).

\begin{definition}[Local graph] \label{def:local-graph}
    Let \(\A\) be an agreement complex as above. For every \(v \in \A(0)\) the local graph \(\A^v\) is a graph whose vertices are all \(Y_i \ni v\). The edges are all \(Y_i \sim Y_j\) such that \(Y_i \cap Y_j \ni v\), where we choose an edge according to the distribution of the agreement complex, given that it was labeled by \(v\).
\end{definition}

\begin{theorem} \label{thm:decomposition-to-coboundary-expanders}
Let \(X\) be a \(2\)-dimensional simplicial complex and let \((\mathcal{Y},\A,\nu,\pi)\) be a GK-decomposition. Let \(\alpha, \beta, \gamma, \eta > 0\). Assume that the following holds.
\begin{enumerate}
    \item Every \(Y_i\) is a coboundary expander with \(h^1(Y_i)\geq \beta\), with respect to the distribution \(\nu|_{Y_i}\).
    \item \(h^1(\A) \geq \gamma\).
    \item Let \(A \subseteq X(0)\) be the set of vertices that are contained in at least two \(Y_i\)'s. For every \(v \in A\), the local graph \(\A^v\) is a \(\eta\)-edge expander.
    \item Recall that \(\mu_i\) is the distribution of \(i\)-faces in \(X\). Then the following relations between distributions hold:
    \begin{enumerate}
        \item \((\nu_2,\mu_2)\) and \((\pi_2,\mu_2)\) are \(\alpha\)-smooth\footnote{as in \pref{def:smooth-pair-of-distributions}.}.
        \item \((\mu_1, \nu_1)\) are \(\alpha\)-smooth.
        \item \((\nu_{0,y},\pi_{0,y})\) are \((A_Y,\alpha)\)-smooth. Here \(A_Y\) are all \((v,Y_i)\) such that \(v \in Y_i\) and \(v \in A\).        
        \item \((\pi_{1,y},\nu_{1,y})\) are \(\alpha\)-smooth.
    \end{enumerate}
\begin{figure}
    \centering
    \begin{tikzcd}
	\pi~ && ~\nu~ &&& ~\mu \\
	{X(2)\hbox{ and }\A(2)} && {X\circ(Y_1,\ldots,Y_m)} &&& X
	\arrow[leftrightarrow,from=1-1, to=1-3]
	\arrow[leftrightarrow, from=1-3, to=1-6]
\end{tikzcd}
    \caption{$\pi$ is the joint distribution over $X\times \A$; $\nu$ is the distribution describing the decomposition of $X$ into $Y_1,\ldots,Y_m$; and $\mu$ is the distribution over $X$.}\label{fig:GK}
\end{figure}
\end{enumerate}
Then \(X\) is a coboundary expander and \(h^1(X)\geq \frac{\alpha^4 \beta \gamma \eta}{10}\). 
\end{theorem}
The relations between $\pi,\nu,\mu$ are schematically shown in \pref{fig:GK}. The conditions on the distributions quantify the quality of the decomposition. Note that in case $\alpha=1$ the distributions $\mu$ and $\nu$ are obtained as marginals of the distribution $\pi$. In this case $\mu_i=\nu_i$. The relation between $\nu$ and $\pi$ says that the global distribution on $\A$ is compatible with the localized pieces (described by $\nu$). 

We also encourage the readers to go over examples of where this theorem is used to see examples of decompositions that satisfy these items (such as \pref{prop:base-reduction-using-decomposition-subspaces}, and also \pref{claim:triangle-complex}, \pref{lem:final-decomposition} and \pref{lem:general-case-subspace-complex}), so that one will observe that these are easy to check in ``practical'' use cases.

\begin{proof}[Proof of \pref{thm:decomposition-to-coboundary-expanders}]
    Fix any group \(\Gamma\). Let \(f \in C^1(X,\Gamma)\) be such that \(\wt(f)=\varepsilon\), and we need to find some \(g:X(0)\to \Gamma\) such that \(\dist(f,\coboundary g) \leq \frac{10}{\alpha^4 \beta\gamma\eta} \varepsilon\). In particular by \pref{claim:property-of-smooth-dist}, and by the fact that \((\nu_2,\mu_2)\) are \(\alpha\)-smooth,
    \[\Prob[uvw \sim \nu_2]{\coboundary f(uvw) \ne 0} \leq \alpha^{-1} \Prob[uvw \sim \mu_2]{\coboundary f(uvw) \ne 0} = \alpha^{-1}\varepsilon.\]
    Let \(\coboundary g_i\in B^1(Y_i,\Gamma)\) be such that \(\dist(f|_{Y_i},\coboundary g_i)\) is minimal. By coboundary expansion of every \(Y_i\),
    \[\dist(f|_{Y_i},\coboundary g_i) \leq \beta^{-1} \Prob[uvw \sim \nu|_{Y_i}]{\coboundary f(uvw) \ne 0}.\] So in particular
\begin{align} \label{eq:local-distances}
\Ex[Y_i \sim \nu_y]{\dist(f|_{Y_i},\coboundary g_i)} \leq \beta^{-1} \Ex[Y_i \sim \nu_y]{\Prob[uvw \sim \nu_{Y_i}]{\coboundary f(uvw) \ne 0}} =  \beta^{-1}\Prob[uvw \sim \nu_2]{\coboundary f(uvw) \ne 0} \leq \alpha^{-1}\beta^{-1} \varepsilon.
\end{align}

We now turn to defining the correction function \(\coboundary g\) of \(f\). Denote by \(h:\dir{\A}(1)\to \Gamma\) the agreement function i.e. \[h((Y_i,Y_j)_v)=g_i(v)^{-1} g_j(v).\] 
Let \(\coboundary \ell:\dir{\A}(1)\to \Gamma\) be a coboundary closest to \(h\) (chosen arbitrarily). Let \(\tilde{g}_i\) be the functions \(\tilde{g}_i(v)=g_i(v)\ell(Y_i)\). It is easy to check that 
\begin{equation}\label{eq:ggtilde}    
\coboundary g_i = \coboundary \tilde{g}_i
\end{equation}
More interestingly,
\begin{equation}\label{eq:g_and_H}    
\tilde{g}_i(v) = \tilde{g}_j(v) \quad \hbox{if and only if}\quad h((Y_i,Y_j)_v) = \coboundary \ell((Y_i,Y_j)_v).
\end{equation}
The reason is that \(\tilde{g}_i(v)=\tilde{g}_j(v)\)
    can be expanded to
    \[g_i(v)\ell(Y_i) = g_j(v)\ell(Y_j)\]
    which, by multiplying by \(\ell(Y_j)^{-1}\) on the right and \(g_i(v)^{-1}\) on the left, is equivalent to
    \[\ell(Y_i,Y_j) = \ell(Y_i)\ell(Y_j)^{-1}=g_i(v)^{-1}g_j(v) = h((Y_i,Y_j)_v).\]

We finally define \(g:X(0) \to \Gamma\) by \(g(v) = \maj_{Y_i \ni v} \set{\tilde{g}_i(v)}\) where the most popular assignment is with respect to \(\pi_{0,y}\), conditioned on \(v\) being the chosen vertex.
Let us analyze the distance of $\coboundary g$ to $f$. 
\begin{align*}
    \dist (f, \coboundary g)  &\leq \alpha^{-1} \Prob[uv \sim \nu|_{X(1)}]{f(uv) \ne \coboundary g(uv)} = \alpha^{-1} \Ex[Y_i \sim \nu_Y]{\dist_{Y_i}(f|_{Y_i}, \coboundary g)}\\   
    &\leq \alpha^{-1} \left ( \Ex[Y_i \sim \nu_Y]{\dist_{Y_i}(f|_{Y_i},\coboundary \tilde{g}_i)} +\Ex[Y_i \sim \nu_Y]{\dist_{Y_i}(\coboundary \tilde{g}_i,\coboundary g)} \right )\\
    &\leq \alpha^{-1} \left (\Ex[Y_i \sim \nu_Y]{\dist_{Y_i}(f|_{Y_i},\coboundary \tilde{g}_i)} + 2 \Ex[Y_i \sim \nu_Y]{\dist_{Y_i}(\tilde{g}_i,g)} \right )\\ 
    &= \alpha^{-1} \left (\Ex[Y_i \sim \nu_Y]{\dist_{Y_i}(f|_{Y_i},\coboundary g_i)} + 2 \Prob[(Y_i,v) \sim \nu_{0,y}]{\tilde{g}_i(v) \ne g(v)} \right )\\
    &\leq \alpha^{-2}\beta^{-1} \varepsilon + 2\alpha^{-1}\Prob[(Y_i,v) \sim \nu_{0,y}]{\tilde{g}_i(v) \ne g(v)}.
\end{align*}
Here the first inequality is by \(\alpha\)-smoothness of \((\mu_1,\nu_1)\). The second inequality follows from the triangle inequality and the third inequality is by the fact that if \(\tilde{g}_i\) and \(g\) are equal on both vertices of an edge, then \(\coboundary \tilde{g}_i\) and \(\coboundary g\) are equal on that edge. The last  equality is because \(\coboundary \tilde{g}_i = \coboundary g_i\) and by definition of the expectation over distance, and the last inequality is by \eqref{eq:local-distances}. Altogether,
\begin{align} \label{eq:bound-dist-of-f-g-1}
     \dist (f, \coboundary g) &\leq \alpha^{-2}\beta^{-1} \varepsilon + 2\alpha^{-1}\Prob[(Y_i,v) \sim \nu_{0,y}]{\tilde{g}_i(v) \ne g(v)}.
\end{align}
We move on to bound $\Prob[(Y_i,v) \sim \nu_{0,y}]{\tilde{g}_i(v) \ne g(v)}$. 
\begin{align*}
\Prob[(Y_i,v) \sim \nu_{0,y}]{\tilde{g}_i(v) \ne g(v)} &\overset{(1)}{\leq} \alpha^{-1} \Prob[(v,Y_i) \sim \pi_{0,y}]{\tilde{g}_i(v)\ne g(v)}\\
&=  \alpha^{-1} \Ex[v]{\Prob[Y_i \in \A^v(0)]{\tilde{g}_i(v)\ne g(v)}}\\
&\overset{(2)}{\leq} \eta^{-1} \alpha^{-1} \Ex[v]{\Prob[\set{Y_i,Y_j} \sim \A^v(1)]{\tilde{g}_i(v) \ne \tilde{g}_j(v)}} \\
&= \alpha^{-1} \eta^{-1}  \Prob[\set{Y_i,Y_j}_v \sim \A(1)]{\tilde{g}_i(v) \ne \tilde{g}_j(v)}\\
&\overset{\eqref{eq:g_and_H}}{=} \alpha^{-1} \eta^{-1} \Prob[\set{Y_i,Y_j}_v \sim \A(1)]{h((Y_i,Y_j)_v) \ne \coboundary \ell((Y_i,Y_j)_v)}\\
&= \alpha^{-1} \eta^{-1} \dist(h,\coboundary \ell) \\
&\overset{(3)}{\leq} \alpha^{-1} \gamma^{-1} \eta^{-1} wt(\coboundary h).
\end{align*}
Here \((1)\) is because for every vertex, if \(v \notin A\) then \(\tilde{g}_i(v)=g(v)\), hence the event is contained in \(A_Y\) and we can use \((A_Y,\alpha)\)-smoothness and with \pref{claim:property-of-smooth-dist}. The next inequality \((2)\) is by \pref{claim:expander-and-majority} and the fact that every local graph \(\A^v\) is a \(\eta\)-expander. \((3)\) comes from \(\gamma\)-coboundary expansion of \(\A\).

We conclude the proof by showing that
\(wt(\coboundary h) \leq \alpha^{-1} \varepsilon + 3 \alpha^{-2} \beta^{-1} \varepsilon\). 

Let us first show that if \(\coboundary h (\set{Y_i,Y_j,Y_k}_{u,v,w}) \ne 0\) then either \(\coboundary f(uvw) \ne 0\) or for one of the three edges \(f(xy) \ne \coboundary g_{m}(xy)\) (where \(xy\) in one of the edges in \(uvw\) and \(Y_m\) is the sub-complex that contains this edge, that was selected in the triangle). Otherwise,
\begin{align*}
    \coboundary h (\set{Y_i,Y_j,Y_k}_{u,v,w}) &= h((Y_i,Y_j)_{u}) \cdot h((Y_j,Y_k)_{v})  \cdot h((Y_k,Y_i)_{w}) \\
    &= (g_i(u)^{-1}g_j(u)) \cdot (g_j(v)^{-1}g_k(v)) \cdot (g_k(w)^{-1}g_i(w)) \\
    &= (g_i(w)^{-1} g_i(w))\cdot  (g_i(u)^{-1}g_j(u)) \cdot (g_j(v)^{-1}g_k(v)) \cdot (g_k(w)^{-1}g_i(w)) \\
    &\overset{f=\coboundary g}{=} g_i(w)^{-1} f(uv) f(vw) f(wu)
    g_i(w) \\
    &= g_i(w)^{-1} \coboundary f(uvw) g_i(w) \\
    &\overset{\coboundary f = Id}{=} Id.
\end{align*}

Thus
\[\wt(h) \leq \Prob[uvw \sim \pi_2]{\coboundary f(uvw) \ne 0} + 3\Prob[(uv,Y_i) \sim \pi_{1,y}]{f(uv) \ne \coboundary g_i(uv)}.\]

By \(\alpha\)-smoothness of \((\pi_2,\mu_2)\) it holds that \(\Prob[uvw \sim \pi_2]{\coboundary f(uvw) \ne 0} \leq \alpha^{-1}\varepsilon\). 
By smoothness \(\alpha\)-smoothness of \((\pi_{1,y},\nu_{1,y})\), the rightmost term is bounded by 
\[3\Prob[(uv,Y_i) \sim \pi_{1,y}]{f(uv) \ne \coboundary g_i(uv)} \leq 3\alpha^{-1}\Prob[(uv,Y_i) \sim \nu_{1,y}]{f(uv) \ne \coboundary g_i(uv)} = 3\alpha^{-1} \Ex[Y_i \sim \nu_y]{\dist(f|_{Y_i},\coboundary g_i)}.\]
By \eqref{eq:local-distances} we bound this by \(3\alpha^{-2} \beta^{-1} \varepsilon\). Combining the above yields \(\wt(h) \leq \alpha^{-1}\varepsilon + 3\alpha^{-2} \beta^{-1} \varepsilon\).

Putting things back together we have
\begin{align}
    \dist(f,\coboundary g) &\leq \alpha^{-2}\beta^{-1}\varepsilon + 2 \alpha^{-2}\gamma^{-1} \eta^{-1} \wt(h) \\
    &\leq \alpha^{-1}\beta^{-1}\varepsilon + 2\alpha^{-2} \gamma^{-1} \eta^{-1} \left( \alpha^{-1} + 3\alpha^{-2}\beta^{-1} \right ) \varepsilon\\
    &\leq \frac{10}{\alpha^4 \beta\gamma\eta } \varepsilon.
\end{align}
\end{proof}
\begin{remark}
    We could have made this theorem tighter by accounting for different smoothness parameters \(\alpha_1,\alpha_2,\dots\) instead of a single \(\alpha\). However, in all our examples this tightening would not have gained a significant improvement. We stated the theorem in this generality to keep it simpler. 
\end{remark}
\section{Blow-ups of simplicial complexes} \label{sec:blow-up}
We saw above that analyzing coboundary expansion via the GK decomposition naturally gives rise to simplicial complexes with multi-edges. We show under mild conditions that coboundary expansion of a simplicial complex with multi-edges reduces to the coboundary expansion of its non-multi-edge ``flattening''. 

Recall that we denote by $\set{u,v}_i$ an edge between vertices $u$ and $v$ that has label $i$. We also denote by 
\(\set{u,v,w}_{i,j,k}\) a triangle with vertices \(u,v,w\) and edges \(\set{u,v}_i,\set{v,w}_j\) and \(\set{w,u}_k\).

Let us begin with a definition of a blow-up graph. Let \(G,G'\) be two graphs on the same vertex set \(V\). We say that \(G'\) is a blow-up of \(G\) if \(G\) is a simple graph (that has no multi edges) and \(\set{u,v} \in E(G)\) if and only if there exists an edge \(\set{u,v}_i \in E(G')\), and moreover, for every \(\set{u,v}\in G\), the total weight of edges from $u$ to $v$ in $G'$ equals the weight of $\set{u,v}$ in $G$, 
\[\Prob[G]{\set{u,v}} = \sum_{{i}:\,{\set{u,v}_i\in E(G')}}\Prob[G']{\set{u,v}_i}.\]
Similarly, for simplicial complexes, \begin{definition}[Blow-up]
    Let \(X\) be a \(d\)-dimensional simplicial complex. Let \(\tilde{X}\) be a simplicial complex with multi-edges. We say that \(\tilde{X}\) is an \emph{blow-up} of \(X\) if \(\tilde{X}(0) = X(0)\) and such that for every face \(s \in X\), the probability of sampling a face \(\tilde{s} \in \tilde{X}\) whose vertices are \(s\), is the probability of sampling \(s \in X\).
\end{definition}

\begin{definition}[Label graph]
    Let \(X\) be a \(d\)-dimensional simplicial complex and \(\tilde{X}\) be a blow-up of \(X\). For a fixed edge \(e=\set{u,v} \in X(1)\) let \emph{the label graph} \(G^e=G^e(\tilde{X})\) be the graph whose vertices are all labels of labeled edges in \(\tilde{X}\) between \(u\) and \(v\), i.e. \[V(G^e) = \set{i : \set{uv}_i \in \tilde{X}}.\]
    In order to define the edges we first consider a bipartite graph that connects a label $i$ to a labeled triangle $\set{u,v,w}_{i,j,k}$ that contains $\set{u,v}_i$. We then let the label graph of the edge $\set{u,v}$ be the two step walk on this bipartite graph. In other words, $i$ is connected to ${i'}$ if there are $w,j,k$ such that the completions $\set{u,v,w}_{i,j,k}$ and $\set{u,v,w}_{i',j,k}$ exist. 
\end{definition}

\begin{lemma}\torestate{\label{lem:hdx-blow-up}
Let \(X\) be a \(\beta\)-coboundary expander with respect to coefficients $\Gamma$. Let \(\tilde{X}\) be a blow-up of \(X\) such that for every \(e=uv \in X(1)\), \(G^e\) is an \(\eta\)-edge expander. Then \(\tilde{X}\) is a \(\frac{1}{5}\eta \beta\)-coboundary expander with respect to coefficients $\Gamma$.}
\end{lemma}

Let us give some intuition for \pref{lem:hdx-blow-up} in case $\Gamma=\mathbb{F}_2$. Fix $h\in C^1(\tilde X,\mathbb{F}_2)$, and assume $\coboundary h=0$. Every triangle \(t=\set{u,v,w}_{i,k,\ell} \in \tilde{X}(2)\) defines an equation
\begin{equation} \label{eq:coboundary-equation-for-blow-up}
\coboundary h(t) = h(\set{u,v}_i) + h(\set{v,w}_j) + h(\set{w,u}_k)=0.
\end{equation}
Let \(EQ\) be the set of all such equations (i.e. the solution space of \(EQ\) is \(Z^1(\tilde{X},\mathbb{F}_2)\)). As a warm up, we want to make sure that \(Z^1(\tilde{X},\mathbb{F}_2) = B^1(\tilde{X},\mathbb{F}_2)\), i.e. that \(Z^1(\tilde{X},\mathbb{F}_2)\) is the set of all \(h:\tilde{X}(1) \to \mathbb{F}_2\) such that there exists some \(g:\tilde{X}(0)\to \mathbb{F}_2\) such that 
\[h(\set{u,v}_i) = g(u)+g(v).\]
In particular, we note that for any such \(h\), \(h(\set{u,v}_i)\) is independent of the label \(i\), or in other words for every \(i,j\),
\[h(\set{u,v}_i) = h(\set{u,v}_j).\]
We learn from this that a necessary condition for \(\tilde{X}\) to have \(Z^1(\tilde{X},\mathbb{F}_2) = B^1(\tilde{X},\mathbb{F}_2)\), is that the equations \(h(\set{u,v}_i) = h(\set{u,v}_j)\) are spanned by \(EQ\). For a fixed edge \(\set{u,v} \in X\) and two of its labels \(i,j\), if there are two triangles \(\set{u,v,w}_{i,k,\ell},\set{u,v,w}_{j,k,\ell} \in \tilde{X}\), then by adding up their two corresponding equations as in \eqref{eq:coboundary-equation-for-blow-up} yields
\[h(\set{u,v}_i) = h(\set{u,v}_j).\]
Moreover, one can observe that if \(G^{\set{u,v}}\) is connected, then all the equations
\[h(\set{u,v}_i) = h(\set{u,v}_j)\]
are spanned by \(EQ\). The fact that \(Z^1(\tilde{X},\mathbb{F}_2) = B^1(\tilde{X},\mathbb{F}_2)\) is not enough to lower bound coboundary expansion. However, this hints that a robust notion of connectivity for the \(G^{\set{u,v}}\), may be useful for proving such a lower bound.

Indeed, Suppose that \(h(\set{u,v}_i) = h(\set{u,v}_j)\) is violated for many labels \(i,j\). Let \(V_0,V_1 \subseteq V(G^{\set{u,v}})\) be such that \(i \in V_x\) if \(h(\set{u,v}_i) = x\). If \(h(\set{u,v}_i) = h(\set{u,v}_j)\) is violated for many pairs of labels \(i,j\) that are not necessarily edges, then both \(V_0,V_1\) are large. If both \(V_0\) and \(V_1\) are large, then by the expansion of \(G^{\set{u,v}}\), many edges cross between \(V_0\) and \(V_1\). Every such edge corresponds to two triangles \(\set{u,v,w}_{i,k,\ell}, \set{u,v,w}_{j,k,\ell}\). And because \(h(\set{u,v}_i) \ne h(\set{u,v}_j)\) it holds that \(h(\set{u,v}_i) + h(\set{u,v}_j)=1\). On the other hand, \(h(\set{u,v}_i) + h(\set{u,v}_j)\) is the sum of equations
\[h(\set{u,v}_i) + h(\set{v,w}_k) + h(\set{w,u}_\ell) = 0\]
and 
\[h(\set{u,v}_j) + h(\set{v,w}_k) + h(\set{w,u}_\ell) = 0\]
and if the sum is non zero, then at least one of these equations is violated.

The contra-positive argument is that if most of these equations sum up to zero, that is,
\[\Prob[w,i,k,\ell]{\coboundary h(\set{u,v,w}_{i,k,\ell} =0 } \approx 0,\]
then the cut \(V_0,V_1\) must have few crossing edges. In an expander graph this implies that one of the sets is small, or in other words, that \(i \mapsto h(\set{u,v}_i)\) is almost a constant. Thus defining the majority function \(Mh:X(1) \to \mathbb{F}_2\), \(Mh(\set{u,v}) = \maj_{i}h(\set{u,v}_i)\) we get by the discussion above that
\[\Prob[\set{u,v}_i \in \tilde{X}(1)]{h(\set{u,v}_i) = Mh(\set{u,v})} \approx \wt(\coboundary h).\]
Now we use the coboundary expansion in \(X\) to correct \(Mh\) to some \(\coboundary g \in B^1(X,\mathbb{F}_2)\). This \(g\) will have the property that \(\Omega(\beta) \dist(h,\coboundary g) \leq \wt(\coboundary h)\). 

\begin{proof}[Proof of \pref{lem:hdx-blow-up}]
    Let \(h:\tilde{X}(1) \to \Gamma\) be such that \(wt(\coboundary h) = \varepsilon\). 
    Let us define \(Mh:\tilde{X}(1) \to \Gamma\) to be \(Mh(\set{u,v}_j) = \maj_{i \in V(G^{\set{u,v}})} \set{h(\set{u,v}_{i}}\). 
    By the triangle inequality,
    \begin{equation}\label{eq:triangle}
        \dist(h,B^1(\tilde{X})) \leq \dist(h,Mh) + \dist(Mh,B^1(\tilde{X})).
\end{equation}

    The first term in the right hand side is bounded by
    \begin{align}\label{eq:disthMh}
    \begin{split}
        \dist(h,Mh) &= \Ex[uv \in X(1)]{\Prob[i]{h(\set{u,v}_i) \ne Mh(\set{u,v}_i)}} \\
        &\leq \eta^{-1} \Ex[uv \in X(1)]{\Prob[i,j]{h(\set{u,v}_i) \ne h(\set{u,v}_j) }} \\
        & \leq \eta^{-1} \Ex[uv \in X(1)]{2\Prob[\set{u,v,w}_{i,k,\ell}]{\coboundary h(\set{u,v,w}_{i,k,\ell}) \ne 0}} =  2 \eta^{-1}\varepsilon.
        \end{split}
\end{align}
    where the first inequality follows from the definition of edge expansion, and the second inequality follows from the fact that if \(h(\set{u,v}_i) \ne h(\set{u,v}_j)\) then one of the triangles \(\set{u,v,w}_{i,k,\ell}, \set{u,v,w}_{j,k,\ell}\) is not satisfied, since the only difference between the assignment of \(h\) to the edges of these two triangles is the assignment of \(h(\set{u,v}_i)\) and \(h(\set{u,v}_j)\). The distribution of sampling \(uv \in X(1)\) and then a triangle \(\set{u,v,w}_{i,j,k}\) conditioned on \(uv\) is just the distribution over triangles in \(\tilde{X}(2)\), hence the final equality with \(2\eta^{-1}\) times the weight of \(\coboundary h\).

    We turn to the second term on RHS of \eqref{eq:triangle}. First, we observe that 
    \begin{equation}\label{eq:wtMh}
        wt(\coboundary Mh) \leq 3\prob{Mh \ne h} + wt(\coboundary h) \leq 3 \eta^{-1} \varepsilon.
    \end{equation} 
    The reason is that we can split triangles into those where there is an edge with $Mh\neq h$ and those where \(Mh=h\) on all three edges. The former are accounted for in the first term, and in the later case \(\coboundary Mh(t) = \coboundary h(t)\). 
    
    Since $Mh$ does not depend on the label of an edge, it gives rise to $\bar h\in C^1(X)$ defined by choosing for each edge $\set{u,v}$ an arbitrary label $i$ and setting $\bar h(\set{u,v}) = Mh(\set{u,v}_i)$. Clearly $\dist_{\tilde X}(Mh,B^1(\tilde{X})) = \dist_X(\bar h, B^1(X))$ and $wt_X(\coboundary \bar h) =  wt_{\tilde X}(\coboundary Mh)$ so using the coboundary expansion of $X$,  
\begin{equation}\label{eq:distMh}
    \dist(Mh,B^1(\tilde{X})) = \dist(\bar h, B^1(X)) \leq \beta^{-1} wt(\coboundary \bar h) = \beta^{-1} wt(\coboundary Mh) \leq 3\beta^{-1}\eta^{-1} \varepsilon .
    \end{equation}
where the last inequality comes from \eqref{eq:wtMh}. 
Plugging \eqref{eq:distMh} and \eqref{eq:disthMh} into \eqref{eq:triangle}
yields the result.
\end{proof}
\section{General expansion of the faces complex} \label{sec:generic}
In this section we prove a general bound on the swap coboundary expansion of coboundary expanders.
\restatetheorem{thm:generic-swap-expansion}
The following corollary follows directly from \pref{thm:cosystolic-expansion-from-link-coboundary-expansion}.
\begin{corollary}
    Let \(X\) be an \(n\)-dimensional simplicial complex. Let \(r\) be such that \(7r+8 \leq n\). Assume that for every \(-1<m \leq r\) and \(s \in X(m)\), \(h^1(X_s) \geq \beta\) and that \(X\) is a \(\lambda\)-two sided local spectral for \(\lambda \leq \exp(-O(r \log \beta))\), then \(X\) is a \((\exp(-O(r \log \beta)),r)\)-swap cosystolic expander.
\end{corollary}

We note that a tighter analysis could perhaps lose the \(O\) inside the expression, which would perhaps allow us to get a sub exponential bound if one could show that most links have \(h^1(X_s^{\set{j_1,j_2,j_3,j_4,j_5}}) \geq 1-o(1)\). However, as the state of the art is today, we do not know of such bounds in almost any complex and therefore we did not try to optimize the constant.

The bounds of \pref{thm:generic-swap-expansion} can be improved for certain complexes. For example \pref{thm:coboundary-expansion-intro} shows a better bound for spherical buildings with sufficiently large field size and dimension. For the complete complex we show, in the end of this section, a {\em constant} lower bound on the swap coboundary expansion.\\

We first prove \pref{thm:generic-swap-expansion} for partite complexes (see \pref{prop:colored-exponential-decay-bound}). In \pref{sec:simplebound} we give a simpler exponential bound, and in \pref{sec:improved-bound-faces-swap-gen} we derive an improved bound. We then use a partitification reduction (in \pref{sec:gen-partitif}) to extend the proof to any complex.
\subsection{Exponentially decaying coboundary expansion for partite complexes}\label{sec:simplebound}
In this section we prove the theorem for the important case of $X$ being an $n$-partite complex. 
Recall that for mutually disjoint sets of colors \(J = \set{c_1,c_2,\dots,c_\ell}\) the colored faces complex \(\FX[J]\) is the \(\ell\)-partite complex whose vertices are \(\FX[J](0)=\bigcup_{i=1}^\ell X[c_i]\) and whose top-level faces are all \(s=\set{w_1,\dots,w_\ell}\) such that $w_i\in X[c_i]$ and \(\bigcup_{i=1}^\ell w_i \in X\).

\begin{proposition} \label{prop:colored-exponential-decay-bound}
    Let \(X\) be a \(n\)-partite complex that is a \(\lambda\)-local spectral expander for \(\lambda \leq \frac{1}{2r^2}\). Let \(\ell \geq 5\) and let  \(J=\set{c_1,c_2, \dots ,c_\ell}\) be a set of mutually disjoint colors \(c_j \subseteq [n]\), $\card {c_j}\leq r$. Denote by \(R=\sum_{j=1}^\ell |c_j|\). Let \(\beta > 0\) and assume that for every \(I=\set{i_1,i_2,\dots,i_\ell}\) such that \(i_j \in c_j\) and every \(w \in X^{\cup J \setminus I}\), \(h^1(X_w^I) \geq \beta\). Then \(h^1(\FX[J]) \geq \beta_1^{R}\) for \(\beta_1 = \Omega_{\ell}(\beta)\).
\end{proposition}

\begin{proof}
Fix \(\ell \geq 5\). The proof is via induction on \(R\) (for all partite complexes \(X\) simultaneously). The base cases are when either one of the \(c_i = \emptyset\) or when \(|c_1|=|c_2|=\dots=|c_\ell|=1\). If one of the \(c_j = \emptyset\) then \(h^1(X) = \Omega(1)\) by \pref{claim:cone-is-coboundary-expander}. Otherwise all the \(c_j = \set{i_j}\)'s are singletons. In this case \(\FX[J] \cong X^{\set{i_1,i_2,\dots,i_\ell}}\) and \(h^1(\FX[J]) \geq \beta\) by assumption.

Now let us assume that the proposition holds for \(R\) and prove it for \(c_1,c_2,\dots ,c_\ell\) such that \(\sum_{j=1}^\ell |c_j| = R+1\) and all \(c_j \ne \emptyset\). Fix such \(c_i\)'s.

We choose some \(I=\set{i_1,i_2,\dots,i_\ell}\) such that \(i_j \in c_j\) and show that 
\begin{equation} \label{eq:induction-for-generic-bound-prop}
    h^1(\FX[J]) \geq  \Omega(h^1(X^I))\cdot \min_{j \in [\ell], v\in {X[i_j]}} h^1(\FX[J_j]_v)
\end{equation}
where \( J_j = \set{c_1',c_2',\dots,c_\ell'} \) such that \(c_j'=c_j \setminus \set{i_j}\) and for \(m \ne j\), \(c_m'=c_m\). Observe that \(\sum_{j=1}^\ell |c_j'| = R\) so if \eqref{eq:induction-for-generic-bound-prop} holds then by induction \(h^1(\FX[J]) \geq \Omega(\beta) \cdot \beta_1^R \geq \beta_1^{R+1}\). Indeed, the assumption that for every \(I=\set{i_1,i_2,\dots,i_\ell}\) such that \(i_j \in c_j\) and every \(w \in X^{\cup J \setminus I}\), \(h^1(X_w^I) \geq \beta\) implies that the same holds for \(X_v\) and the \(c_j'\)'s, so we are justified to apply an inductive argument.

We show \eqref{eq:induction-for-generic-bound-prop} by applying \pref{thm:decomposition-to-coboundary-expanders} to the following GK-decomposition \((\mathcal{Y},\A,\nu,\pi)\).
\begin{enumerate}
    \item Let \(\mathcal{Y} = \sett{Y_v}{v\in X^I(0)}\), such that \(Y_v \subseteq \FX[J]\) is the $\ell$ partite complex induced by vertices that either contain \(v\) or are in \(\FX[J](0) \cap X_v\).
    \item We define $\pi$ so that \(\set{v_1,v_2,v_3}_{w,w',w''} \sim \pi\) is chosen as follows:
    \begin{enumerate}
        \item We sample \(\set{w_1,w_2,\dots,w_\ell} \in \FX[J](\ell-1)\).
        \item We sample a random triangle \(t = \set{w,w',w''} \subseteq \set{w_1,w_2,\dots ,w_\ell}\).
        \item We sample distinct \(i_1,i_2,i_3 \in I\) and \(t'=\set{v_1,v_2,v_3}\) to be such that \(col(v_j) = i_j\) and such that every \(v_j\) is contained in one of the \(w_1,w_2,\dots ,w_\ell\) that were sampled in the first step.
        \item We randomly reorder and output \(\set{v_1,v_2,v_3}_{w,w',w''}\).
    \end{enumerate}
    We note that \(\set{w,w',w''}\) is distributed by the triangle distribution of \(X\), \(\mu_{2,\FX[J]}\).
    \item We define \(\nu\) to the a marginal of \(\pi\). That is, we first sample \(\set{v_1,v_2,v_3}_{w,w',w''} \sim \pi\) and then take one of the three \(v_j\)'s and output \((Y_{v_j}, \set{w,w',w''})\).
    \item We identify the unlabeled triangles and edges of \(\A\) with those of \(X^I\).
    \end{enumerate}
    
    Let \(v\) be a vertex of color, say, \(i_1 \in c_1\). Then \(Y_v \cong \FX[J_1]_v\). This is because \(Y_v[c_1] = \sett{\set{v} \dunion w}{w \in X_v[c_1']}\) and for all \(m \ne j\), \(Y_v[c_m] = X_v[c_m]\). Moreover, \(\nu\) is defined such that the distribution of \(\nu|_{Y_v}\) is the distribution over triangles in \(\FX[J_1]_v\) (up to the identification of \(Y_v[c_1] \cong X_v[c_1']\)). 
    It follows that
    \[\min_{v} h^1(Y_v) = \min_{j \in [\ell], v\in {X[i_j]}} h^1(\FX[J_j]_v).\]

    Before analyzing the agreement complex, let us consider the smoothness required in \pref{thm:decomposition-to-coboundary-expanders} for this decomposition. For this we note that \(\mu_{2,\FX[J]} = \nu_2 = \pi_2\) and that \(\mu_{1,\FX[J]}=\nu_1, \pi_{0,y}=\nu_{0,y}, \pi_{1,y}=\nu_{1,y}\). Thus all pairs of distributions are \(1\)-smooth.

    Now let us consider the local graphs of labels \(w\) of edges in \(\A\). Let \(w\) be of color (say) \(c_1\). The local graph \(\A^{w}\) is an \(\ell\)-partite graph. The first part is the (only) \(v \in w\) of dimension \(i_1\). The rest of the parts are the vertices of colors \(i_2,i_3,\dots,i_\ell\) in \(X_w(0)\). Two \(v_j,v_{j'}\) are connected if they belong to different parts and one of the following holds:
    \begin{enumerate}
        \item \(col(v_j) = 1\).
        \item \(col(v_{j'})=1\).
        \item \(\set{v_j,v_{j'}}\in X_w(1)\).
    \end{enumerate}
    The edge distribution is by choosing two distinct \(i_j,i_{j'}\) and then choosing a uniform edge in the first two cases, or and edge in \(X_w[\set{i_j,i_{j'}}]\) in the third case.
    
    This is a constant expander: Observe that there is a graph homomorphism between this graph and \(K_\ell\) (the complete graph over \(\ell\) vertices). Thus by \pref{claim:expansion-from-subexpanders} one needs to verify that for every \(i_j \ne i_{j'}\) the bipartite graph induced by vertices of these colors is a bipartite expander. If \(i_j=i_1\) or \(i_{j'}=i_1\) then this is a complete bipartite graph, otherwise this is the colored swap walk between \(X_w[i_j]\) to \(X_w[i_{j'}]\) which is a \(\lambda\)-expander by assumption that \(X\) is a \(\lambda\)-local spectral.

    It remains to lower bound the coboundary expansion of \(\A\). As seen above, the agreement complex \(\A\) is a blow-up of \(X^I\) because every unlabeled triangle \(\set{v_1,v_2,v_3}\) is chosen with the same probability as in \(X^I\). 
    \begin{claim}\label{claim:annoying-graph-is-an-expander}
        Let \(X\) be a \(\lambda < \frac{1}{2r^2}\) local spectral expander. Then for every unlabeled edge \(e_0=\set{v_1,v_2} \in X^I(1)\), the label graph \(G_{e_0}\) is an \(\Omega_\ell(1)\)-edge expander. 
    \end{claim}
    We defer the proof of \pref{claim:annoying-graph-is-an-expander} to \pref{app:outstanding-coboundary-expansion-proofs} since it is a straightforward calculation. Believing this claim however we get, by \pref{lem:hdx-blow-up}, that \(h^1(\A) \geq \Omega(h^1(X^I))\).

    By \pref{thm:decomposition-to-coboundary-expanders}, it holds that \(h^1(\FX[J]) \geq 
    \Omega(h^1(\A)) \cdot \min_{j,v} h^1(F^{J_j}X_v)\). We fix $\beta_1 = \Omega(\beta)$ so that \(h^1(\FX[J]) \geq 
    \beta_1 \cdot \min_{j,v} h^1(F^{J_j}X_v)\) and get by induction that \(h^1(\FX[J]) \geq \beta_1^{R+1}\). 
\end{proof}

\subsection{Improved bound} \label{sec:improved-bound-faces-swap-gen}
In \pref{prop:colored-exponential-decay-bound} we show a bound of the form \(h^1(\FX[J]) \geq \beta_1^{\sum_{i=1}^\ell |c_i|}\) where the base of the exponent was the worst coboundary expansion in any \(\ell\)-colored link of \(X\). As we shall see in \pref{sec:proof-of-faces-complex-lower-bound}, sometimes this \(\beta_1\) is sub-constant. However, looking closely at the proof, we can observe that the worst case link expansion bound we use in the GK-decomposition can be replaced with a constant lower bound in all but $O(1)$ of the induction steps.

To be more precise, in the inductive step we get to choose some \(I=\set{i_1,i_2,\dots,i_\ell}\) such that \(i_1 \in c_1,\dots, i_\ell \in c_\ell\) and using \pref{thm:decomposition-to-coboundary-expanders} obtain a bound of the form
\[h^1(\FX[J]) \geq \underbrace{\Omega(h^1(X^{\set{i_1,i_2,\dots ,i_\ell}}))}_{A(X,I)} \cdot \underbrace{\min_{j \in [\ell], v \in X[i_j]} h^1(\FX[J_j]_v)}_{B(X,J,I)}\]
where \(J_j=\set{c_1',c_2',\dots,c_\ell'}\) is such that \(c_j'=c_j \setminus \set{i_j}\) and \(c_{m}'=c_m\) for \(m \ne j\).

The first term \(A(X,I)\) is the term we bound directly. The second term \(B(X,J,I)\) is the one we bound using induction, by recursively doing more GK-decompositions. In \pref{prop:colored-exponential-decay-bound} we chose \(I\) arbitrarily and bounded \(A(X,I)\) by the worst possible expansion of \(X^I\) (and later inside the induction, this was bounded by the worst possible coboundary expansion of \(X_w^I\) for some link \(X_w\)).

The proof of \pref{prop:colored-exponential-decay-bound} shows that we can try to optimize over \(I\), that is, the bound we actually get is
\begin{equation} \label{eq:explained-improved-exp-decay}
    h^1(\FX[J]) \geq \max_{I} A(X,I) \cdot B(X,J,I).        
\end{equation}
In this case, the \(A(X,I) = \Omega(h^1(X^I))\) term is straightforward, but one needs to better understand what happens to the \(B(X,J,I)\) term when we go further down the induction.

One can understand this term using the perspective of a two-player game on a tree, as we explain here.

First let us describe the tree.
\begin{enumerate}
    \item The nodes of the tree correspond to \((J',X_w)\) for \(J' \leq J\) and \(w \in X[\cup J \setminus \cup J']\).
    \item The leaves (i.e. basis of the induction) are all \((J',X_w)\) so that either \(J'\) contains an empty color, or so that all colors are singletons.
    \item The root is \((J,X)\) (i.e. \(X=X_{\emptyset}\), which is consistent with \(\emptyset \in X[\cup J \setminus \cup J]\)). 
    \item For every non-leaf \((J',X_{w_1})\), its children are the \((J'',w_2)\) such that \(J'' \leq J'\) and \(w_2 = w_1 \dunion \set{v}\) for some vertex \(v\). (n particular this means that \(|\cup J''| = |\cup J'|-1\)).
\end{enumerate} 

The two-player game is the following, where Player \(1\) tries to maximize a value \(h\), and Player \(2\) tries to minimize it. The game begins at the root of the tree \((X,J)\) with value \(h=h_0=1\). At the first step, Player \(1\) chooses some \(I \leq J\) as above and `gains' \(h^1(X^I)\), i.e. \(h_1 = h_0 \cdot h^1(X^I)\). Then Player \(2\) chooses a child \((J_j,X_v)\) such that \(\cup J \setminus \cup J_j \in I\) (this corrsponds to the minimum in \(B(X,J,I)\). The two players traverse to \((J_j,X_v)\) and the game continues. This process corresponds to the first step of the decomposition in \eqref{eq:explained-improved-exp-decay} (for \(X\)).

In general, when we are at a node \((J',X_w)\), the first player chooses some \(I = \set{\set{i_1},\set{i_2},\dots,\set{i_\ell}} \leq J'\), and `gains' \(h^1(X_w^I)\) (i.e. \(h_{i+1}:=h_i \cdot h^1(X_w^I)\)). Then, if \((J',X_w)\) is not a leaf, the second player chooses a child \((J'',X_{w'})\) such that \(\cup J' \setminus \cup J'' \in I\) and the game continues on \((J'',X_{w'})\). If \((J',X_w)\) is a leaf the game ends and its value is the current \(h_{i+1}\). This step corresponds to \eqref{eq:explained-improved-exp-decay}, but for \(X_w\) instead of \(X\). The maximal value of the game is the largest \(\hat{h}\) such that for any possible Player \(2\), there exists a strategy of Player \(1\) that attains \(h \geq \hat{h}\) at the end of the game. Modifying \pref{prop:colored-exponential-decay-bound} can prove that \(h^1(\FX[J])\geq \hat{h}\).

\medskip

Using this point of view let us see a sufficient condition on \(X\) that implies a better \(\hat{h}\) than the one obtained by the minimum. 

For \(J' \leq J\) let us denote by \[d(J') = \sum_{c_j' \in J'} |c_j'|.\]
We observe that the \(i\)-th round of the game we are at a node \((J',X_w)\) such that \(d(J')=R-i+1\).

Suppose that there are values \(T_q\) such that for every \((J',X_w)\) with \(d(J')=q\), there exists a choice \(I\) for Player \(1\) so that they gain \(h^1(X_w^I) \geq T_q\). In every step in the gave \(d(J')\) decreases by \(1\), and it is never less than \(1\). Thus it is easy to see that \(\hat{h} \geq \prod_{q=1}^R T_q\).

Let us give a more mathematical description to this idea.
Let \(q \leq R\) be an integer. \(\mathcal{J}_q= \mathcal{J}_q(J)\) be all the \(J' = \set{c_1',c_2',\dots,c_\ell'} \leq J\) such that \(d(J')= q\). Let 
\[T_q(X,J) = \min_{(J',X_w), J' \in \mathcal{J}_q, w \in X[\cup J \setminus \cup J']}  \left ( \max_{i_1,i_2,\dots,i_\ell \text{ s.t. } i_j \in c_j'}  \left ( h^1(X_{w}^{\set{i_1,i_2,\dots ,i_\ell}}) \right ) \right ).\]

To state this explicitly, this is the largest \(T_q\) such that Player \(1\) is guaranteed to get when at a node \((J',X_w)\) where \(J' \in \mathcal{J}_q(J)\).

The following proposition follows directly from the discussion above.
\begin{proposition} \label{prop:improved-generic-lower-bound-for-colors}
    Let \(X\) be a partite \(\lambda\)-one sided local spectral expander for \(\lambda \leq \frac{1}{2r^2}\). Let \(J = \set{c_1,c_2,\dots ,c_\ell}\) and let \(R = \sum_{j=1}^\ell |c_j|\). Then \(h^1(\FX[J]) \geq \prod_{q=1}^R \Omega_{\ell}(T_q(X,J))\).
\end{proposition}

\begin{proof}
    In the proof of \pref{prop:colored-exponential-decay-bound} we showed \eqref{eq:induction-for-generic-bound-prop}. Now we show that \eqref{eq:induction-for-generic-bound-prop} implies this proposition as well. 

    By definition \(T_{R+1}(X,J)\) is the minimum over the root only, i.e. we are only looking at \((X,J)\) (since \(J\) is the only set with \(d(J)=R+1\)). Let \(I = \set{i_1,i_2,\dots,i_\ell}\) be such that \(h^1(X^{I})\) is maximized. Then \(h^1(X^I) \geq T_{R+1}(X,J)\).

    By \eqref{eq:induction-for-generic-bound-prop}, we have that
        \[h^1(\FX[J]) \geq \Omega(h^1(X^I) )\cdot \min_{j \in [\ell],v \in X[i_j]}(h^1(\FX[J_j]_v) \geq T_{R+1}(X,J)) \cdot \min_{j \in [\ell],v \in X[i_j]}(h^1(\FX[J_j]_v).\]
    since \(h^1(X^I) \geq T_{R+1}(X,J)\). By applying the induction hypothesis to the second term we get 
    \[\geq T_{R+1}(X,J) \cdot \min_{j \in [\ell], v \in X{[i_j]}}\prod_{q=1}^R T_q(X_v,J_j).\]
    Let us show that for every \(j,v\) in the minimum, 
     \[\prod_{q=1}^R T_q(X_v,J_j) \geq \prod_{q=1}^R T_q(X,J).\]
     In fact we will show that for every \(i_j\) and vertex \(v \in X[i_j]\), \(T_{q}(X_v,J_j) \geq T_{q}(X,J)\).

     The expression for \(T_q(X_v,J_j)\) is, by definition,
     \[T_{q}(X_v,J_j) = \min_{(J',X_w), J' \in \mathcal{J}_q(J_j), w \in X_v[\cup J_j \setminus \cup J']}  \left ( \max_{i_1,i_2,\dots,i_\ell \text{ s.t. } i_j \in c_j'}  \left ( h^1(X_{v \dunion w}^{\set{i_1,i_2,\dots ,i_\ell}}) \right ) \right ).\]
     By making a change of variables \(w' = v \dunion w\) and observing that \(w \in X_v[\cup J_j \setminus \cup J']\) if and only if \(w' \in X[\cup J \setminus \cup J']\), we can change this expression to
     \[T_{q}(X_v,J_j) = \min_{(J',X_{w'}), J' \in \mathcal{J}_q(J_j), v \in w' \in X[\cup J \setminus \cup J']}  \left ( \max_{i_1,i_2,\dots,i_\ell \text{ s.t. } i_j \in c_j'}  \left ( h^1(X_{w'}^{\set{i_1,i_2,\dots ,i_\ell}}) \right ) \right )\]
     
    As we can see this is the same minimum as in \(T_q(X,J)\), only that we consider less \(J'\) (since \(J' \leq J_j\) implies \(J' \leq J\)) and less \(w'\) (since in \(T_q(X_v,J_j)\) we require \(v \in w'\)). Thus \(T_{q}(X_v,J_j) \geq T_{q}(X,J)\) and the proposition follows.
\end{proof}

\subsection{Proof of \pref{thm:generic-swap-expansion}}\label{sec:gen-partitif}
Recall the definition of the partitification of a complex \(X\) in \pref{sec:preliminaries}.
\begin{proof}[Proof of \pref{thm:generic-swap-expansion}]
    By \pref{claim:coboundary-expansion-of-complex-as-good-as-its-partitification} \(h^1(F^r X) = \Omega(h^1((F^r X)^{\dagger_7}))\) so we show coboundary expansion of the \(7\)-partitification (not that the assumption on the dimensions of \(X\) imply that this complex exists).
    
    For \(i=1,2,\dots,7\) let \(c_i=\set{(i-1)r+1,(i-1)r+2, \dots ir}\).
    Note that \((F^r X)^{\dagger_7} \cong (X^{\dagger 7r})^{c_1,c_2,c_3,c_4,c_5,c_6,c_7}\).
    
    For any \(I=\set{i_1,i_2,\dots,i_7}\) such that \(i_j \in c_j\), and \(w \in X^{\dagger 7r}\) whose color is disjoint from \(I\), \((X_{w}^{\dagger_{7r}})^{i_1,i_2,\dots,i_7} \cong X_{p_1(w)}^{\dagger_7}\) We note that the dimension of \(X_w\) is at least \(5\) from the assumption on the dimension of \(X\). Thus by \pref{claim:coboundary-expansion-of-partitification-as-good-as-original} 
    \[h^1((X_{w}^{\dagger_{7r}})^{i_1,i_2,\dots,i_7}) = h^1(X_{p_1(w)}^{\dagger_7}) = \Omega(h^1(X_{p_1(w)})) = \Omega(\beta).\]
    By \pref{prop:colored-exponential-decay-bound} \(h^1((F^r X)^{\dagger_7}) = \beta_1^{r}\) for some \(\beta_1 = \Omega(\beta)\) and the theorem is proven.
\end{proof}

\subsection{Short detour: swap coboundary expansion of the complete complex}
\begin{claim} \label{claim:complete-faces-complex-is-a-coboundary-expander}
    Let \(\Delta\) be the complete complex on \(n\) vertices, then for any \(r \leq \frac{n}{6}\), \(h^1(F^r \Delta) \geq \frac{1}{5}\).
\end{claim}

The proof of this claim just follows from the fact that the faces of the complete complex has small diameter, and that every small enough cycle in it lies in the link of some other vertex. Thus one can contract any small cycle using the link.

\begin{proof}[Proof of \pref{claim:complete-faces-complex-is-a-coboundary-expander}]
    The group \(Sym(n)\) acts transitively on \(F^r \Delta(2)\) (which implies that \(Aut(F^r \Delta)\) acts transitively on the \(2\)-faces), therefore by \pref{lem:group-and-cones} it is enough to construct a cone of diameter \(5\). Let \(v_0 = \set{1,2,\dots,r}\). Let us begin with constructing paths. For \(u\) such that \(v_0 \cap u = \emptyset\) we set \(P_u = (v_0,u)\) and otherwise we set \(P_u = (v_0,w,u)\) for some \(w\) that is disjoint from \(v_0 \cup u\) (the fact that \(n \geq 3r\) allows us to find such a \(w\)).

    Now let us take an edge \(u_1,u_2\), and consider the cycle \(C_0 = P_{u_1} \circ (u_1,u_2) \circ P_{u_2}^{-1}\) = \((v_0,v_1,\dots,v_m,v_0)\) where \(m \leq 5\). Let \(A = \bigcup_{w \in C_0} w \subseteq [n]\). Then \(|A| \leq r |C_0| \leq 5r\). Recall that \(6r \leq n\) and thus there exists some \(x \in F^r \Delta\) such that \(x \cap A = \emptyset\) and in particular, for every edge \(\set{a,b}\) in \(C_0\), \(\set{x,a,b} \in F^r \Delta(2)\). Thus we can define \(T_{u_1 u_2}\) to be the sequence \(C_i = (v_0,x,v_1,x,v_2\dots,x,v_i,v_{i+1},\dots,v_m)\) for \(i=1,2,\dots, m\). The loop \(C_m\) is equivalent to the trivial loop \((v_0)\) since one can contract any \((x,v_i,x)\) back to \((x)\). Thus \(|T_{u_1 u_2}| \leq 5\) and by \pref{lem:group-and-cones} \(h^1(F^r \Delta) \geq \frac{1}{5}\). 
\end{proof}
\section[Coboundary expansion of the spherical-building faces-complex]{Coboundary expansion of the spherical building's faces complex} \label{sec:proof-of-faces-complex-lower-bound}
In this section we apply the tools we developed in previous sections to lower bound the coboundary expansion of the faces complex of the spherical building.
\begin{theorem}[More general version of \pref{thm:coboundary-expansion-intro}] \label{thm:coboundary-expansion}
   Let \(\d_1,n\) be integers such that \(n>\d_1^5\). There is some \(q_0=q_0(d)\) such that the following holds. Let \(q > q_0\) be any prime power.  Let \(\S\) be either \(SL_{n+1}(\mathbb{F}_q)\)-spherical building, or a link of a \(\d_1\)-face in the \(SL_{n+1}(\mathbb{F}_q)\) spherical building. Let \(\X = \FS[\d_1]\) be its faces complex. Then \(\X\) is a coboundary expander and \(h^1(\X) \geq \exp(-O(\sqrt{\d_1}))\).
\end{theorem}
From this theorem we immediately derive \pref{thm:cosys-expansion-intro}.
\begin{proof}[Proof of \pref{thm:cosys-expansion-intro} from \pref{thm:coboundary-expansion}]
    Let \(X\) be a complex as in \pref{thm:cosys-expansion-intro}. By \pref{thm:cosystolic-expansion-from-link-coboundary-expansion} and the fact that its links are coboundary expanders (since they are isomorphic to faces complexes of links of the spherical building), it holds that \(h^1(X) \geq \exp(-O(\sqrt{\d_1}))\) (as a cosystolic expander).
\end{proof}

\subsection{Notation for this section}
We denote by \(\X = \FS[\d_1]\) and \(\tilde{\X} = \FS[]\).

Let $\S$ denote the spherical building of $SL_n(\mathbb{F}_q)$. 
For a vertex $v\in \S(0)$, we denote its dimension as a linear subspace by $col(v)$. We can write the link of $v$ as
\[ \S_v = \sett{s'\dunion s''\in \S}{s'\in \S^{B'}, s''\in \S^{B''}}\]
where 
\[ B' = \sett{i\in\mathbb{N}}{0< i< col(v)}, \quad B'' = \sett{i\in\mathbb{N}}{ col(v)<i <n}.\]
For a general face $w = \set{v_0 < v_1<\cdots<v_r}\in \S$, we can write the link
\[ \S_w = \sett{s^0\dunion\cdots\dunion s^r\in \S}{\forall i\in \set{0,\ldots,r},\; s^i\in \S^{B_i}}\]
where $B_i = \sett{i\in\mathbb{N}}{col(v_{i-1}) < i< col(v_i)}$ for $0<i\leq r$, and $B_0=\sett{i\in\mathbb{N}}{0 < i< col(v_0)}$.\\

Recall that $\S$ is $n$-partite, and let $c\subset [n]$ be a set of colors. The same link decomposition holds for $\S_w^c$ with respect to a face $w$ whose colors are disjoint from $c$:
\[ \S^c_w = \sett{s^0\dunion\cdots\dunion s^r\in \S}{\forall i\in \set{0,\ldots,r},\; s^i\in \S^{B_i}}\]
where $B_i = \sett{i\in c}{col(v_{i-1}) < i< col(v_i)}$ for $0<i\leq r$, and $B_0=\sett{i\in c}{0 < i< col(v_0)}$.

\subsection[Proof Roadmap]{Proof Roadmap} \label{sec:roadmap}
The proof of \pref{thm:coboundary-expansion} proceeds by decomposing the complex $\X=\FS[\d_1]$ into smaller and smaller pieces, for which we are able to prove coboundary expansion, and so that we are able to go back up in the decomposition and deduce coboundary expansion for the entire complex. The decomposition takes the following steps (each based on a different technique):
\begin{enumerate}
\item \textbf {Color restriction} $\X \to     \set{\X^J}_J :$ In \pref{sec:proof-of-good-colors-to-all-the-complex} we lower bound the coboundary expansion of $\X$ by the coboundary expansion of color-restricted sub-complexes $\X^J$, for sets $J$ of $\sqrt{\d_1+1}$ mutually disjoint colors. 
    We show that if most complexes \(\X^J\) are \(\exp(-O(\sqrt{\d_1}))\)-coboundary expanders, then $\X$ is too. This step is similar to the proof of \pref{thm:coboundary-expansion-from-colors} which appears in \cite{DiksteinD2023cbdry}.
\item \textbf {Trickling down the links} 
     $\X^J \to \set{\X_s^{J'}}_s :$ In \pref{sec:trick} we lower bound the coboundary expansion of $\X^J$ by coboundary expansion of ``deep'' links \(\X_s^{J'}\). We take $s \in \X^J$ with size $|s|=|J|-5$ so that $J' = col(s)\setminus J$ consists of $5$ colors from $\C$, and $(\X_s)^{J'} = \X^{J'}_s$ is a $5$-partite complex. We use \pref{thm:cosystolic-expansion-from-link-coboundary-expansion} in an inductive ``trickling down'' manner, to deduce coboundary expansion of \(\X^J\) from coboundary expansion of \(\X_s^{J'}\). 
\item \textbf {Tensor decomposition:} 
In section \pref{sec:tensor} we show (\pref{lem:tensor}) that $\X^{J'}_s \cong K \otimes \bigotimes_{i=0}^r \FS[J_i]_s \cong K \otimes \tilde \X^{\tilde J}$ where $\tilde \X\supset \X$ stands for $\FS[]$, $K$ is a complete five-partite complex and $\tilde J$ still has five colors like $J'$ but is {typically} much narrower than $J'$, where the {\em width} of a color set $J'$ is $|\dunion J'| = \sum_{c\in J'} |c|$. Then by \pref{claim:triangle-complex} we show that the coboundary expansion of $\X^{J'}_s$ is at least a constant times the coboundary expansion of $\tilde \X^{\tilde J}_s$. 
\item \textbf {GK decomposition:}
     $\tilde \X_s^{\tilde J}\to \set{\tilde \X_s^{ I}}_{I}$. We use the GK decomposition (described in \pref{sec:gk-decomposition}) to repeatedly reduce the size of $|\dunion \tilde J|$ by deleting an element from one of the five colors at every step. We stop either if some $c_j$ becomes empty, or when all five colors are singletons, namely $|c_j|=1$ for $j=1,\ldots,5$. Each step costs us a constant multiplicative decrease, so the expansion constant loses a factor that is exponential in the width of $\tilde J$. 
\item \textbf {Spherical building:} We lower bound the coboundary expansion of $\tilde \X_s^I$ when $I=\set{\set{i_1},\set{i_3},\set{i_3},\set{i_4},\set{i_5}}$ is a set of singletons. We observe that $\tilde \X_s^I \cong \S_{\dunion s}^{\set{i_1,\ldots,i_5}}$, so we move to study such $5$-partite restrictions of links of the spherical building $\S$. This is done in \pref{sec:spherical-building-color-complexes} by a combination of cones arguments and recursive GK-decomposition steps. 
\end{enumerate} 

\subsubsection*{Notation and Parameters}
Fix $\d_1,n\in\mathbb{N}$ so that $\d_1^5\leq n$.  Fix $m = \sqrt{\d_1+1}$.

We let $\S$ denote the spherical building of $SL_{n+1}(\mathbb{F}_q)$. This is an $n$-partite complex whose colors are denoted $[n]$. We denote $\X = \FS[\d_1]$ and $\tilde \X = \FS[]\supset \X$. We let $\mathcal{C}= \binom{[n]}{\d_1+1}$ be the set of possible colors of vertices of $\X$.

We use $u,v$ to denote vertices of $\S$, and $w$ to denote vertices of $\X$, which are faces of $\S$. Faces of $\X$ are denoted by $s$.
We denote subsets of colors of $\FS[]$ that are mutually disjoint by the letters $J,I$ (so $J,I\in\FD[]$).  
Let us now prove a slightly weaker statement than \pref{thm:coboundary-expansion}, that is, let us show that
\begin{equation}
    h^1(\X) \geq \exp(-O(\sqrt{\d_1}\log^2 \d_1)).
\end{equation}
The proof of \pref{thm:coboundary-expansion} is almost identical, but removing the \(\log^2 \d_1\) in the exponent requires some more technicalities that obscure the main ideas in the proof. We prove the stronger version below.
\begin{proof}[Proof of \pref{thm:coboundary-expansion}(weaker version)]
To bound $h^1(\X)$ we follow the steps of the decomposition.
Let $\J$ be the set of well-spread $J$'s per \pref{def:good-colors}. By \pref{prop:prob-of-good-colors-tend-to-one}, at least half of the sets $J$ are in $\J$. Therefore, by \pref{lem:colorest},
\begin{equation}\label{eq:colors}
        h^1(\X) \geq \Omega(1)\cdot \min_{J\in\J} h^1(\X^J) 
\end{equation}
Fix $J\in\J$. By \pref{claim:trick}
\begin{equation}\label{eq:trickle}
        h^1(\X^J) \geq \exp(-O(m))\cdot \min_{s\in \X^J(m-6)} h^1(\X_s^J) 
\end{equation}
Fix any $s\in \X^J(m-6)$. By \pref{cor:tensor-bound}
\begin{equation}\label{eq:tensor}
        h^1(\X_s^J) \geq const\cdot  h^1(\tilde \X^{\tilde J}_s) 
\end{equation}
and by \pref{claim:link-of-a-faces-complex} \(\tilde \X^{\tilde J}_s \cong \FS[\tilde{J}]_{\cup s}\).

Next, denote by $\beta = \min_{w,I} h^1(\S_{\cup s \dunion w}^I)$ 
where the minimum is taken over sets $I$ consisting of five singletons such that $I\cup col(s)\leq J$, and \(w \in \S_{\cup s}\) such that \(col(w) \subseteq \cup J\) and \(col(w) \cap I = \emptyset\). 

By \pref{prop:colored-exponential-decay-bound},
\begin{equation}\label{eq:GK}
        h^1(\tilde \X_s^{\tilde J}) \geq const\cdot(\beta_1)^R
\end{equation}
where \(\beta_1 = \Omega(\beta)\) and $R = \sum_j |\tilde{c}_j|$. By item 3(c) of \pref{def:good-colors}, for every $\tilde{c}_j$, the number of indices in crowded bins is at most $O \left (\frac{ \d_1\log \d_1}{m\log m} \right )$, so in total $R = O \left (\frac{\d_1\log \d_1}{m\log m} \right )$.

Finally, by \pref{lem:spherical-links}, 
\begin{equation}
    \beta = \min_{w,I} h^1(\S_{\cup s \dunion w}^I) \geq \exp(-O(\log ^2 \d_1))
\end{equation}
We now plug in each equation into the previous one, to get the desired bound,
\[ 
    h^1(\X) \geq const \cdot \exp \left (- O\left(m + \frac{\d_1\poly\log \d_1}{m\log m}\right) \right ) = \exp(-O(\sqrt{\d_1}\log^2 \d_1)).
\] 
\end{proof}

\begin{proof}[Proof of \pref{thm:coboundary-expansion}(full version)]
Our starting point is \eqref{eq:tensor} in the proof above, namely that
\[h^1(\X) \geq \dots \geq \exp(-O(m)) \min_{J \in \J, s \in \X^{J}(m-6)} h^1(\tilde{\X}^{\tilde{J}}_s).\]
Fixing \(J \in \J\) and \(s\), we bound \(h^1(\tilde{\X}^{\tilde{J}}_s)\). By \pref{claim:link-of-a-faces-complex} \(\tilde{\X}^{\tilde{J}}_s \cong \FS[\tilde{J}]_{\cup s}\). We wish to use \pref{prop:improved-generic-lower-bound-for-colors}. Towards this, recall the definition of \(T_q(\S_{\cup s},\tilde{J})\) as in \pref{sec:improved-bound-faces-swap-gen}. \(T_q(\S_{\cup s},\tilde{J})\) is the largest constant such that for every choice of \(J' = \set{c_1',c_2',\dots,c_5'}\) such that \(c_j' \subseteq \tilde{c}_j\) and \(\sum_{j=1}^5|c_j'| = q\) there are indexes \(i_j \in c_j'\), such that the coboundary expansion of \(h^1(\S_{\cup s \dunion w}^{i_1,i_2,\dots,i_5}) \geq T_q(\S_{\cup s},\tilde{J})\) for every face \(w \in \S_{\cup s}[\tilde{J} \setminus \cup J']\). By \pref{prop:improved-generic-lower-bound-for-colors}
    \[h^1(\FS[\tilde{J}]_{\cup s}) \geq \exp(-O(R)) \cdot \prod_{q=1}^R T_q(\S_{\cup s},\tilde{J})\]
where \(R = \sum_{j=1}^5 |\tilde{c}_j|\). As we saw in the weaker version's proof, \(R = O(\frac{\d_1 \log \d_1}{m \log m})\). By the definition,
\[T_q(\S_{\cup s},\tilde{J})\geq \min_{w,I} h^1(S_{\cup s \dunion w}^I)\]
so by \pref{lem:spherical-links}, \(T_q(\S_{\cup s},\tilde{J}) \geq \exp(-O(\log ^2 \d_1))\).

However, by \pref{claim:good-tqs} we can obtain a tighter bound on \(T_q(\S_{\cup s},\tilde{J})\). 
Let \[q_0 = \max_{B, c}|c \cap B|\]
where \(B\) is a \(col(\cup s)\)-bin and \(c \in \tilde{J}\).
By \pref{def:good-colors} \(q_0 = O \left (\frac{\log \d_1}{\log m} \right )\) and by \pref{claim:good-tqs} for every \(q > 10q_0\),
\(T_q(\S_{\cup s},\tilde{J}) = \Omega(1)\). Thus 
\[\FS[\tilde{J}]_{\cup s} \geq \exp(-O(R)) \cdot \exp(-O(\log ^2 \d_1))^{10q_0} \cdot \exp(-O(R-10q_0)).\]
Plugging in \(m=\sqrt{\d_1}\) we have that \(q_0=O(1)\) so this is at least
\(\exp(-O(R + \log^2 \d_1)) = \exp(-O(\sqrt{\d_1}))\).
In conclusion, we have that \(h^1(\X) = \exp(-O(\sqrt{\d_1}))\).
\end{proof}
\subsection{From color restrictions to the entire complex} \label{sec:proof-of-good-colors-to-all-the-complex}
\begin{lemma}\label{lem:colorest}
    Let $J$ be a uniformly chosen set of \(m\) $m=\sqrt {\d_1+1}$ pairwise disjoint colors from $\C$ (namely, $J \sim \FD[\d_1](m-1)$). If \(\Prob[{J \sim \FD[\d_1](m-1)}]{h^1(\X^J)\geq \beta} = p\), then $h^1(\X) = \Omega(\beta p^2)$.
\end{lemma}

The first step in proving \pref{lem:colorest}, is to increase the size of $J$ from $\sqrt{\d_1+1}$ to $(\d_1+1)^2$. We do this because later on in the proof we will need the fact that for any set of colors \(J\) and color \(c\), the fraction of neighbors of \(s \in \X[c_1]\) in \(\X^J(0)\) out of all edges of \(s\) will be almost independent of \(c\). If \(|J| = \sqrt{\d_1+1}\) then some colors \(c\) will intersect all colors \(c' \in J\) non trivially (and therefore \(s\) will have no neighbors in \(\X^J(0)\)), while other \(c\)'s won't intersect any of the colors in \(J\). By increasing the size of the colors to \((\d_1+1)^2\), we make sure that for any color \(c\), the fraction of colors \(c' \in J\) that intersect \(c\) to be a negligible fraction of the colors in \(J\). We use more precise arguments of this form proving smoothness of some distributions in \pref{claim:similar-distributions-edges-connected-with-J} and \pref{claim:similar-distributions-edges-disjoint-from-J}.

This is done using \pref{thm:coboundary-expansion-from-colors} as a black box.
\begin{claim} \label{claim:better-color-guarantee}
    Let \(J \sim \FD[\d_1]((\d_1+1)^2-1)\) of colors. Then with probability \(\frac{p}{2}\), \(\X^J\) is a \(\Omega(\beta p)\)-coboundary expander.
\end{claim}
\begin{proof}[Proof of \pref{claim:better-color-guarantee}]
    The marginal of choosing \(J \sim \FD[\d_1]((\d_1+1)^2-1)\) and then choosing \(J' \subseteq J\) of size \(\sqrt{\d_1+1}\) is just the distribution of choosing \(J' \sim \FD[\d_1](m-1)\). Hence by a Markov argument, 
    \[\Prob[J  \sim {\FD[\d_1]({\d_1}^2-1)}]{ \Prob[J' \subseteq J, |J'|=\sqrt{\d_1}]{h^1(\X^{J'}) \geq \beta} > \frac{p}{2}} \geq \frac{p}{2}.\]
    For every such set of colors \(J \in \FD[\d_1]((\d_1+1)^2-1)\), we claim that \(\X^J\) is a \(\Omega(\beta p)\)-coboundary expander. We prove this via \pref{thm:coboundary-expansion-from-colors}. The colors of \(J\) are mutually disjoint so \(\X^J\) is \((\d_1+1)^2\)-partite. It is a local spectral expander since it is a color restriction of \(\X\) and by the above, there is an \(\Omega(p)\)-fraction of color restrictions \(\X^{J'}\) that are \(\beta\)-coboundary expanders. By \pref{thm:coboundary-expansion-from-colors} \(\X^J\) is a coboundary expander.
\end{proof}

Fix $J$ with $(\d_1+1)^2$ colors. Denote by 
\[E_j = \sett{\set{u,v} \in \X(1)}{\Abs{\X^J(0) \cap \set{u,v}} = j},\qquad j=0,1,2.\]
In words, these are all the edges in \(e \in \X(1)\), such that exactly \(j\) of the vertices \(v \in e\) have that \(col(v) \in J\). 

For a fixed set of colors \(J\) we denote the following two distributions over triangles.
\begin{enumerate}
    \item \(T_{JJn}\) is the distribution over \(uvw \in \X(2)\) given that \(u,v \in \X^J(0)\) and \(w \notin \X^J(0)\).
    \item \(T_{nnJ}\) is the distribution over \(uvw \in \X(2)\) given that \(u,v \notin \X^J(0)\) and \(w \in \X^J(0)\).
\end{enumerate}
For the rest of the proof we fix \(f:\X(1) \to \Gamma\). We need to find some \(g:\X(0) \to \Gamma\) such that \(\beta \dist(f,\coboundary g) \leq \wt(\coboundary f)\).

First we find a set of colors \(J\) that captures the low weight of \(\coboundary f\). More formally, we need the following claim.
\begin{claim} \label{claim:good-colors}
    There exists some set \(J \sim \FD[]((\d_1+1)^2 - 1)\) colors such that:
    \begin{enumerate}
        \item All colors in \(J\) are mutually disjoint and \(\X^J\) is a \(\Omega(\beta p)\)-coboundary expander.
        \item \(wt_{\X^J}(\coboundary f) \leq O(p^{-1}) wt(\coboundary f)\).
        \item \(wt_{T_{JJn}}(\coboundary  f) := \Prob[uvw \sim T_{JJn}]{\coboundary f(uvw) \ne Id} \leq O(p^{-1}) wt(\coboundary f)\).
        \item \(wt_{T_{nnJ}}(\coboundary f) := \Prob[uvw \sim T_{nnJ}]{\coboundary f(uvw) \ne Id} \leq O(p^{-1}) wt(\coboundary f)\).
    \end{enumerate}
\end{claim}
The proof is similar to the proof in \cite[Claim 4.3]{DiksteinD2023cbdry}. We prove it in the end of the subsection.

Fix \(J\) to be as in \pref{claim:good-colors}. We define \(g\) in two steps.
\begin{enumerate}
    \item We begin by defining \(g\) for \(\X^J(0)\). Since \(\X^J\) is an \(\Omega(\beta p)\)-coboundary expander, there exists some \(g\) such that \(\Omega(\beta p) \dist_{\X^J}(f,\coboundary g) \leq wt_{\X^J}(\coboundary f) \leq O(p^{-1}) wt(\coboundary f)\). We take \(g\) to be one such function.
    \item Define \(g\) for \(\X(0) \setminus \X^J(0)\) by taking \(g(v) = \maj_{u \in \X_v^{J}(0)} \set{f(vu)g(u)}\) to be the most popular value (ties broken arbitrarily). We note that $\X_v^{J}(0)\neq \emptyset$ because the color of $v$ can intersect at most $\d_1+1$ colors in $J$ because they are mutually disjoint, and this leaves at least $(\d_1+1)^2-(\d_1+1)$ remaining colors to choose from.
\end{enumerate}

To analyze \(\dist(f,\coboundary g)\) we need the following two smoothness claims regarding the different ways to choose edges with respect to \(\X^J\).
\begin{claim} \label{claim:similar-distributions-edges-connected-with-J}
    Let \(J\) be a set of \((\d_1+1)^2\) colors fixed mutually disjoint colors. Consider the following distributions over \(E_1\):
    \begin{enumerate}
        \item \(D_0\) - The distribution where one samples an edge \(uv \in \X (1)\) conditioned on being in $E_1$.
        \item \(\d_1\) - The distribution where one samples an edge by first sampling a triangle \(uvw \sim T_{nnJ}\) such that \(u,w \notin \X^J(0)\) and \(v \in \X^J(0)\), and then outputting \(uv\).
        \item \(D_2\) - The distribution where one samples an edge by first sampling a triangle \(uvw \sim T_{JJn}\) such that \(v,w \in \X^J(0)\) and \(u \notin \X^J(0)\), and then outputting \(uv\).
    \end{enumerate}
    Then \((D_0,D_2)\) and \((\d_1,D_0)\) are \(\frac{1}{2}\)-smooth.
\end{claim}
In particular, this claim together with \pref{claim:property-of-smooth-dist} implies that for every set \(A \subseteq E_1\) it holds that 
    \begin{equation} \label{eq:claim:similar-distributions-edges-connected-from-J-1}
        \Prob[D_0]{A}\leq 2 \Prob[D_2]{A}.
    \end{equation}
    and
    \begin{equation} \label{eq:claim:similar-distributions-edges-connected-from-J-2}
        \Prob[\d_1]{A}\leq 2 \Prob[D_0]{A}.
    \end{equation}
\begin{claim} \label{claim:similar-distributions-edges-disjoint-from-J}
    Let \(J\) be a set of \((\d_1+1)^2\) mutually disjoint colors. Consider the following distributions over \(E_0\):
    \begin{enumerate}
        \item \(P_0\) - The distribution where one samples an edge \(uv \in \X(1)\) conditioned on \(u,v \notin \X^J(0)\).
        \item \(P_1\) - The distribution where one samples an edge \(uv \in \X(1)\) by first sampling a triangle \(uvw \sim T_{nnJ}\) such that \(u,v \notin \X^J(0)\) and \(w \in \X^J(0)\), and then outputting \(uv\).
    \end{enumerate}
    The \((P_0,P_1)\) are \(\frac{1}{2}\)-smooth.
\end{claim}
In particular this implies that for every set \(A \subseteq E_0\) it holds that 
    \begin{equation} \label{eq:claim:similar-distributions-edges-disjoint-from-J}
        \Prob[P_0]{A}\leq 2 \Prob[P_1]{A}.
    \end{equation}
We prove both of these claims in \pref{app:outstanding-coboundary-expansion-proofs}.

\begin{proof}[Proof of \pref{lem:colorest}]
    We already know that 
    \begin{equation} \label{eq:ind-begining}
        \dist_{E_2}(f,\coboundary g) \leq O(\frac{1}{\beta p^2}) wt(\coboundary f)
    \end{equation}
    since this follows from the coboundary expansion of \(\X^J\) and the fact that \(g\) was chosen to minimize this distance. We show separately that \(\dist_{E_1}(f,\coboundary g) = O(\frac{1}{\beta p^2}) wt(\coboundary f)\), and that \(\dist_{E_0}(f,\coboundary g) = O(\frac{1}{\beta p^2}) wt(\coboundary f)\). Here the distribution over \(E_i\) is the distribution over \(\X(1) \) conditioned on the edge being in \(E_i\). As \(\dist(f,\coboundary g)\) is a convex combination of \(\dist_{E_i}(f,\coboundary g)\), the proposition follows.

    Let us begin with \(\dist_{E_1}(f,\coboundary g) = \Prob[uv \sim D_0]{f(uv) \ne \coboundary g(uv)}\). By \pref{claim:similar-distributions-edges-connected-with-J} we can replace \(D_0\) with \(D_2\) and just show that
    \begin{equation} \label{eq:omega-1-with-d2}
        \Prob[uv \sim D_2]{f(uv) \ne \coboundary g(uv)} \leq O(\frac{1}{\beta p^2}) wt(\coboundary f).
    \end{equation}
    Fix some \(v \in \X(0)\setminus \X^J(0)\), and let \(\varepsilon_v = \Prob[u \in X_v^J(0)]{f(uv)\ne \coboundary g(uv)} = \Prob[u \in X_v^J(0)]{g(v)\ne f(uv)g(u)}\).

    Recall that the underlying graph of \(X_v^J\) is a \(\Theta(1)\)-edge expander. By the edge-expansion and \pref{claim:expander-and-majority}, it holds that 
    \[\Prob[u \in X_v^J(0)]{g(v)\ne f(uv)g(u)} = O(\Prob[uw\in X_v^J(1)]{f(wv)g(w)\ne f(uv)g(u)}).\]
    On the other hand, if \(f(vw)=\coboundary g(vw)\) and \(\coboundary f(vuw)=Id\) then \(f(wv)g(w) =  f(uv)g(u)\) since
    \begin{align}
        f(wv)g(w) &=  f(uv)g(u)  \Leftrightarrow\\
        f(wv)g(w)g(u)^{-1}f(vu) = Id \Leftrightarrow\\
        f(wv)f(uw)f(vu) = Id \Leftrightarrow\\
        \coboundary f(uvw) = Id
    \end{align}
    where the third line follows from \(f(vw)=\coboundary g(vw)\) and last line is true by the assumption that indeed \(\coboundary f(vuw)=Id\). Thus, 
    \begin{align*}
          \Prob[uv \sim D_2]{f(uv) \ne \coboundary g(uv)} &\leq \Ex[v]{\varepsilon_v} \\
          &= O \left(\Ex[v]{\Prob[u \in X_v^J(0)]{g(v)\ne f(uv)g(u)}} \right ) \\
          &= O \left( \Ex[v]{\Prob[uw \sim X_v^J(1)]{\coboundary f(uvw) \ne Id} + \Prob[uw \sim X_v^J(1)]{\coboundary f(uw) \ne g(uw)}} \right )\\
          &= O(wt_{T_{JJn}}(\coboundary f)+\dist_{E_2}(f,\coboundary g)),
    \end{align*}
    
    where the last line follows from the fact that choosing \(v\) and then a triangle \(uw\) as follows amounts to the distribution \(T_{JJn}\). By \eqref{eq:ind-begining} \(\dist_{E_2}(f,\coboundary g) =O(\frac{1}{\beta p^2}) wt(\coboundary f)\) and by \pref{claim:good-colors} \(wt_{T_{JJn}}(\coboundary f) = O(p^{-1})wt(\coboundary f)\) thus 
    \begin{equation} \label{eq:omega-1}
        \dist_{E_1}(f,\coboundary g) =O(\Prob[uv \sim D_2]{f(uv) \ne \coboundary g(uv)}) = O(\frac{1}{\beta p^2})wt(\coboundary f).
    \end{equation}

    We move to show that \(\dist_{E_0}(f,\coboundary g) = \Prob[uv \sim P_0]{f(uv) \ne \coboundary g(uv)} = O(\frac{1}{\beta p^2}) wt(\coboundary f)\). By \pref{claim:similar-distributions-edges-disjoint-from-J} it is enough to show that
    \begin{equation} \label{eq:omega-2-with-P1}
        \Prob[uv \sim P_1]{f(uv) \ne \coboundary g(uv)} \leq O(\frac{1}{\beta p^2})wt(\coboundary f).
    \end{equation}
    \(P_1\) is a marginal of \(T_{nnJ}\) so 
    \[\Prob[uv \sim P_1]{f(uv) \ne \coboundary g(uv)} = \Prob[uvw \sim T_{nnJ}, u,v \notin \X^J(0)]{f(uv) \ne \coboundary g(uv)}.\]
    On the other hand, if \(\coboundary f(uvw) = Id\), \(f(vw)=\coboundary g(vw)\) and \(f(uw)=\coboundary g(uw)\) then
    \begin{align*}
        \coboundary f(uvw) = Id \Leftrightarrow \\
        f(uv)=f(wv)f(uw) \Leftrightarrow \\
        f(uv) = \coboundary g(wv) \coboundary g(uw) \Leftrightarrow \\
        f(uv) = g(u)g(v)^{-1} = \coboundary g(uv).
    \end{align*}
    Thus 
    \begin{multline*}
        \Prob[uv \sim P_1]{f(uv) \ne \coboundary g(uv)} \leq \\
        \Prob[uvw \sim T_{nnJ}]{\coboundary f(uvw) \ne Id} + 2 \Prob[uw \sim \d_1]{f(uw) \ne \coboundary g(uw)} \leq \\
        \Prob[uvw \sim T_{nnJ}]{\coboundary f(uvw) \ne Id} + 4 \Prob[uw \sim D_0]{f(uw) \ne \coboundary g(uw)}  \leq \\
        wt_{T_{nnJ}}(\coboundary f) + 4\dist_{E_1}(f,\coboundary g) = O(\frac{1}{\beta p^2}) wt(\coboundary f).
    \end{multline*}
    Where the third line follows from \pref{claim:similar-distributions-edges-connected-with-J} and the last line follows from \pref{claim:good-colors} and \eqref{eq:omega-1}.
\end{proof}

\begin{proof}[Proof of \pref{claim:better-color-guarantee}]
    Let \(h(uvw)\) be the indicator that \(\coboundary f(uvw) \ne Id\). Then \(\Ex[uvw \sim \X(2)]{\one_{\coboundary f \ne Id}} = \wt(\coboundary f)\). We define three random variables. The first is \(A_0(J) = \Prob[uvw \sim \X^J(2)]{\coboundary f(uvw) \ne Id}= \Ex[uvw \sim \X^J(2)]{h(uvw)}\). Obviously \(\Ex[J]{A(J)} = \wt(\coboundary f)\). The other two random variables will also have the same expectation.
    
    Let \(A_{JJn}(J) = \wt_{T_{JJn}}(\coboundary f)\). Consider the following distribution \(R\):
    \begin{enumerate}
        \item Sample \(J \sim \FD[\d_1](\d_1^2-1)\).
        \item Sample \(t \sim T_{JJn}\) and \(t_2 \sim T_{nnJ}\).
        \item Output \((J,t)\).
    \end{enumerate}
    The marginal over \(t\) is the distribution over triangles in \(\X\).
    
    Its expectation is equal to
    \begin{align}
        \Ex[J]{A_{JJn}(J)}&=\Ex[J]{\Ex[t \sim T_{JJn}]{h(t)}}\\
        &= \Ex[(J,t) \sim R_{JJn}]{h(t)}\\
        &= \Ex[t \sim \X(2)]{h(t)} \\
        &= \wt(\coboundary f).
    \end{align}
    Similarly for \(A_{nnJ}(J) = \wt_{T_{nnJ}}(J)\), \(\Ex[J]{A_{nnJ}(J)} = \wt(\coboundary f)\). Thus
    \[\Ex[J]{A_0(J) + A_{nnJ}(J) + A_{JJn}(J)} \leq 3 \wt(\coboundary f).\]
     The fraction of \(J \sim \FD[\d_1](\d_1^2 - 1)\) such that \(\X^J\) is a \(\Omega (\beta p)\)-coboundary expander is at least \(\frac{p}{2}\). Thus the expectation of \(A_0(J) + A_{nnJ}(J) + A_{JJn}(J)\) over this set so \(J\)'s is at most \(\frac{2}{p} \cdot 3\wt(\coboundary f)\). Thus there exists some fixed \(J\) such that \(\X^{J_0}\) is a \(\Omega (\beta p)\)-coboundary expander, and whose value is less than the expectation, i.e. that
     \(A_0(J) + A_{nnJ}(J) + A_{JJn}(J) \leq \frac{6}{p} \wt(\coboundary f)\). In particular, \(A_0(J),A_{nnJ}(J),A_{JJn}(J) \leq \frac{6}{p} \wt(\coboundary f)\). This \(J\) is the set of colors we need.
\end{proof}

\subsection{Well-spread colors}
We define a well-spread color to have good pseudo-random properties, that is, all indices are roughly equally spaced, and interlaced with one another so that many colors will be isolated. This will facilitate the lower bounds in the next sections

Let \(J \subseteq \mathcal{C}\). Recall that \(\dunion J = \bigcup_{c \in J}c\).

\begin{definition}[Well-spread subsets of colors] \label{def:good-colors}
    Let \(J\) be a set of \(m\) colors in \(\mathcal{C}\). We say that \(J\) is \emph{well-spread} if the following properties hold.
    \begin{enumerate}
        \item Every \(c_1,c_2 \in J\) are disjoint.
        \item For every \(\ell_1,\ell_2 \in (\cup J) \cup \set{0,n}\) it holds that \(\Abs{\ell_1 - \ell_2} \geq \frac{n}{(m(\d_1+1))^3}\).
        \item For every \(J' \subseteq J\) of size \(\Abs{J'} = 5\) and \(\overline{J'} = J \setminus J'\):
        \begin{enumerate}
            \item For every \(\ell \in \cup \overline{J'} \cup \set{0}\) there exists some \(\ell' \in \cup \overline{J'} \cup \set{n}\) such that \(1<\ell'-\ell \leq \frac{100n \log (\d_1+1)}{(\d_1+1) m}\).
            \item For every \(c \in J'\), the number of colors \(i \in c\) that are in \(J'\)-crowded \(\overline{J'}\)-bins is at most \(\frac{100(\d_1+1) \log (\d_1+1)}{m \log m}\).
            \item For every \(c \in J'\) and every \(\overline{J'}\)-bin \(B\), it holds that  \(\abs{B \cap c } \leq 20\frac{\log (\d_1+1)}{\log m}\).
        \end{enumerate}
    \end{enumerate}
We denote by $\J\subset\FD[\d_1]$ the set of well-spread color sets.
\end{definition}

\begin{proposition} \label{prop:prob-of-good-colors-tend-to-one}
    Let \(\d_1, n\) be such that \(n \geq \d_1^5\). Let \(6 \leq m \leq (\d_1+1)\). The probability that \(m\) uniformly chosen colors are \emph{well-spread} tends to \(1\) whenever \(n\geq \d_1^5\) and \(\d_1 \to \infty\).
\end{proposition}
The proof of \pref{prop:prob-of-good-colors-tend-to-one} is just a bunch of elementary probabilistic lower bounds. We prove it in full in \pref{app:outstanding-coboundary-expansion-proofs}.
\subsection{Trickling Down}\label{sec:trick}
\begin{claim}\label{claim:trick}
    Let $J\in \J$ be a well-spread set of $m$ colors, $-1\leq i \leq m-6$,  
\[h^1(\X^J) \geq exp(-O(i))\cdot \min_{s\in \X^J(i)} h^1(\X_s^J). \]
\end{claim}
\begin{proof}
We wish to use \pref{lem:trick-gen} on \(\X^J\). To do so we need to argue that \(\X^J\) is a good enough local spectral expander and that every link \(\X^J_t\) is simply connected for \(|t| \leq i+1\). Local spectral expansion follows from \pref{claim:spherical-building-hdxness} when the subspaces in \(\X\) are are taken with respect to a large enough field.

Therefore we need to argue that every \(\X^J_t\) as above is simply connected. In other words we need to show that it is a \(\beta\)-coboundary expander for some \(\beta > 0\).

Fix a link \(\X^J_t\). Let \(J' = J \setminus col(t)\). By \pref{prop:colored-exponential-decay-bound} it is enough to show coboundary expansion of every \(\S_{\cup t \dunion w}^I\) where \(I \leq J'\) is a set of \(5\)-colors and \(w \in \S_t^{\cup J' \setminus I}\). This gives us some positive bound on coboundary expansion and implies simple connectivity. This holds by \pref{lem:general-case-subspace-complex} proven later in this section.
\end{proof}
We note that even though we only needed to prove simple connectivity of the links \(\X^J_t\) where \(J\) was a set of well spread colors, the requirement that the colors are well is not necessary and one can prove simple connectivity of all links in methods similar to \cite{LubotzkyMM2016}.

\subsection{A tensor decomposition of a generalized link}\label{sec:tensor}
Given $w\in \S$ let $c^* = col(w)
\subset [n]$ and write $c^* = \set{j_1,\ldots,j_r}$. Define $r+1$ bins $B_i = \sett{j\in [n]}{j_i<j<j_{i+1}}$ for $0<i<r$ and also $B_0 = \set{j < j_1}$ and $B_r= \set{j> j_r}$.

We first show that the link of the spherical building $\S$ itself has a product structure.
\begin{claim} For any color $c \subseteq[n]$ and any flag $w\in \S$, 
    $\S_w[c] = \S_w[c^0]\times\cdots\times \S_w[c^r]$, where we let $c^i = c\cap B_i$ for $i=0,\ldots,r$.
\end{claim}
\begin{proof}
    Let $w'\in \S_w[c]$. Since $col(w')=c$ and $w'\dunion w\in \S$ one can clearly write $w' = w'^0\dunion \cdots \dunion w'^r$ where the color of $w'^i$ is $c\cap B_i = c^i$, namely $w'^i\in \S_w[c^i]$. It is also easy to see that the map $w'\leftrightarrow (w'^0,\ldots,w'^r)$ is a bijection.
\end{proof}
Next, let $J = \set{c_1,\ldots,c_\ell}\in \FD[](\ell-1)$ be a set of pairwise disjoint colors $c_j\subseteq [n]$, and assume they are also disjoint from $c^*$. (In our application, we will have $|J|=\ell=5$). We move on to the structure of a generalized link $\tilde{\X}^J_w$, 
\begin{lemma}\label{lem:tensor}
$\tilde{\X}^J_w \cong \otimes_{i=0}^r \tilde{\X}^{J_i}_w$ where $J_i = \set{c_j\cap B_i}_{j=1}^\ell$. 
\end{lemma}
Let us denote \(c_j^i = c_j \cap B_i\). Before proving the lemma we remark that $\tilde{\X}{J_i}_w$ is an $\ell$-partite complex with the $j$-th part being $\tilde{\X}_w[c_j^i]$. So when we define the tensor product of two such complexes, we naturally pair the $j$-th parts with each other, as per \pref{def:partite-tensor}. 
\begin{proof}
    By definition , $\FS[J]_w$ is $\ell$-partite with the $j$-th part consisting of vertices $\FS[J]_w[c_j]$. Viewed as faces of $\S$ these are exactly $\S_w[c_j]$ which equals $\S_w[c_j^0]\times\cdots\times \S_w[c_j^r]$ by the previous claim. We can therefore write each $w'_j\in \tilde{\X}^J_w[c_j]$ uniquely as ${w'}_j = {w'}_j^0\cup\cdots\cup {w'}_j^r$ with ${w'}_j^i\in \S_w[c_j^i]$. We view ${w'}_j^i$ as the projection of ${w'}_j$ to the $i$-th coordinate, and have a bijection between ${w'}_j \leftrightarrow ({w'}_j^0,\ldots,{w'}_j^r)$. This is progress because ${w'}_j$ is a generic vertex of $\tilde{\X}_w^J$ and the RHS is a generic vertex of $\prod_{i=0}^r \tilde{\X}_w^{J_i}$, so we have a bijection between the vertex set of the two complexes.
    
    Fix now two vertices ${w_1,w_2}\in \tilde{\X}^J_w(0)$, and observe that they are connected by an edge iff for all $i$, the $i$-th coordinate is connected by an edge, namely $\set{w_1^i,w_2^i} \in \tilde{\X}^{J_i}_w(1)$. More generally, $\set{w_1,w_2,\ldots,w_\ell}\in \tilde{\X}^J_w$ if and only if for all $i$ and $\set{w^i_0,w^i_1,\ldots,w^i_r}\in \tilde{\X}^{J_i}_w$. One can verify that as measured complexes, both complexes have the same measure on the top level faces, and the lemma is proven. 
\end{proof}

\subsubsection{A refined decomposition}
We now show that we can in fact ignore some of the bins in the decomposition above, simplifying the complex we need to analyze. We continue with the notation of $w\in \S$ and $J= \set{c_1,\ldots,c_\ell}$ where the $c_j$'s are disjoint subsets of $[n]$ as before. We also denote $I = col(w) \subset[n]$, and recall that by assumption $I$ is mutually disjoint from the sets in $J$.
Recall that $c_j^i = c_j\cap B_i$.
\begin{definition}[Crowded Bin]
A bin $B_i$ \emph{crowded} if $J_i = (c_1^i,\ldots,c_\ell^i)$ has more than one coordinate for which $c_j^i\neq\phi$. It is \emph{lonely} if there is exactly one such coordinate, and otherwise it is \emph{empty}. 
\end{definition}
When a bin is lonely or empty the corresponding component in the tensor decomposition becomes trivial,
\begin{claim} \label{claim:lonely-bin-is-one-partite-complex} ~  \begin{enumerate}
        \item Let \(i\) be an index of lonely bin. Then \(\tilde{\X}^{J_i}_w\) is isomorphic to $K_{n_1,1,\ldots,1}$.
        \item Let \(i\) is an index of an empty bin then \(\tilde{\X}^{J_i}_w \cong K_{1,1,\dots,1}\).
    \end{enumerate}
\end{claim}
This allows us to derive a significantly simplified form for  $\tilde{\X}_w^{J}$, 
\begin{lemma} \label{lem:final-decomposition} Let \(R\) be the indices of all crowded bins,
\begin{equation} \label{eq:final decomposition}
    \tilde{\X}^J_{{w}} \cong K \otimes\bigotimes_{i \in R} \tilde{\X}^{J_i}_{{w}}
\end{equation}    
where \(K\) is some complete $\ell$-partite simplicial complex.
\end{lemma}

\begin{proof}[Proof of \pref{claim:lonely-bin-is-one-partite-complex}]
Suppose the \(i\)-th bin is lonely. Then \(J_i = (c_1^i,\ldots,c_\ell^i)\) is a tuple where all but one of the colors \(c_j^i \in J_i\) are empty. Thus all but one of the $\ell$ parts in \(\tilde{\X}^{J_i}_{w}\) have a single vertex (distinct copies of the empty face). 
Let \(c_j^i \in J_i\) be the non-empty set, and set \(n_1 = |\tilde{\X}^{J_i}_{w}[c_j^i]|\). Indeed for every \(w_j \in \tilde{\X}^{J_i}_{w}[c_j^i]\) there is a top-level face containing $w_j$ in the $j$-th part, and all the copies of the empty face in the other coordinates of \(\tilde{\X}^{J_i}_w\). This is \(K_{n_1,1,1,\dots,1}\). 

If the bin is \emph{empty} then all sets in \(J_i\) are empty. Thus  \(\tilde{\X}^{J_i}_{w}\) has a single vertex in every part, and a single top level face. This is \(K_{1,1,\dots,1}\).
\end{proof}

\begin{proof}[Proof of \pref{lem:final-decomposition}]
    By \pref{lem:tensor} we can decompose
    \[\tilde{\X}^{J}_{w} \cong \bigotimes_{i \in R}\tilde{\X}^{J_i}_{w} \otimes \bigotimes_{i \notin R}\tilde{\X}^{J_i}_{w}.\]
    It is enough to show that \(\bigotimes_{i \notin R}\tilde{\X}^{J_i}_{w}\) is a complete partite complex. By \pref{claim:lonely-bin-is-one-partite-complex} every component in this tensor is itself a complete partite complex. The lemma follows from the observation that a partite tensor product of complete partite complexes is again a complete partite complex.
\end{proof}

By this lemma and \pref{claim:triangle-complex} we derive, recalling that $\ell=5$,
\begin{corollary} \label{cor:tensor-bound} There is some constant $\beta=\beta_m>0$ such that 
\begin{equation*}
    h^1(\tilde{\X}_{w}^{J}) \geq \beta\cdot (h^1(\tilde{\X}^{\tilde{J}}_{w}))
\end{equation*}
Where \(\tilde{J} = \set{\tilde{c}_1,\tilde{c}_2,\dots,\tilde{c}_\ell}\) and \(\tilde{c}_j = \sett{i \in c_j}{i \text{ is not in a lonely or empty bin}}\). \(\qed\)
\end{corollary}
This shows we can deduce a bound on the coboundary expansion of \(\tilde{\X}_{w}^{J}\) from a bound on the coboundary expansion of its subcomplex \(\tilde{\X}^{\tilde{J}}_{w}\) that only contains crowded bins.
\subsection{Expansion of colored spherical buildings} \label{sec:spherical-building-color-complexes}
Recall \pref{def:good-colors} of the set $\J$ of well-spread color sets, and let $J=\set{c_1,\ldots,c_m}\in \J$ and recall that we write $J'\leq J$, if $J' = \set{c'_1,\ldots,c'_m}$ where $c'_j\subseteq c_j$. 
\begin{lemma}\label{lem:spherical-links}
Let $I = \set{\set{i_1},\ldots,\set{i_5}}$, and let $s\in \X$ be such that $|s|=m-5$ and there is a well-spread set of colors $J\in \J$ such that $I\cup col(s)\leq J$. Let \(w' \in \S_{\cup s}^{\cup J}\) be such that \(col(w') \cap I = \emptyset\).
Then $h^1(\S_{\cup s \dunion w'}^I)\geq exp(-O(\log ^2(\d_1)))$.
\end{lemma}
\begin{proof}[Proof of \pref{lem:spherical-links}]
    Let us denote by \(w = \cup s \dunion w'\). Fix some \(I\) as above, and let \(I' = \set{i_0 < i_1< i_2<i_3}\) be any four indexes inside \(I\). By \pref{thm:coboundary-expansion-from-colors}, if we show that \(\S_{w}^{I'}\) is a \(\beta\)-coboundary expander, then it holds that \(h^1(\S_{w}^{I'}) \geq \Omega(\beta)\). So from now on we let $I = \set{i_0 < i_1< i_2<i_3}$.

The coboundary expansion $h^1(\S^I_{w})$ depends on $I$ and $w$. We address first the easier ``direct'' cases, and then move to the general case which is gradually reduced to the easier cases, via decomposition steps.

We begin by considering the case where the indices of \(I\) are in more than one \(col(w)\)-bin. Assume for example that \(i_0\) is separate from \(i_1,i_2,i_3\) (the rest of the cases follow from the same argument). In this case, by \pref{claim:subspaces-quotiented-in-the-middle} we will have that \(h^1(\S_{w}^{I}) \geq \Omega(1/diam(\S_{w}^{\set{i_1,i_2,i_3}}))\). If \(i_1,i_2,i_3\) are not in all the same bin, then the diameter is constant (since you can traverse from any subspace of dimension \(i_1\) to any subspace of dimension \(i_3\)). Otherwise, they are in the same bin. 
    Let \([k_0,k_1]\) be the $col(w)$-bin that contains \(i_1,i_2,i_3\). Let \(\tilde{i}_j = i_j - k_0\).
    Let us write down explicitly who are the subspaces in \(\S_{w'}^{i_1,i_2,i_3}\). Let \(v_0,v_1 \in w\) be the subspaces of dimension \(k_0\) and \(k_1\) respectively. Then \(\S_{w}^{\set{i_1,i_2,i_3}}\) contains all spaces that contain \(v_0\) and are contained in \(v_1\) of dimensions \(I\). 
    It is easy to see that \(\S_{w'}^{I}\) is isomorphic to a $3$-partite colored complex whose ambient space is \(\mathbb{F}_q^{k_1-k_0}\), 
    with parts corresponding to dimensions \(\tilde{i}_1,\tilde{i}_2,\tilde{i}_3\). 
    By \pref{claim:diam-of-spherical-building}, we have that \(diam(\S_{w}^{\set{i_1,i_2,i_3}}) = O(\tilde{i}_3/(\tilde{i}_3-\tilde{i}_1))\). We have that \(\tilde{i}_3 \leq k_1-k_0 \leq \frac{100n \log \d_1}{\d_1 m}\), where this inequality is by the well-spreadedness. To state this explicitly, the fact that \(s \in \X(m-6)\) implies that any \(col(\cup s)\)-bin has length \(\leq \frac{100n \log \d_1}{\d_1 m}\) by \(3a\) in \pref{def:good-colors}. The \(col(w)\)-bins cannot be longer since \(\cup s \subseteq w\). Moreover, by item \(2\) in \pref{def:good-colors}, the distance between every two colors in \(J\) is at least \(\frac{n}{(m\d_1)^3}\), so in particular \(\tilde{i}_3-\tilde{i}_1 = i_3-i_1 \geq \frac{n}{(m\d_1)^3}\). Thus we have that \(diam(\S_{w'}^{\set{i_1,i_2,i_3}}) \leq O(\poly(\d_1))\) and in this case we have that \(h^1(\S_{w'}^I) = \Omega(1/\poly(\d_1)) = \exp(-O(\log \d_1))\).
    
    Next, consider the case where all indexes of \(I\) are in the same bin. Similar to before, let us denote by \([k_0,k_1]\) be the \(col(w')\)-bin that contains \(I'\). As before let \(\tilde{I} = \set{\tilde{i}_j}_{j=0}^3\) where \(\tilde{i}_j = i_j - k_0\) and as before \(\S_w^{I}\) is isomorphic to the four-partite complex $\S^{\tilde I}$ whose ambient space is \(\mathbb{F}_q^{k_1-k_0}\), with parts corresponding to dimensions \(\tilde{I}\). 
    
    The proof of this case is the main challenge, and is encapsulated in \pref{lem:general-case-subspace-complex}, which shows that \(h^1(\S_{w}^{\tilde I}) \geq \exp(-O(\log (\frac{\tilde{i}_3}{\tilde{i}_1 - \tilde{i}_0}) \log( \frac{\tilde{i}_3}{\tilde{i_0}})))\). The calculation is similar to the first case: by item \((3a)\) in \pref{def:good-colors}, every bin has size \(O(\frac{n \log \d_1}{\d_1 m})\) so \(\tilde{i}_3 \leq O(\frac{n \log \d_1}{\d_1 m})\). On the other hand, by item \(2\) in \pref{def:good-colors}, the distance between every two colors in \(J\) (and therefore in $I$) is at least \(\frac{n}{(m\d_1)^3}\). In particular this implies that \(\tilde{i}_0 = i_0-k_0 \geq \frac{n}{(m\d_1)^3}\) and so is \(\tilde{i}_1 - \tilde{i}_0 = i_1-i_0 \geq  \frac{n}{(m\d_1)^3}\). Thus \(h^1(\S_w^{\tilde I}) \geq \exp(-O(\log^2(\d_1)))\). To conclude, \(h^1(\S_{w}^I) \geq \exp(-O(\log^2(\d_1)))\).
\end{proof}
\subsubsection{Direct bounds}\label{sec:easy}
In this section we bound $h^1(\S^I_w)$ in two cases which allow a direct bound. We note that our bounds in this section are more general than what we require in \pref{lem:spherical-links}. Henceforth we assume that \(w \in \S\) is any face, including the empty face (unless the claim states that \(w\) is not empty).

The first bound is when not all $i\in I$ are in the same bin. Namely, there is some $v \in w$ such that $i_0 < dim(v) < i_3$. Recall our notation: $\S_w^I = (\S_w)^I$, so in case $v\in w$ and $dim(v)=i\in I$, the complex is still $4$-partite, but there is only one vertex in $\S_v[i]$, namely $v$.
\begin{claim}[Similar to \cite{DiksteinD2023cbdry}] \label{claim:subspaces-quotiented-in-the-middle}
Let \(I = \set{i_0<i_1<i_2<i_3}\), let \(w \in \S\) be such that \(col(w) \cap I = \emptyset\) and let \(v \in w\) be \(i_0 \leq dim(v) \leq i_3\). 
\begin{enumerate}
    \item If \(i_0 \leq dim(v) \leq i_1\) then \(h^1(\S^I_w) \geq \Omega(1/diam(\S_w^{\set{i_1,i_2,i_3}}))\) and if \(i_2 \leq dim(v) \leq i_3\) then \(h^1(\S^I_v) \geq \Omega(1/diam(\S_w^{\set{i_0,i_1,i_2}}))\).
    \item If \(i_1 < dim(v) < i_2\) then \(h^1(\S^I_v) = \Omega(1/d)\) where \(d = \min \set{diam(\S^{\set{i_0,i_1}}_v), diam(\S^{\set{i_2,i_3}}_v)}\).
\end{enumerate}
\end{claim}
This claim is proven in \pref{app:outstanding-coboundary-expansion-proofs} using the cone machinery developed in \pref{sec:cones}. It relies on the following bound for the diameter,
\begin{claim} \label{claim:diam-of-spherical-building}
    Let \(I \subseteq [n], |I| \geq 2\). Then \(diam(\S_w^I) = O(\frac{\max I}{\max I - \min I})\).
\end{claim}

\begin{proof}[Proof of \pref{claim:diam-of-spherical-building}]
    We note that \(diam(\S_w^I) \leq 2+\min_{i_1<i_2 \in I} diam(\S^{\set{i_1,i_2}})\) since for any pair \(v,v' \in \S_w^I(0)\) we can construct a path between \(v,v'\) by first moving from \(v\) to some \(u \in S[i_1]\), from \(v'\) to \(u' \in \S_w[i_2]\) (by using two steps), and constructing a path \(P\) in \(\S_w^{\set{i_1,i_2}}\) from \(u\) to \(u'\). The resulting path \(v \circ P \circ v'\) has length at most \(2+diam(\S_w^{i_1,i_2})\). Hence we prove that \(diam(\S_w^{\set{i_1,i_2}}) = O(\frac{i_2}{i_2-i_1})\) for any \(i_1 < i_2\).

    Fix some \(i_1 < i_2\). Let \(v,v' \in \S_w^{\set{i_1,i_2}}(0)\) and without loss of generality assume that \(v,v' \in S[i_2]\) (we can do this by taking two steps as above. This again adds another constant that in taken into account in the \(O\) notation). Fix \(A = \set{a_1,a_2,\dots a_\ell}\) to be a basis for \(v \cap v'\) (or \(A = \emptyset\) if they intersect trivially). Extend it to a basis for \(v\) \(B = A \dunion B'\) where \(B'=\set{b_1,b_2,\dots,b_\ell}\), and to a basis for \(v'\), \(C=A \dunion C'\) where \(C'=\set{c_1,c_2,\dots,c_\ell}\). By taking two steps in \(\S_w^{\set{i_1,i_2}}\) we can go from \(v\) to a subspace \(v''\) such that \(\dim(v \cap v'') \geq i_1\). In other words, we can take \(i_2-i_1\) vectors from \(B'\) and replace them with vectors in \(C'\) (e.g. if \(i_2-i_1 = m\) we can go from \(v\) to \(v''=span (A \dunion D)\) where \(D=\set{c_1,c_2,\dots,c_m,b_{m+1},b_{m+1},\dots,b_\ell}\). After at most \(\frac{i_2}{i_2-i_1}\) steps we can traverse this way to \(v'\).
\end{proof}
The second direct case is when all $i\in I$ are in the same bin, but they have good gaps between them. In this case, constant coboundary expansion was already proven in \cite{DiksteinD2023cbdry}. We reprove it again in \pref{app:outstanding-coboundary-expansion-proofs} since the proof in the previous paper used separate arguments for abelian and non-abelian groups. Our cone machinery gives a more compact proof of this claim (and a better quantitative bound).
\begin{claim} \label{claim:well-spaced-subspaces}
    Let \(I = \set{i_0<i_1<i_2<i_3}\) such that \(i_1 \geq 2i_0\), \(i_2 \geq i_1+i_0\) and \(i_3 \geq i_2 + 2i_1\). Then \(h^1(S_w^I)\geq \frac{1}{36}\).
\end{claim}
\subsubsection{Decomposing}\label{sec:decomp}
In this section we lower bound $h^1(\S_w^I)$ for an arbitrary set $I$ consisting of $4$ colors.  We gradually decompose this complex into smaller pieces. Each piece is of the form $\S^{I'}_{\set{w \dunion w'}}$ for some $I'$, and some subspace $w' \in \S_w$. Some of the pieces fall into the easy cases as above, and the remainder is then further decomposed until at last we are able to bound everything. In this subsection we suppress \(w\) (and reuse the letter in notation elsewhere) as it does not affect any of the calculations and claims (but we stress that the following \pref{lem:general-case-subspace-complex} applies for \(\S_w\) as well by the exact same proof).
Our main lemma is this,
\begin{lemma} \label{lem:general-case-subspace-complex}
    Let \(I=\set{i_0<i_1<i_2<i_3}\) such that \(i_3 > 21\) and such that \(i_j - i_{j-1} \geq 3\). Then \[h^1(S^I) \geq \exp \left (-O\left ( \log \left (\frac{i_3}{i_1-i_0} \right ) \cdot \log\left (\frac{i_3}{i_1}\right ) \right ) \right ).\]
\end{lemma}

Before proceeding, let us consider the proposition which is a corollary of the color swap lemma, \pref{lem:base-reduction-using-decomposition-subspaces}, proven in \pref{app:color-swap}.

\begin{lemma}\torestate{ \label{lem:base-reduction-using-decomposition-subspaces}
    Let \(X\) be a \(d\)-partite complex. Let \(I = \set{i_0 < i_1 < i_2 < i_3}\), let \(i_j \in I\) and \(i' \notin I\) and let \(I' = \left (I \setminus \set{i_j} \right ) \cup \set{i_j'}\). Assume that the color swap walk between every two sides of \(X\) is a \(\frac{1}{4}\)-spectral expander.
    Then 
    \(h^1(X^{I}) \geq \Omega \left( h^1(X^{I'}) \cdot \min_{v \in {X[i']}} h^1(X_v^{I'}) \right )\).
    }
\end{lemma}
In our case we take \(X = \S\), and the following proposition is immediate.
\begin{proposition} \label{prop:base-reduction-using-decomposition-subspaces}
    Let \(I,i'_j\) and \(I'\) be as above. Let \(\beta = \min_{v\in \S(0),\;dim(v)=i'_j} h^1(\S_v^I)\).
    Let \(\gamma = h^1(\S^{I'})\). Then \(h^1(\S^I) \geq \Omega(\beta \gamma)\).
\end{proposition}

\begin{proof}
    By \pref{claim:spherical-building-hdxness}, \(\S\) is a \(\frac{1}{4}\)-spectral expander for a large enough ambient field. We can conclude by \pref{lem:base-reduction-using-decomposition-subspaces}.
\end{proof}

Now that we have the main tool for this sub section, let us put it to good use.
\begin{claim} \label{claim:low-dim-indexes-2}
    Let \(I = \set{i_0<i_1<i_2<i_3}\) be such that \(i_3 \geq 21i_0 \) and \(i_1 \geq 2i_0\), then \(h^1(\S^I) = \Omega(\frac{\max(i_1-2i_0,i_0)}{i_3-2i_0})\).
\end{claim}

\begin{proof}[Proof of \pref{claim:low-dim-indexes-2}]
If \(i_3 \geq 7i_1\) then we are done by the following sub-claim,
\begin{claim} \label{claim:low-dim-indexes}
    Let \(I = \set{i_0<i_1<i_2<i_3}\) be such that \(i_3 \geq 7i_1 \) and \(i_1 \geq 2i_0\), then \(h^1(\S^I) = \Omega(1)\).
\end{claim}
\begin{proof}[Proof of \pref{claim:low-dim-indexes}]
    If \(i_1+i_0 \leq i_2  \leq 5i_1\) then by \pref{claim:well-spaced-subspaces} we are done. Otherwise, we ``shift'' $i_2$ to \(i_2'= i_1+i_0\). We let \(I'=(I \setminus \set{i_2}) \cup \set{i'_2}\). By \pref{prop:base-reduction-using-decomposition-subspaces} it holds that \(h^1 (\S^I) \geq \Omega(\min_v h^1 (\S^I_v) \cdot h^1 (\S^{I'}))\) (where \(v\) is any space of dimension \(i'_2\)). By \pref{claim:well-spaced-subspaces} \(h^1 (\S^{I'}) = \Omega(1)\), so let us verify that \(h^1(\S^I_v)\) is also constant. Note that \(\dim(v)=i_0+i_1 \geq i_0,i_1\) so by \pref{claim:subspaces-quotiented-in-the-middle} \(h^1(\S^I_v) \geq \Omega(1/D)\) where \(D\) is either the diameter of \(\S^{i_0,i_1}\) or the diameter of \(\S^{i_0,i_1,i_2}\). In both cases this is a constant because \(i_1 \geq 2i_0\) (hence both diameters are \(\leq 4\)). 
\end{proof}

    If \(i_3 < 7i_1\), it follows that \(i_1 \geq 3 i_0\) since \(7i_1 > i_3\) and \(i_3 \geq 21i_0\) (by assumption), so $i_1-2i_0\geq i_0$ and we need to show that \(h^1(\S^I) = \Omega(\frac{i_1-2i_0}{i_3-2i_0})\).
    
    Let \(i'_1 = 2i_0\) and let \(I' = (I \setminus \set{i_1}) \cup \set{i'_1}\). By \pref{prop:base-reduction-using-decomposition-subspaces} we have that \(h(\S^I) \geq \Omega(h^1(\S^{I'}) \cdot h^1(\S_v^I))\) for some \(v\) of dimension \(\dim(v) = 2i_0\). By \pref{claim:low-dim-indexes}, \(h^1(\S^{I'}) = \Omega(1)\). By \pref{claim:subspaces-quotiented-in-the-middle} it holds that \(h^1(\S_v^I) \geq \Omega(1/diam(\S_v^{i_1,i_2,i_3}))\). We note that \(\S_v^{i_1,i_2,i_3} \cong \S^{i_1-2i_0, i_2-2i_0, i_3-2i_0}\) so by \pref{claim:diam-of-spherical-building} \(\Omega(1/diam(\S_v^{i_1,i_2,i_3})) =  \Omega(\frac{i_1-2i_0}{i_3-2i_0})\). The claim follows.
\end{proof}

\bigskip
We describe the main idea of the proof of \pref{lem:general-case-subspace-complex}. Our goal is to recursively decompose $\S^I$, via \pref{prop:base-reduction-using-decomposition-subspaces}. The basic step moves from $I=\set{i_0,i_1,i_2,i_3}$ to $I'=\set{i'_0,i_1,i_2,i_3}$ until $i_0$ becomes small enough to bound $h^1(\S^I)$ directly via \pref{claim:low-dim-indexes-2}. Each basic step must take $i'_0$ large enough to bound $h^1(\S_v^I)$ directly via \pref{claim:low-dim-indexes-2}.

We define a function \(T:\NN \to \NN\), by
\begin{equation}\label{eq:T}
    T(x) = \lceil \max \set{2x-i_1,\frac{21}{20}x - \frac{1}{20}i_3,1} \rceil
\end{equation}
and show that 
\begin{proposition} \label{prop:basic-inequality}
     For every \(I=\set{i_0,i_1,i_2,i_3}\), it holds that
     \begin{equation} \label{eq:base-cob-inequality}
        h^{1}(S^{\set{{i_0,i_1,i_2,i_3}}}) \geq c \cdot h^1(S^{\set{{T(i_0),i_1,i_2,i_3}}}).
    \end{equation}
    where \(c =  2^{-O(\log(\frac{i_3}{i_1}))}\).
\end{proposition}
This \(T\) will decrease \(i_0\), i.e., \(i_0'=T(i_0) < i_0\). Let \(n(I)\) be the minimal number of applications of \(T\) needed such that \(i_0''=T^{n(I)}(i_0)\) satisfies the requirements of \pref{claim:low-dim-indexes-2}, namely that \(2i_0''\leq i_1\) and  \(21i_0''\leq i_3\). Then by an iterated use of \eqref{eq:base-cob-inequality} we get that 
\[h^1(\S^{\set{{i_0,i_1,i_2,i_3}}}) \geq c \cdot h^1(\S^{\set{{T(i_0),i_1,i_2,i_3}}}) \geq c^2 \cdot h^1(\S^{\set{{T^2(i_0),i_1,i_2,i_3}}}) \geq ... \geq c^{n(I)} h^1(\S^{\set{{i_0'',i_1,i_2,i_3}}}) = \Omega(c^{n(I)+1}).\]
Finally, to conclude we show that \(n(I) = O \left( \log \left (\frac{i_3}{i_1-i_0} \right ) \right )\).

Let \(T_1(x) = 2x-i_1\), \(T_3(x) = \frac{21}{20}x - \frac{1}{20}i_3\) so that
\(T(x) = \lceil \max \set{T_1(x),T_3(x),1} \rceil\).
For now, note that \(T(i_0) \leq i_0\) whenever \(1 \leq i \leq i_1-1\) because \(T_1(x) < x\) and \(T_3(x) < x\) whenever \(x < i_1\). We now show a stronger shrinking property for $T$,
\begin{claim} \label{claim:number-of-Ts-needed}
    Let \(m = 2 \log  \left (\frac{i_3}{i_1-i_0 + 1} \right )\), then \(T^m(i_0) =\underbrace{T\circ T \circ \dots \circ T}_{\text{\(m\) times}}(i_0)\leq 1\) and in particular \(n(I) \leq m\).
\end{claim}
The proof is an exercise in calculus, and can be found in \pref{app:outstanding-coboundary-expansion-proofs}.

\begin{proof}[Proof of \pref{prop:basic-inequality}]
    We intend to use \pref{prop:base-reduction-using-decomposition-subspaces}. Let \(I=\set{i_0,i_1,i_2,i_3}\), let \(i'_0=T(i_0)\) and let \(I' = (I \cup \set{i'_0}) \setminus \set{i_0}\). By \pref{prop:base-reduction-using-decomposition-subspaces} it holds that \(h^1(\S^I) \geq \Omega(h^1(\S^{I'}) \cdot h^1(\S_v^I))\) for some \(w\) of dimension \(T(i_0) \geq 2i_0-i_1, \frac{21}{20}i_0 - \frac{1}{20}i_3\).
    To conclude the proposition we need to show that \(h^1(\S_v^I) \geq 2^{-O(\log(\frac{i_3-2i_0}{\max(i_0,i_1-2i_0)}))}\). Indeed, for \(j=0,1,2,3\) let \(\tilde{i}_j=i_j-i'_0\). Observe that \(\S_v^I \cong \S^{\tilde{i}_0,\tilde{i}_1,\tilde{i}_2,\tilde{i}_3}\). We observe that \(\tilde{i}_1 \geq 2\tilde{i}_0\) since unraveling the definitions of \(\tilde{i}_0,\tilde{i}_1\) this is the same as \(i_1-T(i_0) \geq 2i_0 - 2T(i_0)\). Moving things around this equation is equivalent to \(T(i_0) \geq 2i_0 - i_1\) which is a true statement since \(T(i_0)\) is the maximum of \(2i_0 - i_1\) and other expressions.
    
    In addition, \(\tilde{i}_3 \geq 21\tilde{i}_0\). This is equivalent to \(i'_0 =T(i_0)\geq \frac{21}{20}i_3 - \frac{1}{20}i_0\) by the same reasoning as in the previous case. By \pref{claim:low-dim-indexes-2} it holds that \(h^1(\S_v^I) \geq 2^{-O(\log(\frac{i_3-2i_0}{\max(i_0,i_1-2i_0)}))} \geq  2^{-O(\log(\frac{i_3}{i_1}))}\) and the proposition follows.
\end{proof}

\begin{proof}[Proof of \pref{lem:general-case-subspace-complex} (given \pref{claim:number-of-Ts-needed} and \pref{prop:basic-inequality})]
We prove that \(h^{1}(\S^{\set{{i_0,i_1,i_2,i_3}}}) \geq c^{n(I)+1}\) for \(c=  2^{-O(\log(\frac{i_3}{i_1}))}\). We do so by induction on \(n=n(I)\) for all sets of indexes \(I\) such that \(n(I)=n\) at once. The base case for \(n(I)=0\) is by \pref{claim:low-dim-indexes-2}. Assume the claim is true for \(n\), and show for \(n+1\). Let \(I\) be such that \(n(I) = n+1\). By \pref{prop:basic-inequality} it holds that
    \[h^{1}(\S^{\set{{i_0,i_1,i_2,i_3}}}) \geq c \cdot h^1 (\S^{\set{{T(i_0),i_1,i_2,i_3}}}).\]
    It is obvious that \(n(\set{{T(i_0),i_1,i_2,i_3}}) = n\) hence by induction
    \[h^{1}(\S^{\set{{i_0,i_1,i_2,i_3}}}) \geq c^{n+2}.\]
    By \pref{claim:number-of-Ts-needed}, \(n(I) \leq 2 \log  \left (\frac{i_3}{i_1-i_0 + 1} \right )\) and we are done.
\end{proof}

\subsection{An improved bound for large sets of colors}
We prove the following claim,
\begin{claim} \label{claim:good-tqs}
Let \(J  \subseteq \FD[](4)\). Let \(w \in \S\) be such that \(col(w) \cap (\cup J) = \emptyset\). Let 
\[q_0 = \max_{B, c}|c \cap B|\]
where \(B\) is a \(col(w)\)-bin and \(c \in J\).
Then for all \(q > 10q_0\), \(T_q(\S_{w},J) = \Omega(1)\).
\end{claim}
Note that when \(J\) is well spread, and \(w= \cup s\) for \(s \in \Z^J(m-6)\), then \(q_0 = O(1)\).

\begin{proof}
    To bound \(T_q(\S_{w},J)\) we need to show that there for every \(J' = \set{c_1',c_2',\dots,c_5'} \in \mathcal{J}_q\) and for all \(w' \in \S_{w}\) whose color is disjoint from \(I\) we can find \(5\) colors \(I = \set{i_1,i_2,i_3,i_4,i_5}\) such that \(i_j \in c_j'\) and such that \(h^1(\S_{w \dunion w'}^I) =\Omega(1)\).
    
    Fix \(q > 10q_0\) and \(J' = \set{c_1',c_2',\dots,c_5'} \in \mathcal{J}_q\). Let us prove that there exists some set \(I = \set{i_1,i_2,\dots,i_5}\) as above with the property that there exists a \(col(s)\)-bin \(B\) with \(|B \cap I| = 1\). The reason we use such a set is because in this case \(\S_{w \dunion w'}^I\) is a complex as in \pref{cor:complex-with-one-free-side}. That is, if \(B \cap I = \set{i_5}\), then for any \(w'' \in \S_{w \dunion w'}[\set{i_1,i_2,i_3,i_4}]\) and any \(v \in \S_{\cup s \dunion w}[i_5]\), \(w'' \dunion \set{v} \in \S_{w \dunion w'}^I\) (\(i_5\) sits in a different bin so the space \(v\) which is compatible with \(w \dunion w'\) should also be compatible with the flag \(w''\) that is contained in different bins). The walk between vertices and triangles in this complex is a constant expander (by \pref{claim:spherical-building-hdxness} and \pref{thm:color-swap-walk-exp}) so by \pref{cor:complex-with-one-free-side} this implies that \(h^1(\S_{w \dunion w'}^I) = \Omega(1)\).
    
    By assumption \(q > 10q_0\) so there exists a color \(c_j'\) such that \(|c_j'| > 2q_0\). Without loss of generality let us assume this is \(c_5'\). We begin by choosing \(I' = \set{i_1,i_2,i_3,i_4}\) arbitrarily. If there are three bins \(B_1,B_2,B_3\) so that \(B_i \cap I' \ne \emptyset\), then at least two bins have that \(|B_i \cap I'|=1\). In this case no matter how we complete \(I'\) to \(I = I' \dunion \set{i_5}\), there will still be a bin \(B_i\) such that \(|B_i \cap I'|=1\). There are at most two bins \(B_1,B_2\) with \(B_i \cap I' \ne \emptyset\). By the definition of \(q_0\) \(|c_5 \cap B_i| \leq q_0\) but \(|c_5| > 2q_0\) so there exists \(i_5 \in c_5 \setminus (B_1 \cup B_2)\) which we can choose.
\end{proof}

\printbibliography
\appendix
\section{Outstanding Proofs} \label{app:outstanding-coboundary-expansion-proofs}
\subsection{Proofs for Claims in \pref{sec:preliminaries}}
\begin{proof}[Proof of \pref{claim:convex-combination-expanders}]
    Let \(H' = (V,E \setminus E')\) be the complement subgraph. It is easy to verify that the stationary distribution of \(H'\) is also equal to that of \(G\). Let \(\alpha = \prob{H}\). Then it is easy to verify that
    \begin{enumerate}
        \item \(A_G = \alpha A_H + (1-\alpha)A_{H'}\) where \(A_G,A_H,A_{H'}\) are the respective adjacency operators.
        \item If \(f\) is perpendicular to the largest eigenvector of \(A_G\), then it is perpendicular to the largest eigenvector of \(A_H\) and \(A_{H'}\). This follows from the fact that \(A_G,A_H,A_{H'}\) are all self adjoint with respect to the same inner product, and they all share the same first eigenvector which is the constant function.
    \end{enumerate}
    Let \(f\) be of norm \(1\) such that \(A_G f = \lambda(G) f\). Then
    \[\lambda(G)  = \iprod{A_G f,f} = \alpha \iprod{A_Hf ,f} + (1-\alpha) \iprod{A_{H'}f ,f}.\]
    The rightmost term is upper bounded by \(\iprod{A_{H'}f ,f} \leq \iprod{f,f}=1\). By \(\lambda\) expansion of \(A_H\) the leftmost term is bounded by \(\iprod{A_Hf ,f} \leq \lambda \iprod{f,f} =\lambda\). Thus \(\lambda(G) \leq \alpha\lambda + (1-\alpha) = 1-\alpha(1-\lambda)\).
\end{proof}

\begin{proof}[Proof of \pref{claim:expander-and-majority}]
    \begin{align*}
        \varepsilon &\geq \Prob[uv \in E]{\exists i \; u \in S_i \ve v \notin S_i} \\
        &= \sum_{i=1}^m \Prob[uv \in E]{u \in S_i, v \notin S_i} \\
        &\geq \eta \sum_{i=1}^m \prob{S_i}(1-\prob{S_i}) \\
        &\geq \eta (1 - \sum_{i=1}^m \prob{S_i}^2).
    \end{align*}
    Where the second inequality is by edge expansion. Moving things around we have
    \[\sum_{i=1}^m \prob{S_i}^2 \geq 1-\frac{\varepsilon}{\eta}.\]
    We can bound \(\sum_{i=1}^m \prob{S_i}^2 \leq \max_i \prob{S_i} \cdot \sum_{i=1}^m \prob{S_i} \leq \max_i \prob{S_i}\) and the claim is proven.
\end{proof}

\begin{proof}[Proof of \pref{claim:partite-walk-is-a-const-spectral-expander}]
    Sampling an edge in this walk can be done by first sampling \(s \in X(j)\) and then sampling two independent vertices \(v_1,v_2 \in X_s(0)\). Thus the probability that \(v_1,v_2\) are in the same part \(X_s[i]\) is at most \(\frac{1}{d-j-1} \leq \frac{1}{2}\). Consider the subgraph \(H\) of \(G\) that contains all edges that go from vertices of different color. It is easy to see that the stationary distribution of \(H\) is the same as the stationary distribution of \(G\), since for every face \(s \in X(j)\), the marginal distributions of \(v_1,v_2 \in X_s\) are the same for sampling them independently, or sampling them conditioned on \(v_1,v_2\) having different colors. 
    
    By the aforementioned \(\prob{e \in H} \geq \frac{1}{2}\) thus by \pref{claim:convex-combination-expanders}, the proof will follow if we show that \(H\) is a \(\max(\lambda^2,\frac{1}{d-1})\)-spectral expander. 

    Notice that \(col:X(0) \to [d]\) is a graph homomorphism from \(G\) to the complete graph on \(d\) vertices. It is well known that this is a \(\frac{1}{d-1}\)-spectral expander, thus it is enough to show that the bipartite graph induced by \(L=X[i],R=X[i']\) is a \(\lambda^2\)-spectral expander. For any \(i\) and \(i'\), the bipartite adjacency operator of this graph is the following convex combination \[A = \frac{1}{\binom{d-2}{j}}\sum_{J \subseteq [d] \setminus \set{i,i'}} S_{\set{i},J} S_{\set{i'},J}^{*}\]
    where \(S_{\set{i},J}\) is the bipartite adjacency operator for the \(\set{i}\) vs \(J\) colored swap walk. These are all \(\lambda^2\)-spectral expanders so the claim follows.
\end{proof}

\begin{proof}[Proof of \pref{claim:cone-is-coboundary-expander}]
    We define \(g:X^*(0) \to \Gamma\) to be \(g(v_*)=Id\) and \(g(u) = f(v_* u)\). We note that \(f(v_* u) = g(u)g(v_*)^{-1}=\coboundary g(v_* u)\) by definition. For other edges,
    \(f(uw) = g(w)g(u)^{-1} \Leftrightarrow \coboundary f(v_* u w) = e\).
    Thus 
    \[\dist(f, \coboundary g) = \frac{k-1}{k+1} \Prob[uw \in X(1)]{f(uw) \ne \coboundary g(uw)} = \frac{k-1}{k+1}\Prob[uw \in X(1)]{\coboundary f(v_* uw) \ne 0}.\]
    We notice that 
    \[\Prob[uw \in X(1)]{\coboundary f(v_* uw) \ne 0} \leq 3 \Prob[uwx \in X(2)]{\coboundary f(uwx) \ne 0}.\]
    This is because if \(\coboundary f(uwx) \ne 0\) then at least one of the three triangles \(v_* uw, v_* ux, v_* wx\) is also not satisfied. In particular,
    \[\dist(f, \coboundary g) \leq 3\frac{k-1}{k+1} wt(\coboundary f).\]
\end{proof}

\begin{proof}[Proof of \pref{claim:triangle-complex}]
    We prove this by induction on \(\ell\). When \(\ell=0\) then \(Y=X \otimes K_{1,1,...,1} \cong X\) and the claim is clear. 

    Assume the claim holds for \(\ell\) and prove for \(\ell+1\). Without loss of generality \(n_0 > 1\). It is enough to show that if \(h^1(X) \geq \beta\), then \(h^1(X') \geq \Omega(\beta)\) where \(X'\) is the following simplicial complex. Its vertices are \(X'[0] = X(0) \times [n_0]\) and \(X'[i]=X(i)\) for \(i\geq 1\). The top level faces, \(X'(k)\), are
    \[X'(k) = \sett{\set{(v_0,i),v_1,v_2,...,v_k}}{\set{v_0,v_1,...,v_k} \in X(k)}\]
    We sample a face by first sampling \(i \in [n_0]\), and then independently sampling \(\set{v_0,v_1,...,v_4} \in X(4)\).

    The reason this is enough is because
    \[X \otimes K_{n_0,n_1,\dots,n_k} \cong (X \otimes K_{n_0,1,1,\dots,1}) \otimes K_{1,n_1,n_2,\dots,n_k} \cong X' \otimes K_{1,n_1,n_2,\dots,n_k},\]
    so once we know that \(h^1(X') \geq \Omega(\beta)\) we can continue by induction. 

    Indeed we intend to use \pref{thm:decomposition-to-coboundary-expanders}. Let us use the following GK-decomposition \((\mathcal{Y},C,\nu,\pi)\).
    \begin{enumerate}
        \item \(\mathcal{Y} = \set{Y_i}\) where \(Y_i\) complex induced by \((X(0) \times \set{i})\cup (X'(0) \setminus X'[0])\).
        \item \(\nu\) is given by sampling a top-level face \(s=\set{(v_0,i_0),v_1,\dots,v_k} \in X'(k)\), and outputting \((Y_{i_0},t)\) where \(t \subseteq s\) is chosen uniformly at random.
        \item The agreement complex \(C\) is has \(C(0) = [n_0]\) (technically it is \(\set{Y_i}_{i\in [n_0]}\) but we identify the vertices with \([n_0]\) to shorten notation). For every two distinct \(i,j\) and \(v \in Y_i \cap Y_{j}\) there is a labeled edge \(\set{i,j}_v\).
        \item Note that a vertex appears in the intersection if and only if \(col(v) \ne 0\). Denote by \(I'=[k] \setminus \set{0}\) and by \(A = X^{I'}(0)\) the set of vertices that appear on a label.
        \item The distribution \(\pi\) is given by choosing \((u,v,w) \in \dir{X^{I'}}(2)\), independently choosing three distinct \(i,j,k\) (uniformly at random) and outputting \(\set{i,j,k}_{u,v,w}\).
    \end{enumerate}
    Towards using \pref{thm:decomposition-to-coboundary-expanders}, let us observe the following properties of this decomposition.
    \begin{enumerate}
        \item The distribution \(\nu|_{Y_i}\) is the same distribution as \(X\)'s triangle distribution (identifying every \((v,i) \in X'[0]\) with \(v \in X[0]\)). It follows that every \(Y_i\) is a coboundary expander with \(h^1(Y_i) \geq \beta\).
        \item For every \(v \in A\) the local graph \(C^v\) is the complete graph, which is a constant edge expander.
        \item We also note that the set
        \[A_Y = \sett{(v,Y_i)}{v \in A, v \in Y_i} = X^{I'}(0) \times \set{Y_i}.\]
        \item Let us verify that all the smoothness relations hold.
        \begin{enumerate}
            \item \(\nu_2=\mu_{2,X'}\), and in particular \((\nu_2,\mu_{2,X'})\) are \(1\)-smooth.
            \item The same holds for the pair \((\mu_{1,X'},\nu_1)\).
            \item We need to show that \((\nu_{0,y},\pi_{0,y})\) are \((A_Y,\alpha)\)-smooth for a constant \(\alpha\). For every \((v,Y_i) \in A_Y\) it holds that
            \[\Prob[\pi_{0,y}]{(v,Y_i)} = \frac{1}{n_0} \Prob[X^{I'}(0)]{v} = \frac{k}{k-1} \frac{1}{n_0} \Prob[X']{v} =\frac{k}{k-1} \Prob[\nu_{0,y}]{(v,Y_i)}.\]
            In particular \((\nu_{0,y},\pi_{0,y})\) are \((A_Y,1)\)-smooth.
            \item Similarly \(\pi_2\) chooses a random triangle given that it doesn't intersect \(X'[0]\), so \((\pi_2,\mu_2)\) are \(\frac{2}{5}\)-smooth.
            \item Similarly also to item \((c)\), \((\pi_{1,y},\nu_{1,y})\) are \(\frac{3}{5}\)-smooth.
        \end{enumerate}
    \end{enumerate} 

    To conclude we must show that \(C\) is an \(\gamma = \Omega(1)\)-coboundary expander. we notice that \(C\) is a blow-up of the complete complex (all possible (unlabeled) triangles appear with uniform distribution). The complete complex is a \(\beta = 1\)-coboundary expander \cite{LubotzkyMM2016}. We therefore need to verify that the local graphs of every \(i,j\) are edge expanders. If this holds then by \pref{lem:hdx-blow-up} \(C\) is a \(\Omega(1)\)-coboundary expander.
    
     Indeed for every \(\set{i,j}\), the agreement graph is just two steps in the colored swap walk of \(X^{I'}\) between vertices and edges. While in partite complexes this walk doesn't have an optimal spectral gap, it has a constant gap when \(X\) is a good enough \(1\)-sided spectral expander by \pref{claim:partite-walk-is-a-const-spectral-expander}.
     
     Hence by \pref{lem:hdx-blow-up} \(C\) is a \(\Omega(1)\)-coboundary expander, and by \pref{thm:decomposition-to-coboundary-expanders} we conclude that \(X'\) is a \(\Omega(\beta)\)-coboundary expander (Since besides \(\beta\), the rest of the parameters \(\alpha,\gamma,\eta\) is the theorem are constants).
\end{proof}

\begin{proof}[Proof of \pref{claim:coboundary-expansion-of-complex-as-good-as-its-partitification}]
Let \(X\) be such that \(h^1(X^{\dagger_\ell}) = \beta\) and let \(\Gamma\) be any group. Let \(f \in C^1(X,\Gamma)\). We will show that \(\dist(f, B^1(X,\Gamma) \leq O(\beta)\wt(\coboundary f)\).

Let \(\tilde{f} \in C^1(X^{\dagger_\ell},\Gamma)\) be \(\tilde{f}((v,i),(u,j)) = f(vu)\) for every \((v,i) \in X^{\dagger_\ell}[j]\) and \((u,j) \in X^{\dagger_\ell}[j]\). By coboundary expansion there exists some \(\tilde{g} \in C^0(X^{\dagger_\ell})\) such that \
\begin{equation} \label{eq:cob-exp-guarantee-of-tensor-product}
    \dist(\tilde{f},\coboundary \tilde{g}) \leq \beta^{-1} \wt(\coboundary \tilde{f}) = \beta^{-1} \wt(\coboundary f).
\end{equation}
Let \(g:X(0) \to \Gamma\) be \(g(u) = maj_{(u,i) \in X^{\dagger_\ell}(0)}\tilde{g}(u,i)\). Then
\begin{align}
    \dist(f,\coboundary g) &= \Prob[uv \in X(1)]{f(uv) \ne \coboundary g(uv)} \\
    &\leq \Prob[uv \in X(1), i \ne j \in {[\ell]}]{f(uv) \ne \coboundary \tilde{g}((u,i), (v,j))} + \Prob[uv \in X(1)]{\coboundary \tilde{g}((u,i), (v,j)) \ne \coboundary g(uv)} \\
    &\leq \Prob[(u,i),(v,j) \in X^{\dagger_\ell}(1)]{\tilde{f}((u,i),(v,j)) \ne \coboundary \tilde{g}((u,i),(v,j))} + 2\Prob[v \in X(0), i \in {[\ell]}]{g(v) \ne \tilde{g}(v,i)}\\
    &= \dist(\tilde{f},\coboundary \tilde{g}) + 2\Prob[v \in X(0), i \in {[\ell]}]{g(v,i) \ne \tilde{g}(v,i)}.
\end{align}
The first inequality is the triangle inequality. The second inequality is because \(f(uv) = \tilde{f}((u,i),(v,j))\) and because if \((\tilde{g}(u,i)=g(u)\) and \(\tilde{g}(v,j) = g(v)\) then \(\coboundary \tilde{g}((u,i), (v,j)) =\coboundary g(uv)\).
If  we show that\(\Prob[v \in X(0), i \in {[\ell]}]{g(v) \ne \tilde{g}(v,i)} \leq 2\dist(\tilde{f},\coboundary \tilde{g})\) then the claim will follow from \pref{eq:cob-exp-guarantee-of-tensor-product}.

Indeed, the probability of disagreeing with the majority is upper bounded by the probability that \(g\) assigns different values to two random \((v,i_1),(v,i_2)\), i.e.
    \[\Prob[v \in X(0), i \in {[\ell]}]{g(v) \ne \tilde{g}(v,i)} \leq \Prob[v \in X(0), i_1,i_2 \in {[\ell]}]{\tilde{g}(v,i_1) \ne \tilde{g}(v,i_2)}.\]
If \(\tilde{g}(v,i_1) \ne \tilde{g}(v,i_2)\), then for every \((u,j)\) such that \(vu \in X(1)\), \(\coboundary \tilde{g} ((v,i_1)(u,j)) \ne \coboundary \tilde{g} ((v,i_2)(u,j))\). In particular, this means that either \(f((v,i_1),(u,j) \ne \coboundary \tilde{g} ((v,i_1)(u,j))\) or \(f((v,i_2),(u,j) \ne \coboundary \tilde{g} ((v,i_2)(u,j))\). Thus
    \[\Prob[v \in X(0), i_1,i_2 \in {[\ell]}]{\tilde{g}(v,i_1) \ne \tilde{g}(v,i_2)} \leq 2\Prob[vu \in X(1), i,j \in {[\ell]}]{\coboundary \tilde{g}((v,i),(u,j)) \ne \tilde{f}((v,i),(u,j))} \leq 2\dist(\tilde{f},\tilde{g}).\]
We are done.
\end{proof}

\begin{proof}[Proof of \pref{claim:coboundary-expansion-of-partitification-as-good-as-original}]
Let \(X\) a simplicial complex. We use the following GK-decomposition \((\mathcal{Y},\A,\nu,\pi)\).
\begin{enumerate}
    \item The subcomplexes \(\mathcal{Y} = \sett{Y_v}{v \in X(0)}\) where \(Y_v\) is the complex induced by \(\set{(v,i)} \dunion \sett{(u,j)}{u \in X_v(0)}\).
    \item \(\A\) will be a blow up of \(X\). Edges \(\set{u_1,u_2}_{(v,j)} \in \A(1)\) if \(v \in X_{u_1 u_2}(0)\) and triangles are all \(\set{u_1,u_2,u_3}_{(u_4,j_4),(u_5,j_5),(u_6,j_6)} \in \A(2)\) such that \(\set{u_1,u_2,u_3,u_4,u_5,u_6} \in X(5)\) and \(j_4,j_5,j_6\) are distinct.
    \item \(\pi\) is given by:
    \begin{enumerate}
        \item Sampling \(s \in X^{\dagger_\ell}(5)\).
        \item Sampling two triangles independently \(t_1,t_2 \subseteq s\). \(t_1 = \set{(u_1,i_1),(u_2,i_2),(u_3,i_3)}, t_2 = \set{(u_4,i_4),(u_5,i_5),(u_6,i_6)}\).
        \item Randomly ordering the vertices of \(t_1,t_2\) and outputting \(\set{u_1,u_2,u_3}_{(u_4,j_4),(u_5,j_5),(u_6,j_6)}\).
    \end{enumerate}    
    \item \(\nu\) is given by the marginal of \(\pi\), that is, given \(\set{u_1,u_2,u_3}_{(u_4,j_4),(u_5,j_5),(u_6,j_6)} \sim \pi\) we choose a \(v \in \set{u_1,u_2,u_3}\) uniformly at random and output \((\set{(u_4,j_4),(u_5,j_5),(u_6,j_6)}, Y_v)\).
\end{enumerate}

Let us show that we can apply \pref{thm:decomposition-to-coboundary-expanders} to this decomposition.
\begin{enumerate}
        \item Recall the definition of \(Y^*\) for some complex \(Y\) as in \pref{def:cone-of-a-complex}. The distribution \(\nu|_{Y_v}\) gives that \(Y_v \cong ((X_v)^{*})^{\dagger_\ell}\) since every face in \(Y_v\) is connected to some \((v,i)\).
        We will show that this implies that every \(Y_v\) is a constant coboundary expander below since this requires yet another GK-decomposition.
        Let us deffer the following claim for later.
        \begin{claim} \label{claim:cone-plus-tensor-with-complete-complex}
            Let \(Z = Y^*\) for some complex \(Y\). Then \(h^1(Z^{\dagger_\ell}) =\Omega(1)\).
        \end{claim}
        Given this claim \(h^1(Y_v) = \Omega(1)\).
        \item Every \((v,i)\) appears in an intersection of two sub-complexes Therefore \(A = X^{\dagger_\ell}(0)\). 
        \item The local graph \(\A^{(v,i)}\) is (isomorphic to) \((\tilde{X}_v)^*\) where \(\tilde{X}\) is the skeleton of \(5\)-skeleton of \(X\). This is because given \((v,i)\), \((\set{Y_{u_2},Y_{u_3}},(v,i)) \sim \pi_{1,y}\) by choosing some \(s \ni (v,i)\), and in this case, \(Y_{u_2},Y_{u_3}\) are either chosen such that \(u_2 u_3 \in X_v(1)\) (with probability \(\frac{2}{3}\)) or such that one of \(u_2,u_3\) is equal to \(v\) and the other is in \(X_v(0)\). This is a constant expander whenever \(X\) is a \(\lambda\)-local spectral expander.
        \item Let us verify that all the smoothness relations hold.
        \begin{enumerate}
            \item \(\nu_2=\pi_2 = \mu_{2,X^{\dagger_\ell}}\), and in particular \((\nu_2,\mu_{2,X}), (\pi_2,\mu_{2,X})\) are \(1\)-smooth.
            \item As \(\nu\) is a marginal of \(\pi\), \(\nu_{0,y}=\pi_{0,y}, \nu_{1,y}=\pi_{1,y}\) and the smoothness relations hold as well for these two.
        \end{enumerate}
        \item \(\A\) is a blow-up of \(X\). If we show that for every \(\set{u_1,u_2} \in X(1)\), the label graph is a constant edge-expander, it will follow that \(h^1(\A) =\Omega(h^1(X))\). Fix \(\set{u_1,u_2} \in X(1)\). The vertices of the label graph for \(\set{u_1,u_2}\) are all \(V = \sett{(w,k)}{w \in X_{vu}(0)} \dunion (\set{u_1,u_2} \times [\ell])\). Consider a set \(S \subseteq V\) such that \(S \ne \emptyset,V\). We need to show that \(\prob{E(S,V \setminus S)} \geq c \prob{S} \prob{V \setminus S}\) for some constant \(c\) (here \(E(S,V \setminus S)\) is the probability of sampling a directed edge \((((v_1,j_1),(v_2,j_2))\) such that \((v_1,j_1) \in S\) and \((v_2,j_2) \notin S\)).
        
        We note that for every \(v \ne u_1\), there is a constant probability \(c >0\) to traverse from \((v,k)\) to some \((u_1,k')\). A similar claim holds for \(u_2\). This is because in the label graph we traverse from \((v,k)\) to any triple \((u_3,(u_5,j_5),(u_6,j_6))\) such that \(\set{u_1,u_2,u_3}_{(v,k),(u_5,j_5),(j_6,j_6)}\) and then take a second (independent) step from the  \((u_3,(u_5,j_5),(u_6,j_6))\) back to \((v',k')\) by re completing it it a triangle \(\set{u_1,u_2,u_3}_{(v',k'),(u_5,j_5),(u_6,j_6)}\). If \(u_5, u_6 \ne u_1\) (which happens with constant probability) then the probability that the completing label \((v',k')\) is such that \(v'=u_1\) is a constant from the definition of the distribution.

        Thus for every \(j \in [\ell]\), \(\prob{(u_1,j)} \geq c'\) for some constant \(c'\) (and the same is true for \(u_2\)). Thus fix some subset \(S\) and assume without loss of generality that \((u_1,1) \notin S\) (or otherwise consider the complement):
        \begin{enumerate}
            \item If \(\prob{S \setminus \sett{(u_1,j)}{j \in {[\ell]}}} \geq c' \prob{S}\) then 
            \begin{align*}
                \prob{E(S,V \setminus S)} &\geq \Prob[((v,k),(v',k'))]{(v,k) \in S, (v',k') = (u_1,1)} \\
                &\geq c' \Prob[((v,k),(v',k'))]{(v,k) \in S \setminus \sett{(u_1,j)}{j \in {[\ell]}}} \cdot c \\
                &\geq cc' \prob{S} \\
                &\geq cc' \prob{S}\prob{V \setminus S}.
            \end{align*}
            \item Otherwise \((u_2,j) \notin S\) for any \(j \in [\ell]\) since \(\prob{(u_2,j)} > \prob{S \setminus \sett{(u_1,j)}{j \in {[\ell]}}}\). In this case similarly
            \begin{align*}
                \prob{E(S,V \setminus S)} &\geq \Prob[((v,k),(v',k'))]{(v,k) \in S, v' = u_2} \\
                &\geq c\prob[((v,k),(v',k'))]{S}\\
                &\geq c \prob{S} \prob{V \setminus S}.
            \end{align*}
        \end{enumerate}
    \end{enumerate}

By \pref{thm:decomposition-to-coboundary-expanders} \(h(X^{\dagger_\ell}) = \Omega(h^1(X))\).
\end{proof}

\begin{proof}[Proof of \pref{claim:cone-plus-tensor-with-complete-complex}]
Fix \(Y\) and let \(Z = Y^*\). Denote the vertex that participates in all faces by \(v_*\) (if there is more than one we choose one arbitrarily). We use the following GK decomposition \((\mathcal{Z} = \set{Z_i}_{i=1}^\ell,\A,\nu,\pi)\).
\begin{enumerate}
    \item \(Z_i\) is the complex induced by \(\set{(v_*,i)} \dunion \sett{(u,j)}{j \ne i}\).
    \item \(\A\) is a blow up of the complete complex over \(\ell\) vertices. for every \(i_1,i_2 \in [\ell]\) the edges are all \(\set{i_1,i_2}_{(u,i_3)}\) such that \(i_3 \notin \set{i_1,i_2}\) and \(u \ne v_*\). The triangles are all \(\set{i_1,i_2,i_3}_{(v,i_4),(u,i_5),(w,i_6)}\) such that \(i_1,i_2,\dots,i_6\) are all distinct and \(\set{u,v,w} \in Y(2)\) (and in particular none of them are equal to \(v_*\)).
    \item The distribution \(\nu\) is given by choosing a top level face in \(s \in Z^{\dagger_\ell}\), and then taking \((Y_i,t)\) such that \((v_*,i) \in s\) and \(t \subseteq s\) is a triangle chosen uniformly at random. We note that the marginal of this distribution is equal to the distributions over triangles in \(Z^{\dagger_\ell}\).
    \item The distribution \(\pi\) is given by:
    \begin{enumerate}
        \item Sampling a triangle \(\set{u,v,w} \in Y\).
        \item Independently sampling a tuple of six distinct indexes \((i_1,i_2,\dots,i_6)\) from \([\ell]\). 
        \item Outputting \(\set{i_1,i_2,i_3}_{(u,i_4),(v,i_5),(w,i_6)}\).  
    \end{enumerate}
    We note the marginal of this distribution over the labels is the distribution over \(Y^{\dagger_\ell}(2)\), or equivalently, the distribution over triangles in \(Z^{\dagger_\ell}\) conditioned on not containing any \((v_*,j)\).
\end{enumerate}

\begin{enumerate}
        \item The distribution \(\nu|_{Z_i}\) gives that \(Z_i \cong( X^{\dagger_{\ell-1}})^{*}\). This is because whenever a triangle is chosen in \(Z_i\), it is because a top-level face is chosen that contains \((v_*,i)\). Thus by \pref{claim:cone-is-coboundary-expander} \(h^1(Y_i) = \Omega(1)\).
        \item The set of vertices that appear in an intersection between two \(Z_i\)'s are
        \[A = \sett{(u,i)}{i \in [\ell],u \in Y(0)}.\]
        That is, all vertices such that the left coordinate is \(\ne v_*\).
        
        The local graph \(\A^v\) is the complete graph over \([\ell-1]\) vertices, which is a constant edge expander.
        \item We also denote by \(A_Z = \sett{((u,j),Z_i)}{(u,j) \in A, i \ne j}\).
        \item Let us verify that the smoothness relations in \pref{thm:decomposition-to-coboundary-expanders} hold. 
        \begin{enumerate}
            \item \(\nu_2=\mu_{2,Z^{\dagger_\ell}}\) so they are \(1\)-smooth. The same holds for \((\mu_{2,Z^{\dagger_\ell}},\nu_1)\).
            \item \(\pi_2\) is the distribution over \(Z^{\dagger_\ell}\) conditioned on not containing any \((v_*,j)\). This event happens with probability \(1-\frac{3}{\ell} \leq \frac{4}{7}\). Therefore \((\pi_2,\mu_{2,Z^{\dagger_\ell}})\) are \(\frac{4}{7}\)-smooth.
            \item Let us show that \((\nu_{0,y},\pi_{0,y})\) are \((A_Z,1)\)-smooth for a constant \(\alpha\). Fix \(((u,j),Z_i) \in A_Z\) and recall that \(u \ne v_*\).
            \[\Prob[\nu_{0,y}]{((u,j),Z_i)} = \Prob[Z^{\dagger_\ell}(0)]{(u,j)} \frac{\ell-1}{\ell},\]
            where the \(\frac{\ell-1}{\ell}\) is because conditioned on \((u,j)\), \(Z_i\) is chosen by choosing any \(j \ne i\) uniformly at random.

            For the other distribution, the expression is similar
            \[\Prob[\pi_{0,y}]{((u,j),Z_i)} = \cProb{Z^{\dagger_\ell}(0)}{(u,j)}{u \ne v_*} \frac{\ell-1}{\ell},\]
            The conditioning only increases the probability of \(u\) thus \((\nu_{0,y},\pi_{0,y})\) are \(1\)-smooth.
            \item Finally let us see that \((\pi_{1,y},\nu_{1,y})\) are \(\alpha\)-smooth for a constant \(\alpha\). For every edge \((\set{(u,i),(v,j)},Z_k)\), if this edge contains \((v_*,m)\) for some \(m\), its probability under \(\pi_{1,y}\) is \(0\). Otherwise, Its probability is
            \[\Prob[\pi_{1,y}]{(\set{(u,i),(v_j)},Z_k)} = \cProb{Z^{\dagger_\ell}(1)}{\set{(u,i),(v_j)}}{u,v \ne v_*} \frac{\ell-2}{\ell},\]
            Where the \(\frac{\ell-2}{\ell}\) is the probability of choosing \(Z_k\) which in this case is just the probability of choosing \(k \ne i,j\).
            
            The probability of this event under \(\nu_{1,y}\) is 
            \[\Prob[\nu_{1,y}]{(\set{(u,i),(v_j)},Z_k)} = \Prob[Z^{\dagger_\ell}(1)]{\set{(u,i),(v_j)}} \frac{\ell-2}{\ell},\]
            
            The event we are conditioning on occurs with probability \(1-\frac{2}{\ell} \geq \frac{5}{7}\). Therefore \((\pi_{1,y},\nu_{1,y})\) are \(\frac{5}{7}\)-smooth.
        \end{enumerate}
        \item The vertices of the label graph for \(i,j\) are all \((v,k)\) such that \(k \ne i,j\). We traverse from vertex to vertex via two steps in the colored swap walk from vertices to triangles of \(Z^{\dagger_\ell}[{[\ell] \setminus \set{i,j}}]\). The complex \(Z^{\dagger_\ell}[{[\ell] \setminus \set{i,j}}]\) is a \(\lambda\)-one sided and partite local spectral expander for constant \(\lambda\) so this two step walk is a constant expander by \pref{claim:partite-walk-is-a-const-spectral-expander}.
    \end{enumerate}
By \pref{thm:decomposition-to-coboundary-expanders} \(Z^{\dagger_\ell}\) is a \(\Omega(1)\)-coboundary expander.
\end{proof}

\begin{proof}[Proof of \pref{lem:trick-gen}]
Assume that for every \(s \in X(i)\), \(h^1(X_s) \geq \beta\) and prove that \(h^1(X) \geq \beta \exp(-Ci)\) (for a large enough but fixed constant \(C > 0\)). We prove by induction on \(i\). For \(i=-1\) the claim holds trivially and for \(i=0\) this holds by \pref{thm:cosystolic-expansion-from-link-coboundary-expansion}. So we assume the claim is true for $i$ and show it for faces \(i+1\). 

By the induction hypothesis, 
\begin{equation}\label{eq:trickle-down-induction-hypothesis}
    h^1(X) \geq exp(-C(i-1))\cdot \min_{t\in X(i-1)} h^1(X_t)
\end{equation}
By assumption, for every \(t \in X(i-1)\) and \(v \in X_t^J(0)\), \((X_t^J)_v = X_{\set{t\dunion \set{v}}}\), so all links of vertices in \(X_t\) are \(\beta\)-coboundary expanders. It follows that \(X_t\) is also a local spectral expander, that is, \(X\). It then holds by \pref{thm:cosystolic-expansion-from-link-coboundary-expansion} that
\[h^1(X_t) \geq \beta \cdot \frac{1-\lambda}{24} -e\lambda  \geq \beta \exp(-C).\]
Plugging this expression back in the minimum in \eqref{eq:trickle-down-induction-hypothesis} gives us \(h^1(X) \geq \beta \exp(-Ci)\).
\end{proof}

\subsection{Proofs for Claims in \pref{sec:proof-of-faces-complex-lower-bound}}

\begin{proof}[Proof of \pref{prop:prob-of-good-colors-tend-to-one}]
    We bound the probability of every item in \pref{def:good-colors} \emph{not occurring} by a quantity that goes to zero. 
    \begin{enumerate}
        \item The probability that \(i \in c_1, j \in c_2\) are the same is \(\frac{1}{n}\). There are at most \((m(d+1))^2\) such pairs, so by union bound the probability that there exists such \(i \in c_1 \cap c_2\) is at most \(\frac{(m(\d_1+1))^2}{n} \leq \frac{1}{\d_1+1}\).
        \item Fix \(\ell \in (\cup J) \cup \set{0,n}\). The probability that another \(\ell'\ne \ell\) doesn't violate this event, i.e. is sampled such that \(\Abs{\ell'-\ell} \geq \frac{n}{(m(\d_1+1))^3}\) is \((1-\frac{2}{(m(\d_1+1))^3})\). A priori, when we sample all other \(\ell' \in (\cup J)\), they are not independent of one another (since they must be distinct). However, even when we condition on sampling some \(k \leq m (\d_1+1)\) other \(\ell'\)'s into \((\cup J)\), the probability of sampling the next \(\ell''\) such that \(\Abs{\ell''-\ell} \geq \frac{n}{(m(\d_1+1))^3}\) is always at least \((1-\frac{2}{(m(\d_1+1))^3} \frac{n}{n-m (\d_1+1)}) \geq (1-\frac{3}{m(\d_1+1)^3})\) since \(n \geq \d_1^5\).

        To summarize, for every \(\ell\), the probability that there exists some \(\ell' \in \cup J\) such that \(\Abs{\ell'-\ell} < \frac{n}{(m(\d_1+1))^3}\) is upper bounded by 
        \(1- (1-\frac{3}{(m(\d_1+1))^3})^{m (\d_1+3)}.\)
        By Bernoulli's inequality this is at most
        \[\leq 1- (1-\frac{4}{(m(\d_1+1))^2}) = \frac{4}{(m(\d_1+1))^2}.\]
        Union bounding for all \(m(\d_1+1)+2\) elements in \((\cup J) \cup \set{0,n}\) gives us the desired bound.
        \item To bound the probabilities that \(3a,3b\) and \(3c\) are violated we will calculate a bound on the probability of a single subset \(J' \subseteq J\) violating these events that will be \(o(m^{-5})\). Then we union bound over all possible \(J' \subseteq J\) (a union bound over \(\leq m^5\) events). We get that all subsets \(J'\) will not violate the events. Let us indeed fix some \(J' \subseteq J\) and bound the probability that one of \(3a,3b\) or \(3c\) doesn't hold.
        \begin{enumerate}
            \item For \(\ell \in \cup \overline{J'} \cup \set{0}\), the probability that there is no \(\ell' \in \cup \overline{J'} \cup \set{n}\) so that \(1<\ell'-\ell \leq \frac{100n \log (\d_1+1)}{(\d_1+1) m}\) is at most
            \[(1-\frac{100 \log (\d_1+1)}{(\d_1+1)(m-5)})^{(\d_1+1)(m-5)} \leq e^{-50\log (\d_1+1)} = \frac{1}{(\d_1+1)^{50}}.\]
            The explanation to this probability is similar to the proof of item \(2\).
            \item Fix some \(c_1 \in J'\). Let us look more closely on the way we sample \(J\). We can sample \(J\) by first sampling \((\d_1+1) m\) colors into \(\cup J = \set{i_0<i_1<...<i_{(\d_1+1)m-1}}\). 
            Then we sample the \((\d_1+1)\)-colors that go \(c_1\). Afterwards, out of the rest of \(\cup J \setminus c_1\) we sample the other colors \(c_2,c_3,c_4,c_5\) that go into \(J'\). Finally we partition the rest of indexes sampled in to colors that go inside \(\overline{J'}\).
    
            We call sequence of maximal consecutive indexes in \(\cup J\), that were all sampled into \(c_1\) \emph{a segment of length \(r+1\)}. I.e. a set \(\set{i_j<i_{j+1}<...<i_{j+r}} \subseteq c_1\) such that \(i_{r-1},i_{j+r+1} \notin c_1\) (when \(j=1\) there is no \(i_{j-1}\) so this is vacuously true, and similarly when \(r+j=(\d_1+2)m-1\)). We note that the probability that there exists segment of length \(r = 6+\frac{\log (\d_1+1)}{\log m}\) is at most
            \begin{equation}\label{eq:segs}
            (\d_1+1) \cdot \prob{i_1,i_2,...,i_{r} \in c_1} \leq \frac{\d_1+2}{m^r} = o(m^{-5}).
            \end{equation}
            We explain. \((\d_1+1)\) is an upper bound on the number of possible starting points of a segment. Fixing a segment that starts at \(i_1\) (for example), for every \(1 \leq \ell \leq r\) it holds that
            \[\cProb{}{i_\ell \in c_1}{i_1,i_2,\dots,i_{\ell-1} \in c_1} = \frac{\d_1+1-\ell}{(\d_1+2) m - (\ell-1)} \leq \frac{1}{m}.\]
            This equality follows from the fact that the even we condition on is that \(c_1\) contains \(\ell-1\) other indexes. Thus there are \(\d_1-\ell\) remaining indexes to sample inside \(c_1\), out of a total of \((\d_1+1) m - (\ell-1)\). Multiplying all these probabilities given \(\frac{\d_1+2}{m^r}\) in \eqref{eq:segs} (note that in this calculation we didn't even take into account the fact that some indexes were supposed to be outside of \(c_1\), but this could only decrease the bound further).
            
            Next we bound the probability that more than \(\frac{10(\d_1+2)}{m}\) of the segments appear in a \(J'\)-crowded \(\overline{J'}\)-bin. If a segment \(\set{i_j<i_{j+1}<...<i_{j+r}}\) is in a \(J'\)-crowded \(\overline{J'}\)-bin, then either \(i_{j-1} \in \bigcup_{\ell=2}^5 c_\ell\) or \(i_{j+r+1} \in \bigcup_{\ell=2}^5 c_\ell\): if both \(i_{j-1}\) and \(i_{j+r+1}\) are not in \(\bigcup_{\ell=2}^5 c_\ell\) then this means that these indexes are both in \(\overline{J'}\), which means that the bin is lonely (a similar argument holds in the edge cases where \(j+r\) is the last index or \(j=1\)). Thus we bound the number of segments where this occurs. We will show how to bound the number of segments such that \(i_{j+r+1} \in \bigcup_{\ell=2}^5 c_\ell\), the other event will be bounded similarly. Let \(R\) be the random variable that counts the number of such ``bad'' segments. We will bound the probability that \(R\) is large using a Chernoff bound for negatively correlated random variables:
            
            Denote by \(R = \sum_{\ell} R_\ell\) where \(R_\ell\) is the indicator that the \(\ell\)-th segment has that \(i_{j+r+1} \in  \cup (J' \setminus \set{c_1})\). It is easy to see that \(\ex{R_\ell} = \frac{4(\d_1+2)}{(\d_1+2)(m-1)}\) and there are at most \((\d_1+2)\) segments so \(\ex{R} \leq \frac{4(\d_1+2)}{m-1} \leq \frac{5(\d_1+1)}{m}\). Moreover,
            \begin{align*}
                \ex{R_{\ell_1} \cdot R_{\ell_2} \cdot ... \cdot R_{\ell_p}}\leq & \\
                & \leq \frac{4(\d_1+2)}{(\d_1+2)(m-1)} \cdot \frac{4(\d_1+2)-1}{(\d_1+2)(m-1)-1} \cdot ... \cdot \frac{4(\d_1+2)-p+1}{(\d_1+2)(m-1)-p+1} \\
                & \leq  \left (\frac{4(\d_1+2)}{(\d_1+2)(m-1)} \right )^p  = \left ( \frac{4}{(m-1)} \right )^p
            \end{align*}  
            This is true because conditioned on \(R_{\ell_1},R_{\ell_2},\dots,R_{\ell_q} = 1\), the probability that \(R_{\ell_{q+1}}=1\) only decreases from the unconditional probability, we are only conditioning previous colors to be in \(\overline{J'}\). More formally, there are \(4(\d_1+2)-q\) colors we still need to choose into \(c_2,\dots,c_5\) (since the first \(q\) colors were already chosen), from a total of \((\d_1+2)(m-1)-q\) colors.
            Hence we can use Chernoff's bound for negatively correlated random variables and get that 
            \[\prob{R > \frac{10(\d_1+1)}{m}} \leq \exp \left (-\Omega \left (\frac{(\d_1+1)}{m} \right ) \right) = o(m^{-5}).\]
            
            If indeed every segment has at most \(\frac{\log (\d_1+1)}{\log m} + 6\) elements, then when all these events hold, the number of \(i \in c_1\) that are in a \(J'\)-crowded \(\overline{J'}\)-bin is at most \(\frac{10(\d_1+1)}{m} \left (\frac{\log (\d_1+1)}{\log m} + 6 \right ) \leq \frac{100(\d_1+1) \log (\d_1+1)}{m \log m}\). A union bound for all five \(c_1,c_2,c_3,c_4,c_5\) gives us the claim.
            \item We will show something stronger: that the number of colors in \(\cup J'\) in a single \(\overline{J'}\)-bin \(B\) is no larger than \(20 \frac{\log (\d_1+1)}{\log m}\). As before, when we choose \(J\) and \(J'\) we first choose \((\d_1+1)m\) colors \(\cup J = \set{i_0 <i_1 <\dots < i_{\d_1m-1}}\). Then we decide which \(5(\d_1+1)\) out of \(\cup J\) go into \(\cup J'\), and which go to \(\cup \overline{J'}\) (we do not care which index goes to which color inside this partition). In this context, let a segment of length \(r+1\) be a maximal sequence of consecutive indexes in \(\cup J\), that were chosen into \(\cup J'\).  That is some \(\set{i_j < i_{j+1} < \dots i_{j+r}} \subseteq \cup J'\) such that both \(i_{j-1},i_{j+r+1} \notin \cup J'\) (as before, if \(j=0,(\d_1+1)m-1\) the index doesn't exist and the statement is vacuously true). We need to show that there are no segments of length \(20 \frac{\log (\d_1+1)}{\log m}\) with high probability. As in the previous item, the probability that there exists segment of length \(r = 20\frac{\log (\d_1+1)}{\log m}\) in \(\cup J'\) is at most
            \[(5\d_1+5) \cdot \prob{i_1,i_2,...,i_{r} \in c_1} \leq \frac{(\d_1+1)}{m^r} = o(m^{-5}).\]
            The item follows.
        \end{enumerate}
    \end{enumerate}
\end{proof}

\begin{proof}[Proof of \pref{claim:subspaces-quotiented-in-the-middle}]
    Let us show the case where \(i_1 \leq \dim(v) \leq i_2\). The other two cases are similar; the same cone we define here will work in the other cases as well. Assume without loss of generality that the diameter of \(S_w^{i_0,i_1}\) is the minimal diameter. The automorphism group of \(S_w^I\) acts transitively on \(S_w^I (3)\), so by \pref{lem:group-and-cones} if we construct a cone \(C\) with diameter \(diam(C) = O(diam(S_w^{i_0,i_1}))\), this will imply that \(h^1(S_w^I) \geq \Omega(1/diam(S_w^{i_0,i_1}))\).

    Fix any base vertex \(v_0 \in S_w^{i_0,i_1}\). For \(u \in S_w^{i_0,i_1}(0)\) we set \(P_u\) to be an arbitrary shortest path between \(v_0\) and \(u\). We are guaranteed that \(|P_u| \leq diam(S_w^{i_0,i_1})\). For \(u \in S_w^{i_2,i_3}\) we note that \(v_0 u \in S_w^I(1)\) so we just set \(P_u = (v_0,u)\).

    Now for an edge \(ux \in S_w^I(1)\) let us give a contraction \(T_{ux}\). Let \(P_0 = P_x \circ (x,u) \circ P_u^{-1}\).
    \begin{enumerate}
        \item If \(ux \in S_w^{i_2,i_3}(1)\) then \(P_0 = (v_0,x,u,v_0)\) where \(uxv_0 \in S_w^I(2)\). In this case we just take \(P_1=(v_0,w,v_0)\) and \(T_{ux}=(P_0,P_1)\).
        \item If \(u \in S_w^{i_0,i_1}(0)\) and \(x \in S_w^{i_2,i_3}(0)\) then \(P_0=(v_0,x,u) \circ P_u^{-1}\). Denote by \(P_u^{-1}=(v_m=u,v_{m-1},\dots,v_0)\). We define a contraction \(T_{ux}=(P_0,P_1,\dots,P_m,)\). Here \(P_i= (v_0,x) \circ (v_{m-i},v_{m-i-1},\dots,v_0)\). Indeed we note that \(P_i \sim_1 P_{i+1}\) where the triangle \(t_{P_i,P_{i+1}}\) is \(x v_{m-i}v_{m-i-1}\) (it is a triangle because \(v_{m_i} v_{m-i-1}\) is an edge, and for every edge \(e \in S_w^{i_0,i_1}\) there is a triangle \(\set{x} \cup e \in S_w^I(2)\)).
        \item Finally, if \(ux \in S_w^{i_0,i_1}\) then \(P_0\) is a cycle contained in \(S_w^{i_0,i_1}\), so let us just denote it \(P_0=(v_0,v_1,v_2,\dots,v_m,v_0)\). Let \(y \in S_w^{i_2,i_3}(0)\) be an arbitrary vertex. For every edge \(e\) in this cycle, it holds that \(e \cup \set{y} \in S_w^I\). Therefore we can define \(T_{ux} = (P_0,P_1,\dots,P_{m+1})\) where 
        \[P_i=(v_0,y,v_1,y,v_2,y,\dots,y,v_i,v_{i+1},v_{i+2},\dots,v_m,v_0).\]
        Indeed the triangle \(y v_{i-1}v_i\) is the triangle that allows us to traverse from \(P_i\) to \(P_{i+1}\). Finally, we note that \(P_{m+1}=(v_0,y,v_1,y,v_2,y,\dots,v_m,y,v_0)\) and by a sequence of backtracking relations taking \((y,v_i,y)\) to \((y)\) it is equivalent to the trivial loop.
    \end{enumerate}
    In all cases we see that the number of \(P_i\)'s in the contractions depends only on the length of \(P_0\), which is always at most twice the diameter of \(S_w^{i_0,i_1}\). The claim follows.
\end{proof}

\begin{proof}[Proof of \pref{claim:well-spaced-subspaces}]
    The group of automorphisms acts transitively on the \(3\)-faces of \(\S_w^I\) so the claim will follow from existence of a cone of diameter \(\leq 10\) and \pref{lem:group-and-cones}. Fix any \(v_0 \in \S_w[i_0]\). For every \(u \in \S_w\) there exists a path \(P_u\) of length \(\leq 3\) such that every \(u' \in P_u \setminus \set{u}\) is in \(\S_w^{i_0,i_1}\). This path is constructed by taking some \(u_1 \subseteq u\) of color \(i_0\), and then taking two steps \((v,u_0,u_1,u)\) for \(u_0 \supseteq v + u_1\) (here we use the fact that \(i_1 \geq 2i_0\) to promise that there exists such a \(u_0\). Let \(\set{P_u}_{u \in \S_w}\) be a set of such paths.
    
    Let \(\set{u,u'}\) be an edge and let \(C=C_0 = P_u \circ (u,u') \circ P_u^{-1}\). Without loss of generality assume that \(u \subseteq u'\). Let us first do the case where \(u,u' \notin \S_w[i_3]\) then one can check that the sum of spaces in the cycle, \(\sum_{x \in C_0}x\) has dimension (or color) \(\leq i_2 + 2i_1\): For example, assume that \(col(u)=i_1, col(u')=i_2\). Then our cycle is composed of
    \[P_u=(v_0,u_0,u_1,u)\]
    \[P_{u'}^{-1} = (u',u_1',u_0',v_0).\]
    We observe that \(u_1 \subseteq u \subseteq u'\), that \(u_0,u_1 \subseteq u_0\) and that \(u_1' \subseteq u_0'\). Thus the sum of spaces is equal to \(u' + u_0 + u_0'\) which has dimension at most \(i_2 + 2i_1\). The other cases are similar.

    Thus there is a subspace \(x \in \S_w[i_3]\) that this sum, and therefore \(x \supseteq v_0,u_1,u_1,u,u',u_1',u_0',v_0\). In particular, every edge in the path is contained in a triangle together with \(x\). By a similar sequence of cycles as in \pref{claim:subspaces-quotiented-in-the-middle} we can move from \(C_0\) to
    \[C_m = (v_0,x,u_0,x,u_1,x,u,x,u',x,u_1',x,u_0',x,v_0)\]
    where \(m \leq 7\) is the number of edges. This cycle contracts to the trivial cycle by backtracking relations. Therefore the diameter of \(T_{uu'}\) is \(\leq 7\) in this case.

    We move to the case where (say) \(u' \in \S_w[i_3]\) and claim that we can use at most three additional triangle relations to reduce to the previous case.
    
    Indeed, start with the path \(C_0=P_u \circ (u,u') \circ P_u^{-1}\) as above. If \(u \notin \S_w[i_0]\) then the cycle is composed of 
     \[P_u=(v_0,u_0,u_1,u)\]
    \[P_{u'}^{-1} = (u',u_1',u_0',v_0).\]
    In this case \(u' \supseteq u,u_1,u_1'\). let \(y \in \S_w[i_1]\) be such that \(y \supseteq u_1'+u_1\) (these subspaces are of dimension \(i_0\) so by assumption that \(i_1 \geq 2i_0\) there exists such a \(y\)). Using the triangle \((u',u,u_1), (u',y,u_1)\) and \((u,y,u_1')\) we can contract \(C_0\) to 
    \[C_3 = (u,u_0,u_1,y,u_1',u_0',u_0).\]
    We contract \(C_3\) as in the previous case, by another \(6\) triangles. In this case \(T_{uu'}\) has diameter \(9\). The case where \(u \in \S_w[i_0]\) is similar (and in fact requires one less triangle).

    Thus the diameter of this cone is at most \(9\).
\end{proof}

\begin{proof}[Proof of \pref{claim:number-of-Ts-needed}]
Let \(S_1(x) = 2x-i_1 + 1, S_3(x) = \frac{21}{20}x - \frac{1}{20}i_3 + 1\) and let \(S(x) = \max \set{S_1(x),S_3(x)}\). It is obvious that for any \(n\), \(T^n(x) \leq S^n(x)\) hence it is enough to show that in \(m\) steps it holds that \(S^m(i_0)\leq 1\). The reason we do so, is that we wish to use possibly non-integral values in the analysis. It is easy to verify that when \(x \in [0,i_0]\) it holds that \(S_i(x)\) is monotone increasing and that \(S_i(x)\leq x\). Moreover, we claim that for any \(n\), it holds that \(S^{2n}(x) \leq \max \set{S_1^n(x), S_3^n(x)}\):
\begin{enumerate}
    \item There exists some \(\bar{v} \in \set{1,3}^{2n}\) such that 
    \begin{equation}\label{eq:S-two-n}
        S^{2n}(x) = S_{v_{2n}} \circ ... \circ S_{v_2} \circ S_{v_1}(x).
    \end{equation}
    \item If the number of \(3\)-s in the sequence \(\bar{v}\) is \(\geq n\) then \(S^{2n}(x)\leq S_3^n(x)\) since removing \(S_1\) from the expression in \eqref{eq:S-two-n}, only increases the function (since \(S_1\) of the inner expression is less or equal to the expression, and the outer function is monotone).
    \item A similar argument holds if the number of \(1\)-s is \(\geq n\), showing in this case that \(S^{2n}(x) \leq S_1^n(x)\).
\end{enumerate}

Thus it is enough to show that \(S_i^{m/2}(i_0) \leq 1\) for both \(i=1,3\). Take \(S_1\) for instance. Solving a recursion relation shows that \(S_1^n(i_0) = (i_0 - i_1-1)2^{n} + (i_1+1)\). Obviously \(S_1^n(i_0) \leq 1\) if \(2^n \geq \frac{i_1}{i_1-i_0+1}\), which obviously holds when \(\frac{m}{2}=n \geq \log(\frac{i_3}{i_1-i_0})\). A similar argument holds for \(S_3\).
\end{proof}

\begin{proof}[Proof of \pref{claim:annoying-graph-is-an-expander}]
Fix \(e_0=\set{v_1,v_2} \in X^I\) of colors (say) \(i_1,i_2\). For this proof we denote by \(c_1'=c_1 \setminus \set{i_1}, c_2' = c_2 \setminus \set{i_2}\) and for all \(j \geq 3\), \(c_j' = c_j\). Let us describe the vertices of its label graph \(G_{e_0}\):
\begin{enumerate}
    \item Faces \(w \in X_{v_1,v_2}[c_3']\dunion X_{v_1,v_2}[c_4'] \dunion \dots \dunion X_{v_1,v_2}[c_\ell']\).
    \item Faces \(w \in X_{v_2}[c_1]\) that contain \(v_1\). These correspond to faces \(\tilde{w} \in X_{v_1,v_2}[c_1']\).
    \item Faces \(w \in X_{v_1}[c_2]\) that contain \(v_2\). These correspond to faces \(\tilde{w} \in X_{v_1,v_2}[c_2']\).
\end{enumerate}
Let us denote the vertices by \(V\), and partition them into \(\ell\) parts such that for \(j \geq 3\) \(V_j = X_{v_1,v_2}[c_j']\) and for \(j=1,2\) \(V_j = \sett{\set{v_j} \dunion w' \in X[c_j]}{w' \in X_{v_1,v_2}[c_j']} \cong X_{v_1,v_2}[c_j']\).

Let us define a graph \(H\) to be a complete graph on \(\ell\) vertices \(V'=\set{1,2,\dots,\ell}\) with self loops. Its weight function is \(\mu_H(\set{t,j}) = \Prob[e \in C^{e_0}]{e \in E(V_t,V_j)}\). Here \(E(V_t,V_j)\) is the set of edges between \(V_t\) and \(V_j\), and for clarity we state that this formula is also for the case of self loops, i.e. \(j=t\).

By definition the map \(\rho:V \to V'\) given by \(\rho(v)=j\) for all \(v \in V_j\) is a homomorphism. Hence by \pref{claim:expansion-from-subexpanders} if we argue that the subgraphs \(E(V_t,V_j)\) and \(H\) are \(\lambda < 1\) spectral expanders then it follows that the label graph is also a spectral expander.

Let us begin with \(H\). Note that \(H\) is connected and non-bipartite since all edges between different \(j,t\) have positive weight. Moreover \(\Prob[e \in C^{e_0}]{e \in E(V_t,V_j)}\) only depends on \(j,t\) and \(\ell\) - not on \(v_1,v_2\) or the colors \(c_i\). Thus \(H\) is a single graph that is connected and non-bipartite, and thus\(\abs{\lambda}(H) < 1\).

Moving on to the subgraphs \(E(V_j,V_t)\). There are five cases to consider. \begin{enumerate}
    \item \(j=t \in \set{1,2}\).
    \item \(j=t \in \set{3,4,\dots, \ell}\).
    \item \(j \ne t\), \(j \in \set{1,2}\) and \(t \in \set{3,4,\dots,\ell}\) (or vice versa).
    \item \(j=1, t=2\) (or vice versa).
    \item \(j\ne t, j,t \in \set{3,4,\dots,\ell}\).
\end{enumerate}
In all these cases we can partition the graph \(E(V_j,V_t)\) to (a subset of) the following subgraphs with the intention of using \pref{claim:convex-combination-expanders}.
\begin{enumerate}
        \item \(H_{v_3 \in w}=(V_j,V_t,E_{j,t,v_3 \in w})\), the graph induced by edges \(\set{w,w'}\) such that \(w \in V_j, w' \in V_t\), that come from traversing through triangles \(t_1 = \set{v_1,v_2,v_3}_{w,w_2,w_3}\) and \(t_2 = \set{v_1,v_2,v_3}_{w',w_2,w_3}\) \emph{such that \(v_3 \in w\)}.
        \item \(H_{v_3 \in w'}=(V_j,V_t,E_{j,t,v_3 \in w'})\), the graph induced by edges \(\set{w,w'}\) such that \(w \in V_j, w' \in V_t\), that come from traversing through triangles \(t_1 = \set{v_1,v_2,v_3}_{w,w_2,w_3}\) and \(t_2 = \set{v_1,v_2,v_3}_{w',w_2,w_3}\) \emph{such that \(v_3 \in w'\)}.
        \item \(H_{v_3 \notin w,w'}=(V_j,V_t,E_{j,t,w \notin w,w'})\), the graph induced by edges such that \(v_3 \notin w,w'\).
\end{enumerate}
With respect to the cases above for \(j\) and \(t\), the first graph appears as a subgraph whenever \(j \notin \set{1,2}\), the second appears whenever \(t \notin \set{1,2}\). The third graph appears in all cases.

Let us verify that we can indeed use \pref{claim:convex-combination-expanders}. In all cases, we observe that the stationary distribution of all graphs is the same: if \(j \ne t\) it is just the distribution of the bipartite graph between \(X_{v_1,v_2}[c_j']\) and \(X_{v_1,v_2}[c_t']\) in \(X\), and if \(j=t\) it is the distribution over \(X_{v_1,v_2}[c_j']\).

We show that is a spectral expander and that for any \(j,t\), \(\Prob[e \in E(V_j,V_t)]{e \in H_{v_3 \notin w,w'}}\) is lower bounded by a constant (independent of the colors). Indeed, we note that in this case, traversing from \(w\) to \(w'\) is done by first sampling \(w_2,w_3,v_3 \in X_{v_1,v_2}\), and then independently sampling \(w \in X_u[c_j'], w' \in X_{u}[c_t']\) where \(u = \set{v_1,v_2,v_3} \cup s_2 \cup s_3\). This is a convex combination of two steps in colored swap walks, and there for a spectral expander by the assumption that \(X\) is a \(\frac{1}{2r^2}\) local spectral expander. We note that this is a convex combination of swap walks and not just one swap walk since, the colors of \(w_2\) and \(w_3\) are sometimes not determined, and in addition in some cases \(v_3 \in w_2 \dunion w_3\) and some times it is not, but all this doesn't matter for the analysis.

Let us argue that \(H_{v_3 \notin w,w'}\) appears with constant probability for any \(j,t\). For any \(j,t\), there is constant probability that \(w_2\) or \(w_3\) are in \(\bigcup_{m=3}^\ell V_m\) and that \(v_3 \in w_2 \dunion w_3\). This implies that \(E(V_j,V_t)\) is a constant expander by \pref{claim:convex-combination-expanders}.
\end{proof}

\begin{proof}[Proof of \pref{claim:similar-distributions-edges-connected-with-J}]
    Fix \(uv \in \Omega_1\).
    Note that we can describe all three probability distributions by first choosing the colors of the vertices \(u,v\) and then choosing the vertices themselves uniformly at random given those colors. Thus, if we denote by \(B_{c_1,c_2}\) the set of edges between colors \(c_1\) and \(c_2\), we have that \(\Prob[D_i]{uv} = \frac{1}{\Abs{B_{col(u),col(v)}}} \Prob[D_i]{B_{col(u),col(v)}}\). Hence it is enough to show for every two colors \(c_1 \in J,c_2 \notin J\) that 
    \begin{equation*}
        \Prob[D_0]{B_{c_1,c_2}} \leq 2 \Prob[D_2]{B_{c_1,c_2}} \ve \Prob[\d_1]{B_{c_1,c_2}} \leq 2 \Prob[D_0]{B_{c_1,c_2}}.
    \end{equation*}
    If \(c_1 \cap c_2 \ne \emptyset\) then all probabilities in question are \(0\). Otherwise, we calculate and bound. First,
    \[\Prob[D_0]{B_{c_1,c_2}} = \d_1^{-2} \left ( \binom{n-(\d_1+1)}{(\d_1+1)} - ((\d_1+1)^2-1) \right )^{-1}\]
    Since \((\d_1+1)^{-2}\) is the probability of choosing \(c_1 \in J\) and \(\binom{n-(\d_1+1)}{(\d_1+1)} - ((\d_1+1)^2-1)\) are the number of choices for \(c_2 \notin J\) that is disjoint from \(c_1\).
    Let us consider \(D_2\). Here for choosing the triangle \(uvw \sim D_2\) we choose mutually disjoint colors \(c_1,c',c_2\) such that \(c_1,c' \in J\) and \(c_2 \notin J\) (uniformly from all such triples). Thus,
    \[\Prob[D_2]{B_{c_1,c_2}} \geq \d_1^{-2}(1-\frac{(\d_1+1)}{(\d_1+1)^2-1}) \left (\binom{n-2(\d_1+1)}{(\d_1+1)}-((\d_1+1)^2-2) \right )^{-1}\]
    Since the probability of choosing \(c_1\) is \((\d_1+1)^{-2}\), the probability of choosing \(c' \in J\) such that \(c'\ne c_1\) and \(c' \cap c_2 = \emptyset\) is at least \((1-\frac{(\d_1+1)}{(\d_1+1)^2-1})\) (\(c_2\) intersects at most \((\d_1+1)\) colors in \(J\) since $|c_2|=\d_1+1$ and by mutual disjointness). Finally, \(\binom{n-2(\d_1+1)}{(\d_1+1)}-((\d_1+1)^2-2)\) is the number of colors outside of \(J\) that are disjoint from both \(c_1,c'\).
    It is easy to verify that provided that \(n \geq (\d_1+1)^3\) we have that \( \frac{\Prob[D_0]{B_{c_1,c_2}}}{\Prob[D_2]{B_{c_1,c_2}}} \leq 2\) for all \(\d_1\) large enough.

    Now let us consider \(D_1\). Here for choosing the triangle \(uvw \sim D_1\) we choose mutually disjoint colors \(c_1,c',c_2\) such that \(c_1 \in J\) and \(c',c_2 \notin J\) (uniformly from all such triples). Thus,
    \[\Prob[\d_1]{B_{c_1,c_2}} \leq \d_1^{-2} \left (\binom{n-2(\d_1+1)}{(\d_1+1)}-((\d_1+1)^2-1) \right )^{-1}\]
    since \((\d_1+1)^{-2}\) is the probability of choosing \(c_1\), the probability of choosing some \(c'\) that is disjoint from \(c_1\) is at most \(1\) (so we ignore it in the expression), and given \(c_1,c'\) the number of possible colors is \(\binom{n-2(\d_1+1)}{(\d_1+1)}-((\d_1+1)^2-1)\). 
    
    It is easy to see that in this case \( \frac{\Prob[D_1]{B_{c_1,c_2}}}{\Prob[D_0]{B_{c_1,c_2}}} \leq 2\) for all \(\d_1\) large enough, as well.
    \end{proof}

\begin{proof}[Proof of \pref{claim:similar-distributions-edges-disjoint-from-J}]
    Similar to the proof of \pref{claim:similar-distributions-edges-connected-with-J} the claim will follow if we show that for every two \(c_1,c_2 \notin J\) it holds that 
       \begin{equation*}
        \Prob[P_0]{B_{c_1,c_2}} \leq 2 \Prob[P_1]{B_{c_1,c_2}} 
    \end{equation*}
    First we note that
    \[\Prob[P_0]{B_{c_1,c_2}} \leq \left ( \binom{n}{\d_1+1}-(\d_1+1)^2\right )^{-1} \left ( \binom{n-(\d_1+1)}{(\d_1+1)}-(\d_1+1)^2\right )^{-1}.\]
    This is because choosing \(c_1\) happens with probability one over the number of colors not in \(J\), i.e. \(\binom{n}{(\d_1+1)}-(\d_1+1)^2\). Given \(c_1\) the probability of choosing \(c_2\) is at least \(\binom{n-(\d_1+1)}{(\d_1+1)}-(\d_1+1)^2\) since this is the minimal number of colors disjoint from \(c_1\) that are not in \(J\).
    On the other hand, choosing the colors of \(uvw \sim T_{nnJ}\) is choosing \(c',c_1,c_2\) such that they are all mutually disjoint and such that \(c' \in J, c_1,c_2 \notin J\). Thus,
    \[\Prob[P_1]{B_{c_1,c_2}} \geq (1-\frac{2(\d_1+1)}{(\d_1+1)^2-1}) \binom{n-3(\d_1+1)}{(\d_1+1)}^{-2}\]
    since choosing some \(c' \in J\) disjoint from \(c_1,c_2\) is with probability at least \((1-\frac{2(\d_1+1)}{(\d_1+1)^2-1})\), and given such a \(c' \in J\), choosing a pair \(c_1,c_2\) that are disjoint is at least \(\binom{n-3(\d_1+1)}{(\d_1+1)}^{-2}\).

    For large enough \(\d_1\) and any \(n \geq \d_1^3\), it holds that \(\Prob[P_0]{B_{c_1,c_2}} \leq 2 \Prob[P_1]{B_{c_1,c_2}}\) as required.
\end{proof}
\section{A Color Swapping Lemma} \label{app:color-swap}
In this subsection we prove our color swapping lemma.
\restatelemma{lem:base-reduction-using-decomposition-subspaces}

Let us denote by \(i' = i_j'\) for ease of notation. This lemma follows directly from a decomposition as in \pref{thm:decomposition-to-coboundary-expanders}. For a vertex \(v \in X^{I'}(0)\) we define a subcomplex \(Y_v\) in our decomposition. If \(col(v) \in I \cap I'\) this complex is induced by the vertices of \(\set{v} \cup X_v^{I \setminus col(v)}\). If \(col(v) = i'\) then \(Y_v = X_v^{I'}\). We note that in the first case, the coboundary expansion of \(Y_v\) is constant because all faces are connected to \(v\), so the minimal coboundary expansion of the subgraphs really depends on \(\min_{v \in {X[i']}} h^1(X_v^{I'})\); this is the reason for the appearence of this expression in the lemma.
We will see below that the agreement complex \(\mathcal{A}\) is actually (a blow-up of) \(X^{I'}\).

Let us define everything more formally.
\begin{definition}[\(GK(X,I,I')\) decomposition]
    Let \(GK(X,I,I') = (\mathcal{Y},\A,\nu,\pi)\) be as follows. 
\begin{enumerate}
    \item \(\mathcal{Y} = \sett{Y_v}{v \in X^{I'}(0)}\) are the following.
    \begin{enumerate}
        \item When \(col(v) = i'\) then \(Y_v = X_v^I\).
        \item When \(col(v)\ne i'\) then \(Y_v\) is the complex induced by \(\set{v} \cup X_v^{I \setminus \set{col(v)}}(0)\). We note that in this case \(Y_v \cong (X_v^{I \setminus \set{col(v)}})^*\), because every face in \(X_v^{I \setminus \set{col(v)}}\) is connected to \(v\).
    \end{enumerate}
    \item Let \(\A(0) = X^{I'}(0)\) (technically \(\A(0) = \set{Y_v}\) but it is shorter to identify \(Y_v \in \A(0)\) with \(v\)). An edge \(\set{v,v'}_{v''} \in \A(1)\) if and only if \(v\ne v'\) and \(\set{v,v',v''} \in X\)
    (either they form a triangle or it could be that \(v'' = v\) or \(v''=v'\)). 
    Similarly a triangle \(\set{v_1,v_2,v_3}_{v_4,v_5,v_6} \in \A(2)\) if and only if \(\set{v_1,v_2,v_3}, \set{v_4,v_5,v_6} \in X(2)\) 
    and \( \set{v_1,v_2,v_3,v_4,v_5,v_6} \in X\) 
    (note that some of these spaces may be equal).
    \item We define \(\pi\) as follows.  We choose \(\set{u_1,u_2,u_3}_{u_4,u_5,u_6}\) by:
    \begin{enumerate}
        \item We sample \(w=\set{v_1,v_2,v_3,v_4,v_5} \in X^{I \cup \set{i'}}(4)\).
        \item We sample \(w'=\set{u_1,u_2,u_3} \subseteq w\) uniformly given that \(col(w') \subseteq I'\). We randomly reorder \(w'\). 
        \item We sample \(w''=\set{u_4,u_5,u_6} \subseteq w\) uniformly given that \(col(w'') \subseteq I\), independent from the previous step. We randomly reorder \(w''\).
        \item We randomly output \(\set{u_1,u_2,u_3}_{u_4,u_5,u_6}\).
    \end{enumerate}
    \item We define \(\nu\) to be a marginal of \(\pi\). In particular, after sampling \(\set{u_1,u_2,u_3}_{u_4,u_5,u_6} \sim \pi\) we choose \(v \in \set{u_1,u_2,u_3}\) uniformly at random and output \((Y_{v},\set{u_4,u_5,u_6})\).
\end{enumerate}
\end{definition}
The following proposition describes the main method by which we apply this decomposition.

\begin{proof}[Proof of \pref{lem:base-reduction-using-decomposition-subspaces}]
    We use the decomposition \(GK(X,I,I')\) in order to apply \pref{thm:decomposition-to-coboundary-expanders}. Let us calculate the parameters. 
    \begin{itemize}
        \item \textbf{Coboundary expansion of every \(Y_v\).} Fix \(v\). If \(col(v) \in I\) then \(Y_v\) is a partite complex such that one part has only one vertex, i.e. \(v\) itself, since \(v\) is the only subspace compatible with itself of dimension \(col(v)\). It follows from \pref{claim:cone-is-coboundary-expander} that \(h^1(Y_v) = \Omega(1)\). The other case is when \(col(v)=i'_j\), in which case its coboundary expansion some \(h^1(X_v^I)\) since \(Y_v \cong X_v^I\). Thus the expansion of every sub complex is \(\min \set{\min_{v \in {X[i']}} h^1(X_v^I),\Omega(1)} = \Omega(\min_{v \in {X[i']}} h^1(X_v^I))\).
        \item \textbf{The expansion of \(\A\).} It is easy to see that \(\A\) is a blow-up of \(X^{I'}\) because ignoring the labels, \(\pi\) just samples a triangle from \(X^{I'}(2)\) uniformly at random. We define \(\gamma = h^1(X^{I'})\). If we show that for every (unlabeled) edge \(\set{v_1,v_2} \in X^{I'}(1)\), the label graph is a constant edge expander, it will follow from \pref{lem:hdx-blow-up} that \(h^1(\A) = \Omega(\gamma)\). Indeed fix any \(\set{v_1,v_2}\) of colors (say) \(\set{i_1,i_2}\). The labels of this edge are all \(u\) such that there is a flag that contains both $u$ and $v_1,v_2$ (including the case that $u$ equals one of $v_1,v_2$).
    
    Consider the following \(5\)-partite complex \(Z\) that corresponds to the distribution $\pi$ conditioned on $v_1,v_2$. The parts are \[Z[1]=\set{v_1}, \;Z[2]=\set{v_2}, \;Z[3]=X_{v_1,v_2}[i_3], \;Z[4]=X_{v_1,v_2}[i_4],\; Z[5]=X_{v_1,v_2}[i_5] \footnote{Meticulously speaking, there are five sides, one for every \(i \in I \cup I'\). The side corresponding to \(i_1,i_2\) are singletons and the rest correspond to \(X_{v_1,v_2}[i]\). We reannotated for convenience.}\]
    The top level faces are all \(\set{v_1,v_2,v_3,v_4,v_5} \in X^{I\cup I'}(4)\) (with the same distribution as of the first step in the description of \(\pi\)). Note that every labeled triangle \(\set{v_1,v_2,u_3}_{u_4,u_5,u_6}\) that contains \(v_1,v_2\) corresponds to a flag $\set{v_1,v_2,u_3,u_4,u_5} \in X^{I\cup I'}(4)\). This complex \(Z\) is a local spectral expander because \(X_{v_1,v_2}^{i_3,i_4,i_5}\) is a local spectral expander: fix a link \(Z_t\). The walk between every two sides in \(Z_t\) is either a complete bipartite graph (if at least one of the sides is \(Z_t[1]\) or \(Z_t[2]\)) or corresponds to a walk in two sides of \(X_{t \cup \set{v_1,v_2}}\).
    
    Given \(u\), we traverse to another \(u'\) by first selecting a triangle \(\set{v_1,v_2,v_3}_{u,u_2,u_3}\) (where \(v_1,v_2\) and \(u\) are fixed). Then we select a triangle \(\set{v_1,v_2,v_3}_{u',u_2,u_3}\) (where all but \(u'\) are fixed). This is the same as doing two steps in the vertex vs triangle swap walk in \(Z\), i.e. \(S_{0,2}(Z)\). By \pref{claim:partite-walk-is-a-const-spectral-expander} this is a constant spectral (and hence edge) expander.

    We conclude that \(h^1(\A) = \Omega(\gamma)\).
    \item \textbf{The expansion of the local graphs \(\A^u\).} First we note that in this case the set of elements that appear in the intersection of at least two \(Y_v\)'s is all \(X^I(0)\). For every \(u \in X^I(0)\) the vertices in \(\A^u\) are all \(X_u^{I'}(0)\), with edges sampled according to the distribution of \(X_u^{I'}(1)\). By assumption this graph has \(\zeta = \Theta(1)\) edge expansion. 
    \item \textbf{Smoothness of distributions.} We note that \(\pi_2=\nu_2=\mu_{2,S^I}\) so all three are \(1\)-smooth. Moreover, the distributions \(\nu_{1,y}=\pi_{1,y}\) and \(\nu_{0,y}=\pi_{0,y}\) so they are also \(1\)-smooth. All five items of smoothness follow.
    \end{itemize}
    Hence we apply \pref{thm:decomposition-to-coboundary-expanders} to \(GK(X,I,I')\) with parameters \(\alpha = 1, \Omega(\beta), \Omega(\gamma), \zeta = \Omega(1)\) which shows that \(h^1(X^I) \geq \Omega(\beta \gamma)\) as required.
\end{proof}
We note that the actual requirements from this theorem are that the actual requirements of expansion in this theorem are that the swap walks in the links have some constant edge expansion (even \(\frac{1}{3}-\varepsilon\) for any constant \(\varepsilon>0\)), but the constant behind the \(\Omega\) notation tends to \(0\) as the edge expansion tends to \(0\). This version suites our purposes.
\end{document}